\documentclass[11pt,a4paper]{article}
\usepackage[UKenglish]{babel}
\usepackage[titletoc,title]{appendix}
\usepackage{amsmath,amsfonts,amssymb,bm,amsthm}
\usepackage{bbm}
\usepackage{esint}
\usepackage[colorlinks,citecolor=blue,linkcolor=blue,pagebackref]{hyperref}
\usepackage{tikz}
\usetikzlibrary{shapes,arrows}
\usetikzlibrary{intersections,shapes.arrows}
\usepackage{caption}
\numberwithin{equation}{section}
\newtheorem{theorem}{Theorem}[section]
\newtheorem{lemma}[theorem]{Lemma}
\newtheorem{proposition}[theorem]{Proposition}
\newtheorem{corollary}[theorem]{Corollary}
\theoremstyle{definition}

\newtheorem{definition}[theorem]{Definition}
\newtheorem{assumption}[theorem]{Assumption}
\theoremstyle{remark}
\newtheorem{remark}[theorem]{Remark}
\newtheorem{notation}[theorem]{Notation}
\usepackage{nicefrac}

\newcommand{\dive}{\operatorname{div}}
\newcommand{\R}{\mathbb{R}}
\newcommand{\Z}{\mathbb{Z}}
\newcommand{\N}{\mathbb{N}}
\newcommand{\E}{\mathbb{E}}
\newcommand{\A}{\mathcal{A}}
\newcommand{\B}{\mathcal{B}}
\newcommand{\I}{\mathcal{I}}
\newcommand{\M}{\mathcal{M}}
\renewcommand{\O}{\mathcal{O}}
\renewcommand{\o}{\omega}
\renewcommand{\tilde}{\widetilde}

\renewcommand{\epsilon}{\varepsilon}
\newcommand{\dist}{\operatorname{dist}}

\renewcommand{\hat}{\widehat}

\renewcommand{\l}{\lambda}
\newcommand{\e}{\varepsilon}
\newcommand{\sym}{\operatorname{sym}}
\newcommand{\W}{\mathbf{W}}
\newcommand{\Lb}{\mathbf{L}}
\newcommand{\Hb}{\mathbf{H}}
\newcommand{\Cb}{\mathbf{C}}
\newcommand{\Nb}{\mathbf{N}}
\newcommand{\Gb}{\mathbf{G}}
\newcommand{\Vb}{\mathbf{V}}
\newcommand{\C}{\mathsf{C}}
\newcommand{\bu}{\mathbf{u}}
\newcommand{\bv}{\mathbf{v}}
\newcommand{\bb}{\mathbf{b}}
\newcommand{\bw}{\mathbf{w}}
\newcommand{\bg}{\mathbf{g}}
\newcommand{\bdf}{\mathbf{f}}
\newcommand{\be}{\mathbf{e}}
\newcommand{\bpsi}{\boldsymbol{\psi}}

\newcommand{\bvarphi}{\boldsymbol{\varphi}}
\newcommand{\bbeta}{\boldsymbol{\beta}}
\renewcommand{\hom}{{\rm hom}}
\newcommand\wtto{\stackrel{2}\rightharpoonup}

\usepackage{multicol}
\usepackage{longtable}

\usepackage[T1]{fontenc}
\usepackage[top=1in, bottom=1in, left=1in, right=1in]{geometry}
\usepackage{fancyhdr}
\usepackage{enumerate}
\usepackage{xcolor,cancel}
\usepackage{relsize}

\allowdisplaybreaks

\usepackage{pifont}
\renewcommand*{\backrefalt}[4]{%
\ifcase #1 %
No citations%
\or
\ding{43}~p.~#2%
\else
\ding{43}~pp.~#2%
\fi}

\begin{document}

\title{High-contrast random systems of PDEs: \\homogenisation and spectral theory}
\author{Matteo Capoferri\thanks{
\textcolor{black}{Dipartimento di Matematica,
Università degli Studi di Milano,
Via C.~Saldini 50,
20133 Milano, Italy}
\textit{and}
Maxwell Institute for Mathematical Sciences \&
Department of Mathematics,
Heriot-Watt University,
Edinburgh EH14 4AS,
UK;
\text{\textcolor{black}{matteo.capoferri@unimi.it}},
\url{https://mcapoferri.com}.
}
\and
Mikhail Cherdantsev\thanks{
School of Mathematics,
Cardiff University,
Senghennydd Road,
Cardiff CF24~4AG,
UK;
\text{CherdantsevM@cardiff.ac.uk}.
}
\and
Igor Vel\v{c}i\'c\thanks{
Faculty of Electrical Engineering and Computing, 
University of Zagreb, Unska 3, 
10000 Zagreb, Croatia;
\text{igor.velcic@fer.hr}.
}}


\date{}

\maketitle
\begin{abstract}
We develop a qualitative homogenisation and spectral theory for elliptic systems of partial differential equations in divergence form with highly contrasting (i.e., non uniformly elliptic) random coefficients. The focus of the paper is on the behaviour of the spectrum as the heterogeneity parameter tends to zero; in particular, we show that in general one doesn't have Hausdorff convergence of spectra. The theoretical analysis is complemented by several explicit examples, showcasing the wider range of applications and physical effects of systems with random coefficients, when compared with systems with periodic coefficients or with scalar operators (both random and periodic).

\


\

{\bf 2020 MSC classes: }
primary 
74S25,  
74A40;  	
secondary 
35B27,  
74Q15,  
35B40,  
60H15,  
35P05. 
\end{abstract}

\tableofcontents

\allowdisplaybreaks

\section{Introduction}
\label{Introduction}

Homogenisation theory addresses the problem of determining effective properties of heterogeneous materials (\emph{composites}).  Over time,  this has become a well-established branch of mathematics,  and a large number of classical monographs are now available. Yet, due to breadth and relevance of the subject, mathematical and physical understanding of the properties of composites remains at the forefront of modern research in science and industry.

A very interesting class of problems --- often referred to as \emph{high-contrast problems} --- is obtained by considering near-resonant inclusions dispersed in a base medium (often referred to as ``matrix''), in such a way that the macroscopic properties of the material are accounted for by resonances occurring at the microscopic scale.  Mathematically,  an important family of such problems is modelled by second order elliptic operators of the form 
\vspace{-2mm}
\begin{equation}
\label{operator D epsilon general}
\mathcal{D}^\epsilon=-\nabla \cdot a_\epsilon \nabla
\vspace{-2mm}
\end{equation}
defined on bounded or unbounded domains of $\R^d$, where the coefficients $a_\epsilon\in [L^\infty(\R^d)]^{d\times d}$, describing the physical properties of a material with scale of heterogeneity $\epsilon$, are of order $\epsilon^2$ in the inclusions and of order $1$ in the surrounding base medium.  This critical \emph{double-porosity} type scaling results in ``micro-resonances'',  which ultimately gives rise to numerous novel effects that are not present in the classical (uniformly elliptic) setting.
One such effect, and possibly the most important, is the band-gap structure of the spectrum, which makes high-contrast problems particularly relevant from the point of view of applications.

One of the mathematical challenges of working with high-contrast media is that the family of operators \eqref{operator D epsilon general} is not uniformly elliptic as $\epsilon\to0$,  which, in turn, results in a loss of compactness in the limit.  In pioneering work,  Allaire \cite{allaire} showed, in the {\it periodic} setting, that the limiting $\mathcal{D}^\mathrm{hom}$, analogue of the homogenised operator, possesses a \emph{two-scale} nature, that is,  it accounts for both macroscopic and microscopic properties, coupled in a particular way.  Zhikov \cite{zhikov2000,zhikov2004} later defined and analysed the two-scale limit operator, and established spectral convergence.  By decoupling the macroscopic and microscopic scales in the two-scale spectral problem $\mathcal{D}^\mathrm{hom}u=\lambda u$,  one obtains a spectral problem for the macroscopic component $\mathcal{D}^\mathrm{hom}_{\mathrm{mac}}u_1(x)=\beta(\lambda)u_1(x)$ with highly nonlinear dependence on the spectral parameter $\lambda$ expressed via  \emph{Zhikov's $\beta$-function}.   The latter, being defined explicitly in terms of the resonant frequencies of the inclusions,  characterises the macroscopic behaviour of the composite by virtue of capturing length-scale interactions.  In particular, the spectrum of the limit operator is described by
\vspace{-2mm}
\begin{equation*}
\sigma(\mathcal{D}^\mathrm{hom})=\{\lambda:\ \beta(\lambda)\ge0\} \cup \{\text{eigenfrequencies of the inclusions}\},
\vspace{-2mm}
\end{equation*}
and, as inferred by the explicit representation of $\beta(\l)$, has infinitely many gaps.

The literature in the high-contrast periodic setting is very rich.  The seminal results by Zhikov have been extended in numerous directions,  encompassing both scalar operators and systems of PDEs,  rigorously demonstrating new non-trivial properties, such as, for instance,  memory effects and spatial nonlocality \cite{CSZ,avila08}.  Note that even in the absence of high-contrast in the constituents, by carefully choosing the frequency of a Bloch wave type solution in relation to the wavenumber along the fibres one can \emph{artificially} obtain a high-contrast operator on the cross-section of the material \cite{CKS14}. Amongst recent advancements, Cherednichenko, Ershova and Kiselev \cite{kirill2} established a correspondence between homogenisation limits of micro-resonant periodic composites and a class of time-dispersive media.  Their interesting and new techniques, based on norm-resolvent estimates \textcolor{black}{(see also \cite{CCarma} for earlier work in this regard in the high-contrast setting)},  offer a recipe for the \textcolor{black}{design} of media with prescribed dispersive properties from periodic composites whose individual components are non-dispersive,  opening the way to the possibility of interpreting high-contrast media as frequency-converting devices. These scalar results were generalised to the case of systems in \cite{iosip}. Very recently, significant advancement in the subject was made by Cooper, Kamotski and Smyshlyaev: in \cite{CKS23} they developed a generic approach to study abstract family of asymptotically degenerating variational problems including, while going far beyond, \textcolor{black}{periodic} high-contrast highly oscillatory PDEs. Their asymptotic analysis yields uniform operator-type error estimate, and their novel approach provides approximations of the spectra of the associated spectral problems in terms of the spectrum of a certain ``bivariate'' operator. The latter arises as an abstract generalisation of the two-scale limiting operators traditionally \textcolor{black}{emerging} in the high-contrast literature.

Since fabricating perfectly periodic composites may be challenging from a practical point of view, it is natural to ask whether the properties of high-contrast materials desirable in applications (e.g., the band-gap structure of the spectrum) persist,  in a quantitatively controllable manner, when the geometry of the inclusions, the coefficients of \textcolor{black}{the} operator, or both involve an element of randomness.  Stochastic homogenisation,  qualitatively set out in \cite{kozlov,papanicolau,kunnemann},  is nowadays a very lively field of research, with a big community 
 concerned mainly with uniformly elliptic (non-high contrast) operators.  A most challenging and delicate aspect of the stochastic theory is to obtain growth estimates in the appropriate norms for the \emph{corrector}, a quantity that appears in the first-order term of the asymptotic expansion ansatz for the homogenisation problem. Comprehensive \emph{quantitative} results were first obtained by Gloria, Neukamm and Otto in the early 2010's in a series of ground-breaking papers \cite{gloria1,gloria2, gloria3} in which,  building upon earlier results by Yurinskii \cite{yurinskii} and Naddaf--Spencer \cite{naddaf}, they proved that the corrector is essentially of bounded growth. Similar results were obtained using a different (variational)  argument by Armstrong, Smart, Kuusi and Mourrat, see e.g. \cite{ArmSmart}, \cite{armstrong} and references in the latter. These developments  prompted an intensive research activity, which resulted in a number of new elegant ideas whose potential has not yet been fully explored.

For \emph{high-contrast} operators with random coefficients the picture is completely different, in that much less is known and even the \emph{qualitative} theory is still under development. Indeed, moving from high-contrast periodic to high-contrast stochastic PDEs involves a step change in both conceptual and technical difficulty. Stochastic two-scale resolvent convergence for high-contrast scalar operators on bounded domains and Hausdorff convergence of spectra were established by Cherdantsev, Cherednichenko and Vel\v{c}i\'c in \cite{CCV1}, building upon earlier results by Zhikov and Piatnitski \cite{ZP}.  While stochastic two-scale convergence works equally well in unbounded domains, the same authors showed \cite{CCV2} that in the whole space $\R^d$ the lack of compactness combined with the presence of infinitely many areas with ``atypical'' distributions of inclusions allows one to construct quasimodes for the random operator $\mathcal{D}^\epsilon$ ``escaping to infinity'', namely, weakly converging to zero, as $\varepsilon\to0$. The latter entails that generically 
the spectrum $\sigma(\mathcal{D}^\mathrm{hom})$ of the two-scale limiting operator, characterised by a stochastic version of $\beta(\lambda)$,  is a proper subset of the limiting spectrum $\lim_{\epsilon\to0}\sigma(\mathcal{D}^\epsilon)$, in turn characterised by $\beta_\infty(\lambda)$, a ``local'' modification of the stochastic $\beta(\lambda)$. In fact, it can be shown that the limiting spectrum has infinitely many gaps for a range of physically meaningful examples \cite[Section 4.6]{CCV2}. Also in the scalar setting, the authors of the current paper studied defect modes for operators with random coefficients \cite{decay}. Finally, 
we should mention the paper \cite{APZ21}, in which the authors study the large time behaviour of a Markov process associated with a symmetric diffusion in a high-contrast random environment and characterise the limit semigroup and the limit process under the diffusive scaling. Although there are certain similarities with \cite{CCV2}, the setting of \cite{APZ21} is somewhat less general and the focus \textcolor{black}{is} quite different.

\

The properties of high-contrast \emph{systems} with random coefficients are, to date, completely unexplored. Already in the periodic setting, high-contrast systems exhibit additional novel effects, compared with scalar operators.
For example, in the macroscopic component of the limiting spectral problem for systems the $\beta$-function is replaced by a $\bbeta$-matrix, which can be thought of as an ``effective density''.  It has been shown that for high-contrast periodic systems the $\bbeta$-matrix is frequency dependent,  and may be anisotropic as well as non-positive definite \cite{bouchitte,avila08,pastukova}, and even wavenumber dependent \cite{smyshlyaev_materials}.  What determines the structure of the spectrum is the signature of the $\bbeta$-matrix, which no longer has just a binary nature (positive/negative), thus resulting in a richer (and more complex) theory. Indeed, on top of ``strong'' bands (\textcolor{black}{maximal number of propagation modes} permitted) and gaps (propagation completely forbidden), one can have the intermediate situation of ``weak'' bands. This means that the number of macroscopic propagating modes may vary with frequency \cite{avila08} and, in the case of interconnected inclusions, with direction as well \cite{smyshlyaev_materials}.

Due to their multidimensional nature, systems also allow for the intermediate situation of \emph{partial degeneracy} (or \emph{partial high contrast}), in which the coefficients in the inclusions pick up an additional term of order $1$ which is non-degenerate \textcolor{black}{for certain components} in certain directions. Examples of this are, for instance, linear elasticity with inclusions hard in compression and soft in shear or Maxwell's equations with high contrast in the electric permittivity. The introduction of partial high contrast leads to a constrained microscopic kinematics, which needs to be accounted for in the homogenisation process \cite{cooper1,maxwell}.  
 Typically, one needs additional assumptions on the coefficients to ensure the well-posedness of the corrector equation in the space of admissible microscopic two-scale limit fields. A general theory for high-contrast partially degenerate \textcolor{black}{systems of} PDEs in the periodic setting was developed by Kamotski and Smyshlyaev in \cite{KS2018}.
 
\

Given the above rich picture in the periodic setting alongside the recent developments \cite{CCV1,CCV2,decay} for operators with random coefficients, it is very natural and timely to ask what new phenomena are brought about by allowing high-contrast \emph{systems} to have \emph{random} coefficients. Due to the additional freedom given by the random setting, one could examine scenarios precluded in the periodic or scalar settings, such as having a mixture of fully degenerate and partially degenerate inclusions or allowing the direction of partial degeneracy to vary randomly from inclusion to inclusion, to name just a couple of examples. Moreover, it may be possible to devise random media with spectral properties highly desirable in applications, such as eigenvalues embedded into the weak spectral bands.

In the current paper we carry out the first step of this programme, setting out a general homogenisation and spectral theory for high-contrast systems with random coefficients, thus laying firm theoretical foundations for a more specialised future analysis of the range of problems outlined above. In particular, random systems with partial degeneracy will be the subject of a separate paper.

\subsection*{Structure of the paper}
\addcontentsline{toc}{subsection}{Structure of the paper}

Our paper is structured as follows.

\

In Section~\ref{Statement of the problem} we introduce our geometric and probabilistic settings, and outline the problem we intend to study. In the end of Section~\ref{Statement of the problem} we provide a list of symbols and notation recurring throughout the paper, with precise pointers to where each of them is defined, to facilitate the reading.

In Section~\ref{Main results} we summarise the main results of our paper, with forward references to the key theorems and propositions, and elaborate on the novelty and differences with respect to the existing literature.

Section~\ref{The homogenised operator and its spectrum} is concerned with the examination of the two-scale limiting operator $\A^\mathrm{hom}$. After defining it, we study its spectrum and the relation thereof with the spectrum of $\A^\epsilon$.

In Section~\ref{Spectral convergence} we study the limiting spectrum $\lim_{\epsilon\to 0}\sigma(\A^\epsilon)$. In Subsection~\ref{Upper bound for the limiting spectrum} we prove an \textcolor{black}{outer} bound in terms of a certain set $\mathcal{G}$, which is shown in Subsection~\ref{Lower bound on the limiting spectrum} to be also an \textcolor{black}{inner} bound, under additional assumptions \textcolor{black}{on the finiteness of the range of dependence}.

Section~\ref{Examples} is devoted to several explicit examples, illustrating concretely the role of the various quantities appearing in Sections~\ref{The homogenised operator and its spectrum} and~\ref{Spectral convergence}, and showcasing the potential of the stochastic setting for high-contrast systems, compared to the periodic case.

The paper is complemented by one appendix, Appendix~\ref{Appendix auxiliary materials}, containing three sub-appendices with auxiliary technical material.

\section{Statement of the problem}
\label{Statement of the problem}

\subsection{Probabilistic and geometric setting}
\label{Probabilistic and geometric setting}

We work in Euclidean space $\mathbb{R}^d$, $d\ge 2$, equipped with the Lebesgue measure.  Given a measurable set $A\subset \R^d$, we denote by $|A|$ its Lebesgue measure, by $\overline{A}$ its closure, by $\mathbbm{1}_A$ its characteristic function. We adopt the notation
\begin{eqnarray}
\label{square x L}
   \square_x^L:=[-L/2, L/2]^d+x, \qquad x\in \R^d,  \ L>0. 
\end{eqnarray}
 We define a cut-off function $\eta\in C_0^{\infty}(\square_0^1)$ such that 
\begin{equation}\label{def eta}
0\leq \eta(x)\leq 1 \quad \mbox{ and } \quad \left. \eta\right|_{\square_0^{1/2}} =1,
\end{equation}
 and put
 \begin{equation}\label{def eta_L}
 	\eta_L(x):=\eta(x/L).
 \end{equation}
Clearly, we have 
\begin{equation}
	\label{bound grad eta}
	|\nabla \eta_L(x)|\le \frac{C}{L}, \quad \forall x\in \R^d.
\end{equation}

Function spaces with values in $\mathbb{R}^d$ will be denoted by bold letters. For example,
\begin{eqnarray}
\label{vector valued function spaces}
    \mathbf{L}^2(\R^d):=[L^2(\R^d)]^d, \qquad \mathbf{W}^{1,p}(\R^d):=[W^{1,p}(\R^d)]^d, \qquad \mathbf{C}_0^\infty(\R^d):=[C_0^\infty(\R^d)]^d\,.
\end{eqnarray}

Here $W^{k,p}(\R^d)$ is the usual Sobolev space of functions $p$-integrable together with \textcolor{black}{all} their first $k$ partial derivatives. Similarly, we shall denote by bold letters vector-valued functions. In the notation for norms, we will often write $\|\cdot\|_{\Lb^p}:=\|\cdot\|_{\Lb^p(\R^d)}$ and $\|\cdot\|_{\W^{1,p}}:=\|\cdot\|_{\W^{1,p}(\R^d)}$. Furthermore, we will write $\mathbf{H}^k(\R^d):=\W^{k,2}(\R^d)$.




 \

Our probability space $(\Omega, \mathcal{F}, P)$ is defined as follows.

\

We define $\Omega$ to be the set of all possible collections of randomly distributed {\it inclusions} in $\R^d$, that is, elements $\omega\in \Omega$ are subsets of $\R^d$ satisfying appropriate geometric conditions. Namely, we essentially require that individual inclusions are sufficiently regular, of comparable size and that they do not come too close to one another. These assumptions, rigorously formalised below as Assumptions~\ref{main assumption} and~\ref{main assumption 2}, mainly serve two purposes: on the one hand, they prevent the formation of clusters of inclusions (\emph{percolation}); on the other hand, they guarantee the validity of Korn--Sobolev extension theorems with uniform constants.

\color{black}

\begin{assumption}
\label{main assumption}
For all $\omega\in \Omega$
the set $\R^d \setminus \omega$ is connected, and $\omega$ can be written as a disjoint union
\begin{equation*}
\label{realisation as sum of inclusions}
\omega=\bigsqcup_{k\in \N} \omega^k
\end{equation*}
of sets $\omega^k$ (the inclusions) satisfying the following properties.
\begin{enumerate}
\item[(a)]
For every $k\in \N$ the set $\omega^k$ is open and connected.

\item[(b)]
$
\operatorname{diam} \omega^k <\frac12 
$.

\item[(c)]
There exists $\tau>0$, uniform in $\omega$, such that $\operatorname{dist}(\omega^k,\omega^j)>\tau$ for all $j\neq k$.
\end{enumerate}
\end{assumption}

%
%
%
%
%

%

%
\begin{assumption}
\label{main assumption 2}
There exist an integer $N\in \mathbb{N}$, pairs $\{(P_1^j, P_2^j), \ j=1,\dots, N\}$ of (open) Lipschitz domains in $\R^d$, and universal positive constants $c$, $C$ and $\tilde{C}$ with the following properties.
\begin{enumerate}[(i)]
\item
For all $j\in\{1,\dots, N\}$ we have
$\overline{P_1^j}\subset P_2^j$.

  \item For every $\omega \in \Omega$ and $k\in\mathbb{N}$, there exist 
\begin{enumerate}
\item     
$j=j(k)\in\{1,\dots, N\}$,
\item 
a bounded open neighbourhood $\widetilde{\B_\omega^k}\supset \omega^k$ and
\item 
a $C^2$ diffeomorphism $\varphi_k: P_{2}^j\to  \widetilde{\B_\omega^k}$ such that
\[
\varphi_k(\overline{P_1^j})=\overline{\omega^k},
\]
and
\[
c |x-y|\le |\varphi_k(x)-\varphi_k(y)|\le C |x-y|,  \qquad \|\varphi_{k}\|_{C^2(P_2^j)}\le \tilde{C}\,.
\]
\end{enumerate}        
\end{enumerate}
\end{assumption}

%

\begin{remark}
\label{remark on sufficiently regular}
Assumptions~\ref{main assumption} and~\ref{main assumption 2} warrant a number of remarks.
\begin{enumerate}[(i)]
\item
Assumption~\ref{main assumption} requires inclusions to be approximately of same size and at a ``safe'' uniform distance from one another. The restriction on the size of the inclusions is merely conventional and, at the same time,  unimportant: one can recover our setting from more general boundedness assumptions on the inclusions by elementary scaling arguments.

\item
Assumption \ref{main assumption 2} requires that all inclusions are $C^2$ diffeomorphic to a finite number of shapes, together with a small neighbourhood thereof, where the constants associated with these diffeomorphisms are uniformly bounded. Observe that Assumption \ref{main assumption 2} automatically implies that the inclusions are Lipschitz.


\item 
The papers \cite{CCV2, decay} feature an analogue of Assumption~\ref{main assumption} with Lipschitzness (see item (ii) above) replaced by the stronger assumption of minimal smoothness \cite[Chapter~VI,  Section~3.3]{stein}. Indeed, minimal smoothness was sufficient there to ensure the uniformity of the constants in the relevant inequalities needed in the proofs. In the current paper, Lipschitzness is enough for Theorem~\ref{theorem extension} for each individual inclusion.
However, since we will be dealing with symmetric gradients --- and, hence, Korn-type inequalities --- in order to have uniform (in $k$) constants of the extension operator stronger conditions (such as, e.g., Assumption~\ref{main assumption 2}) are needed to close the argument.

\item
Let us emphasise that our regularity assumptions on the inclusions are \emph{not} optimal. We need the regularity properties from Assumption~\ref{main assumption 2} in exactly two places: for the uniformity of Korn's constant given by Proposition~\ref{corollary uniform C}, and in the proof of Theorem~\ref{theorem higher regularity corrector} on the higher regularity of corrector, which, in turn, is used in the proof of Theorem~\ref{theorem G sub limiting spectrum} --- see~\eqref{proof part2 equation 10}, 
\eqref{proof part2 equation 29} and~\eqref{proof part2 equation 36}.

As far as the former is concerned, our assumptions are sufficient (although not necessary) to ensure the uniformity of the Korn constant on the domain. Note that Assumption~\ref{main assumption 2} is a considerably stronger than minimal smoothness assumption adopted for scalar operators in \cite{CCV2,decay} and, hence, takes us quite far from optimality.  We should like to point out, however, that the study of the dependence of Korn's inequality on the domain and its uniformity with respect to certain classes of domains is an active area of current research (see, e.g., the relatively recent result in \cite{griso} for star-shaped domains); in this perspective, pursuing optimal conditions for the uniformity of the Korn constant in our setting goes beyond the scope of the current paper.

As for the latter, in Theorem~\ref{theorem higher regularity corrector} one needs a bit more than just a uniform Korn constant. Indeed, by examining the argument in the  proof of Theorem~\ref{theorem higher regularity corrector} one realises that a uniform \emph{local} Korn inequality is required. Here is where the $C^2$ equivalence of the inclusions to a \emph{finite} number of shapes plays a role.
\end{enumerate}
\end{remark}

\color{black}

We define $\mathcal{F}$ to be the $\sigma$-algebra on $\Omega$ generated by the mappings $\pi_q:\Omega \to\{0,1\}$, where $q \in \mathbb{Q}^d$ and 
\begin{equation*}
\label{map pi}
\pi_q(\omega):=\mathbbm{1}_{\omega} (q). 
\end{equation*}
That is, $\mathcal{F}$ is the smallest $\sigma$-algebra that makes the mappings $\pi_q$, $q \in \mathbb{Q}^d$, measurable\footnote{Under Assumption~\ref{main assumption}, it is easy to see that this implies that the mappings $\pi_x:\Omega \to\{0,1\}$ for $x \in \mathbb{R}^d$ are also measurable. We postulate it only for $q \in \mathbb{Q}^d$ to ensure that $\mathcal{F}$ is countably generated, which in turn implies that the spaces $L^p(\Omega)$ are separable for $1\leq p< \infty$.}.
There is a natural group action of $\R^d$ on $\Omega$, and hence on $\mathcal{F}$, given by translations. Namely, for every $y\in \R^d$ the translation map $T_y:\Omega\to \Omega$ acts on $\omega\in \Omega$ as
\begin{equation}
\label{T_y}
\omega\mapsto T_y\omega=\{z-y \ | \ z\in \omega\}\subset \R^d.
\end{equation}
\textcolor{black}{Observe that $\Omega$ is, by definition, invariant with respect to all translations in $\R^d$, hence the above map is well defined.}

We equip $(\Omega,\mathcal{F})$ with a probability measure $P$ assumed to be invariant under translations, i.e., we assume that $P(T_y F)=P(F)$ for every $F\in \mathcal{F}$ and $y\in \R^d$, where $T_y F:=\bigcup_{\omega\in F}T_y \omega$.
It is not hard to see that $(T_y)_{y \in \R^d}$  satisfies the following properties:
	\begin{enumerate}[(a)]
	\item $T_{y_1} \circ T_{y_2}=T_{y_1+y_2}\ $ for all $y_1,y_2 \in \R^d$, where $\circ$ stands for composition;
	\item the map $\mathcal{T}: \R^d \times\Omega \to \Omega,\ $ $(y, \omega)\to T_y \omega$ is measurable with respect to the standard $\sigma$-algebra on the product space induced by $\mathcal{F}$ and the Borel $\sigma$-algebra on $\R^d$.
\end{enumerate} 

Finally, we assume the translation group action $(T_y)_{y\in \R^d}$ to be ergodic, i.e. if an element $F\in\mathcal{F}$ satisfies
\[
P((T_yF\cup F)\setminus (T_yF \cap F))=0 \quad \text{for all}\quad y\in\R^d,
\]
then $P(F)\in \{0,1\}$. This will allow us to use the classical Ergodic Theorem, which we recall as Theorem~\ref{ergodic theorem} in Appendix~\ref{The ergodic theorem}, for the reader's convenience.

\textcolor{black}{Given $\overline{\mathbf{f}}:\Omega \to \R^d$,} we adopt the standard notation
\begin{eqnarray}
\label{expectation}
\E[\overline{\mathbf{f}}]:=\int_\Omega \overline{\mathbf{f}}(\omega)\, dP(\omega)
\end{eqnarray}
and we denote by $\Lb^p(\Omega)=[L^p(\Omega)]^d$ the spaces of vector-valued $p$-integrable functions in $(\Omega, \mathcal{F},P)$. Since $\mathcal{F}$ is, clearly, countably generated, $\Lb^p(\Omega)$, $1\le p<\infty$, is separable. We denote by $\mathbf{f}(y,\omega):=\overline{\mathbf{f}}(T_y\omega)$ the \emph{realisation} or \emph{stationary extension} \textcolor{black}{of $\overline{\mathbf{f}}$}. Note that if $\overline{\mathbf{f}}\in \Lb^p(\Omega)$, then $\mathbf{f}\in L^p_\mathrm{loc}(\R^d;\Lb^p(\Omega))$ \cite[Chapter~7]{ZKO}. In the current paper we are mostly concerned with the case $p=2$.

\begin{notation}
Throughout our paper, unless otherwise stated, we will denote with an overline functions on $\Omega$ and we will remove the overline to denote the corresponding realisation (stationary extension). We will reserve the letter $y$ for the (stationary) extension variable. 
So, for example:
\begin{eqnarray}
\label{stationary extension}
    \overline{\mathbf{f}}=\overline{\mathbf{f}}(\omega), \quad \mathbf{f}=\mathbf{f}(y,\omega):=\overline{\mathbf{f}}(T_y\omega), \quad \E[\mathbf{f}]:=\E[\overline{\mathbf{f}}].
\end{eqnarray}
Functions on $\Omega$ may additionally depend on other variables, not necessarily in a stationary manner. For example, if $\overline{\mathbf{f}}(x,\omega):\R^d \times \Omega \to \R^d$, then $\mathbf{f}(x,y,\omega):=\overline{\mathbf{f}}(x,T_y\omega)$ and $\E[\mathbf{f}]=\E[\overline{\mathbf{f}}(x,\cdot)]$ (note that the dependence on the non-stationary variable $x$ remains upon taking the expectation).
\end{notation}

Let $\{e_j\}_{j=1}^d$ be the standard basis of $\R^d$. 
We define the Sobolev spaces $\Hb^s(\Omega)$, $s\in \N$, as
\begin{equation}
\label{definition Ws2}
\Hb^s(\Omega):=\left\{\overline{\bdf}\in \Lb^2(\Omega)\ | \ \bdf\in H^s_\mathrm{loc}(\R^d;\Lb^2(\Omega))\right\}.
\end{equation}

\color{black}
We should like to emphasise that, given $\overline{\bdf}\in \Hb^s(\Omega)$, for every multi-index $\boldsymbol{\alpha}=(\alpha_1,\ldots,\alpha_d)\in \N_0^d$, $\sum_{l=1}^d\alpha_l\le s$, the quantity $\partial^{\boldsymbol{\alpha}}\bdf$ is the stationary extension of a vector field in $\Lb^2(\Omega)$. The quantity $\overline{\partial^{\boldsymbol{\alpha}}\bdf}$ is to be understood as the random variable whose stationary extension is
$\partial^{\boldsymbol{\alpha}}\bdf$.
  Here $\partial^{\boldsymbol{\alpha}}:=\partial^{\alpha_1}_{y_1}\cdots \partial^{\alpha_d}_{y_d}$.

  In view of the above one defines the norm on \eqref{definition Ws2} as 
  \begin{equation*}
  	\left\|\overline{\bdf}\right\|^2_{\Hb^s(\Omega)} : = \sum_{|\boldsymbol{\alpha}|\leq s} \left\|\overline{\partial^{\boldsymbol{\alpha}} \bdf}\right\|^2_{\Lb^2(\Omega)}.
  \end{equation*}
\color{black}

Furthermore, we define 
\begin{equation*}
\label{definition Winfty2}
\Hb^\infty(\Omega):=\bigcap_{s\in \N}\Hb^s(\Omega)
\end{equation*}
and
\begin{equation}
\label{definition Cinfty}
\Cb^\infty(\Omega):=\left\{\overline{\bdf}\in \Hb^\infty(\Omega)\ | \ \overline{\partial^{\boldsymbol{\alpha}}\bdf}\in \Lb^\infty(\Omega) \text{ for every multi-index } \boldsymbol{\alpha}=(\alpha_1,\ldots,\alpha_d)\in \N_0^d \right\}.
\end{equation}

\begin{remark}
Observe that the definition of probabilistic Sobolev spaces retraces the classical one.
We refrain from introducing the stationary differential calculus on $\Omega$ more formally, as it will not be needed in this paper. We refer the interested reader to \cite[Appendices~A.2 and~A.3]{DG} for further details.
\end{remark}

Finally, we define 
\begin{equation*}
\label{definition L2zero}
\Lb^{2}_0(\Omega):=\left\{\overline{\bdf}\in \Lb^{2}(\Omega)\ | \ \left.\bdf(\cdot, \omega)\right|_{\R^d\setminus \omega}=0 \ \text{for \textcolor{black}{a.e.}} \ \omega \in \Omega\right\},
\end{equation*}
\begin{equation*}
\label{definition Ws2zero}
\Hb^s_0(\Omega):=\left\{\overline{\bdf}\in \Hb^s(\Omega)\ | \ \left.\bdf(\cdot, \omega)\right|_{\R^d\setminus \omega}=0 \ \text{for \textcolor{black}{a.e.}} \ \omega \in \Omega\right\},
\end{equation*}
and
\begin{equation*}
\label{definition Cinfinityzero}
\Cb^\infty_0(\Omega):=\left\{\overline{\bdf}\in \Cb^\infty(\Omega)\ | \ \left.\bdf(\cdot, \omega)\right|_{\R^d\setminus \omega}=0 \ \text{for \textcolor{black}{a.e.}} \ \omega \in \Omega\right\}.
\end{equation*}
The latter are spaces of functions in $\Lb^2(\Omega)$, $\Hb^s(\Omega)$ and $\Cb^\infty(\Omega)$, respectively, whose realisations vanish identically outside of the inclusions.

The above function spaces enjoy the following properties:
(i) $\Hb^s(\Omega)$ is a separable Hilbert space;
(ii) $\Hb^\infty(\Omega)$ is dense in $\Lb^2(\Omega)$;
(iii) $\Cb^\infty(\Omega)$ is dense in $\Lb^p(\Omega)$, $1\le p<\infty$;
(iv) $\Cb^\infty(\Omega)$ is dense in $\Hb^s(\Omega)$. Furthermore, we have at our disposal the following result, whose proof may be found in \cite{CCV1,CCV2}.

\begin{theorem}
\label{density theorem}
Under Assumptions~\ref{main assumption} \textcolor{black}{and \ref{main assumption 2}}, $\Cb_0^\infty(\Omega)$ is dense in $\Lb_0^2(\Omega)$ and in $\Hb^1_0(\Omega)$ with respect to $\|\cdot\|_{\Lb^2(\Omega)}$ and $\|\cdot\|_{\Hb^1(\Omega)}$, respectively.
\end{theorem}

The underlying motivation to much of our geometric assumptions is the fact that they are conducive to good ``extension properties''. \textcolor{black}{In what follows, $\sym \nabla \mathbf{u}:=\frac{\nabla \mathbf{u}+(\nabla \mathbf{u})^T}{2}$ denotes the symmetrised gradient of the vector-function $\mathbf{u}$.}

\begin{theorem}[Extension theorem {\cite[Lemma~4.1]{OSY}}]
\label{theorem extension}
Let $\C_0$ be a constant fourth-order tensor satisfying the symmetry and ellipticity conditions \eqref{eq:Csymmetries} and \eqref{ellipticity} below.
Under Assumption~\ref{main assumption},
for every $p\ge1$ there exists a bounded linear extension operator $E_k:\mathbf{W}^{1,p}(\B_\omega^k\setminus \omega^k) \to \mathbf{W}^{1,p}(\B_\omega^k)$, \textcolor{black}{with $\overline{\omega^k}\subset
\B_\omega^k
\subset
\overline{\B_\omega^k}\subset \tilde{\B_\omega^k}$ and $\B_\omega^k$ open and Lipschitz}, such that for every $\bu\in \mathbf{W}^{1,p}(\B_\omega^k\setminus \omega^k)$ the extension $\tilde{\bu}=E_k \bu\in \mathbf{W}^{1,p}(\B_\omega^k)$ satisfies
\begin{equation}
\label{extension property 1}
 \tilde\bu=\bu \qquad \text{in}\qquad \B_\omega^k\setminus \omega^k,
\end{equation}
\begin{equation}
\label{extension property 2}
\operatorname{div}\C_0\nabla \tilde \bu=0 \qquad \text{in}\qquad \omega^k,
\end{equation}
\begin{equation}
\label{extension property 4}
\| \tilde \bu \|_{\mathbf{W}^{1,p}(\B_\omega^k)} \le C\,\left(
\|\bu\|_{\Lb^p(\B_\omega^k\setminus \omega^k)}
+
\|\sym \nabla \bu \|_{\Lb^p(\B_\omega^k\setminus \omega^k)}
\right),
\end{equation}
\begin{equation}
\label{extension property 6}
\| \sym \nabla \tilde \bu \|_{\Lb^p(\B_\omega^k)} \le C\,\|\sym \nabla \bu \|_{\Lb^p(\B_\omega^k\setminus \omega^k)},
\end{equation}
where the constants $C$ in \eqref{extension property 4}, \eqref{extension property 6} depend on $\omega$, $k$, $p$, and the ellipticity constant of $\C_0$.
\end{theorem}

In plain English, the above extension theorem tells us that we can extend $\Hb^1$ vector-functions into the inclusions in a $\C_0$-harmonic manner, whilst controlling gradient and symmetric gradient of the extended function by gradient and symmetric gradient of the original function.

\begin{proposition}
\label{proposition uniform C}
Under Assumption~\ref{main assumption 2}, the constants $C$ in \eqref{extension property 4}, \eqref{extension property 6} is independent of $\omega$ and $k$ almost surely. \textcolor{black}{Furthermore, $\mathcal{B}_\omega^k$ can be chosen in such a way that $\operatorname{dist}(\mathcal{B}_\omega^k,\mathcal{B}_\omega^j)>\tau'$ for all $j\neq k$ and uniform (in $\omega$) $\tau'>0$.}
\end{proposition}
\label{corollary uniform C}

\begin{proof}
The claim is obtained arguing by contradiction and retracing the steps of the proof \cite[Lemma~1]{velcic} with account of \cite[Remark~6]{velcic}. \color{black}See also~\cite{HMT} for the treatment of transformations of Lipschitz domains under sufficiently smooth mappings.
\textcolor{black}{The second part of the statement follows at once from Assumption~\ref{main assumption 2}.}
\end{proof}

\begin{remark}
The constants $C$ in \eqref{extension property 4}, \eqref{extension property 6} depends on $p$. However, in our paper we will only use Theorem~\ref{theorem extension} for $p=2$ and for one $p>2$. Therefore, for all practical purposes we can suppress the dependence of $C$ on $p$, and treat the constant as if it were uniform with respect to $p$.
\end{remark}

\subsection{Our mathematical model}

Let $\C_j\in \R^{d^4}$, $j=0,1$,  be (constant) fourth-order tensors satisfying the following properties:
\begin{enumerate}[(a)]
\item
\emph{Symmetry}: 
\begin{equation}
\label{eq:Csymmetries}
(\C_j)_{\alpha\beta\mu\nu}=(\C_j)_{\mu\nu\alpha\beta}, \qquad (\C_j)_{\alpha\beta\mu\nu}=(\C_j)_{\alpha\beta\nu\mu},
\end{equation}
for all $\alpha,\beta,\mu,\nu=1,2,\ldots, d$;


\item
\emph{Ellipticity}:  there exists $c>0$ such that
\begin{equation}
\label{ellipticity}
(\C_j)_{\alpha\beta\mu\nu} \xi_{\alpha\beta}\xi_{\mu\nu}\ge c\,  |\xi|^2, \qquad \forall\xi \in \R^{d^2}, \ \xi_{\alpha\beta}=\xi_{\beta\alpha}, \qquad j=0,1.
\end{equation}
Here and further on, we adopt Einstein's summation convention over repeated indices, unless otherwise stated.
\end{enumerate}

\color{black}
\begin{remark}
\label{remark linear elasticity}
Conditions \eqref{eq:Csymmetries} and \eqref{ellipticity} are traditionally interpreted as a model of linear elasticity in dimension $d$.
\end{remark}
\color{black}

For every $0<\varepsilon<1$, for each random set of inclusions $\omega$ we partition $\R^d$ into two regions, up to a set of measure zero: the \emph{inclusions}
\begin{equation}
\label{inclusions epsilon}
\I^\epsilon(\omega):=\epsilon \omega
\end{equation}
and the \emph{matrix}
\begin{equation}
\label{matrix epsilon}
\M^\epsilon(\omega):=\R^d \setminus \overline{\I^\epsilon(\omega)}.
\end{equation}

%

%

We define the high-contrast  tensor field of coefficients as
\begin{equation}
\label{coefficients}
\C^\e = \C^\e(x,\o):=\mathbbm{1}_{\R^d\setminus\e\omega}\,\C_1+\, \mathbbm{1}_{\e\omega}\,\epsilon^2 \C_0.
\end{equation}
Here the characteristic size of the microstructure is of order $\e$, and the coefficients equal to $\C_1$ in the stiff matrix and to $\epsilon^2 \,\C_0$ in the soft inclusions. This type of scaling is often referred to as `double-porosity'. 

\begin{definition}
\label{def: operator A epsilon}
We define $\A^\epsilon(\omega)$ to be the self-adjoint linear operator in $\Lb^2(\R^d)$ associated with the bilinear form\footnote{Here and further on
\[
\C^\epsilon \nabla \bu \,\cdot \nabla \mathbf{v}= (\C^\epsilon)_{\alpha\beta\mu\nu} \partial_\mu \bu_\nu \, \partial_\alpha \mathbf{v}_\beta.
\]
}
\begin{equation}\label{form of the operator}
\int_{\R^d} \C^\epsilon \nabla \bu \,\cdot \nabla \mathbf{v}=\int_{\R^d} \C^\epsilon (\sym \nabla \bu) \,\cdot (\sym \nabla \bv),\, \qquad \bu,\bv \in \Hb^1(\R^d).
\end{equation}
\end{definition}

The operator $\A^\epsilon$ is a matrix operator acting on vector fields.

\

The overall goal of the paper is to examine the spectral properties of the operator $\mathcal{A}^\epsilon$ for sufficiently small $\epsilon$, through the prism of (two-scale) stochastic homogenisation.

\begin{remark}
\label{remark generalising}
Our setting can be generalised in a straightforward manner to cover the following.
\begin{enumerate}[(i)]
\item
One can consider operators acting on sections of more general trivial vector bundles $[L^2(\R^d)]^{m}$,  $m\in \mathbb{N}$. When $m=1$, one recovers the scalar theory.

\item
The coefficients $\C_0$ and $\C_1$ can be allowed to depend on $\omega$,  provided the ellipticity constant in \eqref{ellipticity} is uniform in $\omega$, \textcolor{black}{and $\C_0$, $\C_1$ are uniformly bounded as functions of $\omega$.}

\item 
Throughout the paper, we set the material density \textcolor{black}{which would normally appear in a model of linear elasticity (see Remark~\ref{remark linear elasticity})} to be equal 1, for simplicity. Studying the problem at hand \textcolor{black}{with a more general} material density can be done analogously.
\end{enumerate}
Upgrading our results so as to encompass (i)--(iii) only involves minor adjustments to the proofs. However, we carry out our arguments in the slightly more restrictive setting set out above for the sake of clarity, to avoid inessential technical details.
\end{remark}

\subsection*{List of notation}
\addcontentsline{toc}{subsection}{List of notation}
\begin{longtable}{l l}
\hline
\\ [-1em]
\multicolumn{1}{c}{\textbf{Symbol}} & 
  \multicolumn{1}{c}{\textbf{Description}} \\ \\ [-1em]
 \hline \hline \\ [-1em]
$\overset{2}{\to}$ ($\wtto$) & Strong (resp.~weak) stochastic two-scale convergence --- Definition~\ref{definition stochastic two scale convergence}\\ \\ [-1em]
$\square_x^L$ & Hypercube of size $L$ centred at $x$, see \eqref{square x L}\\ \\ [-1em]
$\mathbbm{1}_A$ & Characteristic function of $A\subset \R^d$\\ \\ [-1em]
$1_\Omega$  & Random function \eqref{definition 1 Omega}\\ \\ [-1em]
$|A|$ & Lebesgue measure of $A\subset \R^d$ \\ \\ [-1em]
$\overline{A}$ & Closure of $A\subset \R^d$  \\ \\ [-1em]
$\A^\epsilon$ & Definition~\ref{def: operator A epsilon}  \\ \\ [-1em]
$\A^\mathrm{hom}$ & Limiting two-scale operator --- see~Definition~\ref{definition Ahom, A1, A0}  \\ \\ [-1em]
$\A_0$ & Microscopic part of $\A^\mathrm{hom}$ --- see~Definition~\ref{definition Ahom, A1, A0}  \\ \\ [-1em]
$\A_0^Q$ & Operator $\A_0$ in $\Lb^2(Q)$ with Dirichlet b.c. --- see~Definition~\ref{definition A^Q}  \\ \\ [-1em]
$\A_1$ & Macroscopic part of $\A^\mathrm{hom}$ --- see~Definition~\ref{definition Ahom, A1, A0}  \\ \\ [-1em]
$\overline{\mathbf{b}_\lambda}$ & Matrix function \eqref{definition of b}, \eqref{definition of b formula 2} \\ \\ [-1em]
$\overline{\mathbf{b}_\lambda^{(i)}}$ & $i$-th column of $\overline{\mathbf{b}_\lambda}$, see also~\ref{equation for bi} and~\eqref{3.18} \\ \\ [-1em]
$\mathcal{B}_\omega^k$ & Extension domain of $\omega^k$ --- see~Theorem~\ref{theorem extension} \\ \\ [-1em]
$\bbeta(\lambda)$ & Zhikov $\bbeta$-matrix \eqref{definition beta matrix} \\ \\ [-1em]
$\beta(\lambda)$ & Largest eigenvalue of $\bbeta(\lambda)$, see~\eqref{non-bold beta} \\ \\ [-1em]
$\beta_\infty(\lambda)$ & Deterministic function given by Definition~\ref{definition beta infinity} almost surely \\ \\ [-1em]
$\C_0$, $\C_1$ & Symmetric and elliptic fourth-order tensors, see \eqref{eq:Csymmetries}, \eqref{ellipticity} \\ \\ [-1em]
$\C^\epsilon$ & Tensor field of coefficients \eqref{coefficients} \\ \\ [-1em]
$\C^\mathrm{hom}$ & Homogenised tensor field of coefficients \eqref{definition Chom} \\ \\ [-1em]
$\mathbf{C}_0^\infty$, $\Hb^k$, $\mathbf{L}^2$, $\mathbf{W}^{1,p}$ & Bold letters denote functions spaces with values in $\R^d$, see~\eqref{vector valued function spaces} \\ \\ [-1em]
$\mathcal{D}(\A)$ & Domain of the opeartor $\A$ \\ \\ [-1em]
$\be_i$, $i\in\{1,\dots,d\}$ & Basis vectors in $\R^d$, $[\be_i]_j=\delta_{ij}$ \\ \\ [-1em]
$\E[\overline{\mathbf{f}}]$  & Expectation \eqref{expectation} \\ \\ [-1em]
$\eta$ and $\eta_L$ & Cut-off \eqref{def eta} and its scaled version \eqref{def eta_L}, respectively \\ \\ [-1em]
$f(y,\omega):=\overline{f}(T_y\omega)$ & Stationary extension \ref{stationary extension} of the random variable $\overline{f}$ \\ \\ [-1em]
$\mathcal{G}$ & Set \eqref{G}\\ \\ [-1em]
$\Hb$ & Function space \eqref{definition of the space H}\\ \\ [-1em]
$\Hb^1_{\mathrm{per}}(Q)$ & Elements of $\Hb^1$ that are periodic with period $Q$ \\  \\ [-1em]
$I$ & Identity matrix \\ \\ [-1em]
$\mathcal{I}^\epsilon$ & $\epsilon$-scaled inclusions \eqref{inclusions epsilon} \\ \\ [-1em]
$\mathcal{M}^\epsilon$ & $\epsilon$-scaled matrix \eqref{matrix epsilon} \\ \\ [-1em]
$\sigma(\A)$ & Spectrum of the operator $\A$ \\ \\ [-1em]
$T_y$ & Dynamical system \eqref{T_y} on $\Omega$ acting by translation \\ \\ [-1em]
$\Vb$ & Function space \eqref{definition of the space V}\\ \\ [-1em]
$\boldsymbol{\mathcal{V}}^2_\mathrm{pot}$ & Potential vector fields with zero mean \eqref{definiton of V_pot}\\ \\ [-1em]
$\omega$ & Collection of inclusions ($\omega\subset \R^d$, $\omega\in \Omega$) --- see~Assumption~\ref{main assumption} \\ \\ [-1em]
$\omega^k$ & Individual inclusion, connected component of $\omega$ --- see~Assumption~\ref{main assumption} \\ \\ [-1em]
$(\Omega, \mathcal{F}, P)$ & Probability space \\ \\ [-1em]
\hline
\end{longtable}

\section{Main results}
\label{Main results}

In this section we summarise, for the convenience of the reader, the main results of our paper.

\

Theorem~\ref{thm 3.3} establishes that the family of operators $\A^\epsilon$ converges \textcolor{black}{(that is, their resolvents converge in the sense of weak/strong stochastic two-scale convergence)}, as $\epsilon\to 0$, to a limiting operator $\A^\mathrm{hom}$ (see Definition~\ref{definition Ahom, A1, A0}). Our \textbf{first main result} is Theorem~\ref{propositon spectrum Ahom}, which characterises the spectrum of the limiting operator $\A^\mathrm{hom}$. The latter is comprised of two parts: the micro-resonances --- i.e., the spectrum of a microscopic limiting operator $\A_0$ --- and the nonnegative real numbers $\lambda$ for which the \emph{matrix} function $\bbeta(\lambda)$ \eqref{beta matrix} has at least one nonnegative eigenvalue. \textcolor{black}{Furthermore, as a consequence of the above convergence result, $\sigma(\mathcal{A}^\hom)$ is contained in the limiting spectrum $\lim_{\epsilon\to0}\sigma(\A^\epsilon)$ (Corollary~\ref{c 4.12}).}

\

Our \textbf{second main result}, Theorem~\ref{theorem limiting spectrum sub G}, provides an \textcolor{black}{outer} bound for the limiting spectrum $\lim_{\epsilon\to0}\sigma(\A^\epsilon)$ in terms of the set $\mathcal{G}$~\eqref{G}, in turn defined by a nonlocal \emph{scalar} function $\beta_\infty(\lambda)$, deterministic almost surely, which is sensitive to areas of space with atypical arrangements of inclusions.

\ 

As it turns out, under the additional assumption of finite-\textcolor{black}{range} \textcolor{black}{dependence}, $\mathcal{G}$ is also an \textcolor{black}{inner} bound for the limiting spectrum. This is the subject of our \textbf{third main result}, Theorem~\ref{theorem G sub limiting spectrum}, which, in conjunction with Theorem~\ref{theorem limiting spectrum sub G}, yields $\lim_{\epsilon\to0}\sigma(\A^\epsilon)=\mathcal{G}$. Since, in general, $\sigma(\A^\mathrm{hom})$ is a proper subset of $\mathcal{G}$, this implies one does not have Hausdorff convergence of spectra, so that the limiting operator $\A^\mathrm{hom}$ captures only partially information about the spectrum of $\A^\epsilon$ for arbitrary small but finite $\epsilon$. This is a distinctive feature of the stochastic setting, already observed in \cite{CCV2} in the scalar case, which persists for systems with random coefficients.

\

Lastly, our \textbf{fourth main result} is Proposition~\ref{proposition 1 example 1}, which, together with the other examples from Section~\ref{Examples}, showcases the wider range of applications and physical effects of stochastic systems, when compared with periodic ones, thus paving the way for the further analysis outlined in Section~1.

\

\color{black}
Let us emphasise that, although the strategy is broadly similar to recent works \cite{CCV1,CCV2} dealing with the scalar case, its technical implementation is not. For instance, in addition to the fact that one must work with a matrix --- as opposed to scalar --- ``macroscopic spectral parameter'' (the Zhikov $\beta$-matrix), it is not \emph{a priori} clear what the right quantity to characterise the set $\mathcal{G}$ for random systems would be, or how such quantity would be related to the stochastic $\beta$-matrix. As it turns our, the set $\mathcal{G}$ is described by a \emph{scalar} quantity even in the case of systems, which calls for a few new tricks to make the proofs work (e.g., in Theorem~\ref{theorem G sub limiting spectrum}). The differences between the scalar case and systems is also transparent from our set of examples. By means of these examples, we expose new spectral features, to be fully explored in subsequent works, which distinguish the setting of the present paper from either periodic high-contrast systems or stochastic high-contrast scalar problems. We refer the reader to Section~6 for further comments and insight.
\color{black}

\section{The limiting operator and its spectrum}
\label{The homogenised operator and its spectrum}

Consider the `two-scale' space
\begin{equation}
\label{definition of the space H}
\Hb:=\Lb^2(\R^d) +L^2(\R^d; \Lb^2_0(\Omega)),
\end{equation}
which is \textcolor{black}{a subspace of} $\Lb^2(\R^d; \Lb^2(\Omega))$ and \textcolor{black}{has as} dense subspace
\begin{equation}
\label{definition of the space V}
\Vb:=\Hb^1(\R^d) +L^2(\R^d; \Hb^1_0(\Omega))\,.
\end{equation}
The spaces $\Hb$ and $\Vb$ are comprised of vector functions of the form $\bu_1(x)+\overline{\bu}_0(x,\omega)$, where the stationary extension $\bu_0(x,y,\omega)$ of $\overline{\bu}_0(x,\omega)$ vanishes for $y$ outside of the set of inclusions $\omega$.

Let us recall that a tensor field $\mathbf{p}\in [\Lb^2_{\mathrm{loc}}(\R^d)]^d$ is said to be
\begin{itemize}
\itemsep.3em

\item
\emph{potential} if there exists $\boldsymbol{\varphi}\in \Hb^1_\mathrm{loc}(\R^d)$ such that $\mathbf{p}=\nabla \boldsymbol{\varphi}$;

\item
\emph{solenoidal} if 
\[
\int_{\R^d} \mathbf{p}\cdot \nabla \boldsymbol{\varphi} =0 \qquad \forall \boldsymbol{\varphi}  \in \Cb_0^\infty(\R^d).
\]
\end{itemize}
See, e.g., \cite{ZKO}, for further details.
By analogy, one defines a tensor field $\overline{ \boldsymbol{p}}\in [\Lb^2(\Omega)]^d$ to be potential (respectively solenoidal) if for a.e.~$\omega$ its stationary extension $y\mapsto \mathbf{p}(y,\omega)\in [\Lb^2_{\mathrm{loc}}(\R^d)]^d$ is such. The classical Weyl's decomposition theorem holds for tensor fields in $[\Lb^2(\Omega)]^d$ as well, see e.g. \cite[Chapter 12.7]{ZKO}, see also \cite[Theorem~2.14]{decay}.

Let $1_\Omega:\R^d\times \Omega \to \{0,1\}$ be the stationary random function defined in accordance with
\begin{equation}
\label{definition 1 Omega}
1_\Omega(y,\omega):=\mathbbm{1}_\omega(y)\,.
\end{equation}
Let $\C^\mathrm{hom}$ be the homogenised tensor of coefficients defined in accordance with
\begin{equation}
\label{definition Chom}
\C^\mathrm{hom}\xi \cdot \xi:=\inf_{\overline{\boldsymbol{\psi}}\in \boldsymbol{\mathcal{V}}^2_\mathrm{pot}(\Omega)} \E\left[  (1-1_\Omega)\,\C_1 (\boldsymbol{\psi}+\xi) \cdot (\boldsymbol{\psi}+\xi) \right], \quad \forall \xi \in \R^{d^2},
\end{equation}
where
\begin{equation}
\label{definiton of V_pot}
\boldsymbol{\mathcal{V}}^2_\mathrm{pot}(\Omega):=\{\overline{\bpsi}\in [\Lb^2(\Omega)]^d \ | \  \bpsi \text{ is potential and } \ \E[\overline{\bpsi}]=0\}\,.
\end{equation}
Consider the following bilinear forms:
\begin{equation}
	\label{3.6}
	\E\left[\C_0 \nabla_y \bu \cdot \nabla_y \bvarphi\right] \,,
	\qquad
	\overline{\bu}, \  \overline{\bvarphi}\in \Hb^1_0(\Omega),
\end{equation}
\begin{equation}
		\label{3.7}
	\int_{\R^d} \C^\mathrm{hom}\nabla \bu \cdot \nabla \bvarphi \,,
	\qquad
	{\bu}, \  {\bvarphi}\in \Hb^1(\R^d),
\end{equation}
and
\begin{equation}
	\label{quadratic form Ahom}
	\int_{\R^d} \C^\mathrm{hom}\nabla \bu_1 \cdot \nabla \bvarphi_1+\int_{\R^d} \E\left[\C_0 \nabla_y \bu_0 \cdot \nabla_y \bvarphi_0\right] \,,
	\qquad
	\bu_1+\overline{\bu_0}, \ \bvarphi_1+\overline{\bvarphi_0}\in \Vb\,.
\end{equation}
By the standard Korn's inequality  
and its direct analogue for the space  $\Hb^1_0(\Omega)$ we see that the above bilinear forms \textcolor{black}{(upon addition of the usual $L^2$ term in \eqref{3.7} and \eqref{quadratic form Ahom})} are coercive, hence, closed.

\begin{definition}[Limiting two-scale operators]
\label{definition Ahom, A1, A0}
We define the linear operators $\mathcal{A}_0$ in $\Lb^2_0(\Omega)$, $\mathcal{A}_1$ in $\Lb^2(\R^d)$, and $\mathcal{A}^\mathrm{hom}$  in $\Hb$ as the self-adjoint operators associated with the bilinear forms \eqref{3.6}, \eqref{3.7} and \eqref{quadratic form Ahom}, respectively. We will refer to the operators $\mathcal{A}_0$ and $\mathcal{A}_1$ as the micro- and macroscopic parts of the {\it limiting operator} $\mathcal{A}^\mathrm{hom}$.
\end{definition}

The operator  $\A_0$ can formally be identified with the operator $-\operatorname{div}_y \C_0 \nabla_y$ acting on the stationary extensions of the random variables from $\mathrm{dom}(\A_0)$. For later use, it is also convenient to introduce a ``restriction'' of this operator to a single inclusion $\o^k$ or, more generally, to a bounded domain in $\R^d$.
\begin{definition}
\label{definition A^Q}
Given a bounded domain $Q\subset \R^d$, we denote by $\A_0^{Q}$ the positive definite self-adjoint operator associated with the bilinear form
\begin{eqnarray*}
\label{definition operator A^Q equation}
\int_{Q} \C_0\,\nabla \bu \cdot \nabla \bv, \qquad \bu,\bv \in \Hb^1_0(Q)\,.
\end{eqnarray*}
\end{definition}
In other word, $\A_0^{Q}$ is nothing but the operator $-\operatorname{div} \C_0 \nabla$ in $\Lb^2(Q)$ with Dirichlet boundary conditions.

\

Let $\mathcal{C} \subset \Cb^\infty(\Omega)$ be a countable dense family of vector-functions in $\Lb^2(\Omega)$ (recall that the latter is separable)  and let $\Omega_t=\Omega_t(\mathcal{C})\subset \Omega$ be a set of probability one such that the claim of the Ergodic Theorem~\ref{ergodic theorem} holds for all $\omega \in \Omega_t$ and $\overline{\mathbf{f}}\in \mathcal{C}$. Elements of $\Omega_t$ are often referred to as \emph{typical}, hence the subscript "t".

\begin{definition}[Stochastic two-scale convergence \cite{ZP}]
\label{definition stochastic two scale convergence}
Let $\{\bu^\epsilon\}$ be a bounded sequence in $\Lb^2(\R^d)$. We say that $\{\bu^\epsilon\}$ weakly stochastically two-scale converges to $\overline{\bu}\in \Lb^2(\R^d\times \Omega)$ (for a given $\omega_0\in {\Omega_t}$) and write $\bu^\epsilon \overset{2}{\rightharpoonup} \overline{\bu}$ if
\begin{equation}
\label{definition stochastic two scale convergence equation 1}
\lim_{\epsilon\to 0}\int_{\R^d} \bu^\epsilon(x)\cdot \varphi(x)\mathbf{f}(x/\epsilon,\omega_0)\,dx=
\E\left[\int_{\R^d} \overline{\bu}\cdot \varphi\,\overline{\mathbf{f}} \right]
\quad \forall  \varphi\in C^\infty_0(\R^d),\,\overline{\mathbf{f}} \in {\mathcal{C}}.
\end{equation}
We say that $\{\bu^\epsilon\}$ strongly stochastically two-scale converges to $\overline{\bu}\in \Lb^2(\R^d\times \Omega)$ and write $\bu^\epsilon \overset{2}{\to} \overline{\bu}$ if it satisfies \eqref{definition stochastic two scale convergence equation 1} and 
\begin{equation*}
\label{definition stochastic two scale convergence equation 2}
\lim_{\epsilon\to0}\|\bu^\epsilon\|_{\Lb^2(\R^d)} =\|\overline{\bu}\|_{\Lb^2(\R^d\times \Omega)}.
\end{equation*}
\end{definition}
For the reader's convenience, basic properties of \textcolor{black}{stochastic} two-scale convergence are summarised in Appendix~\ref{appendix: two-scale}.

\begin{remark}
Observe that one can always add countably many elements to the family $\mathcal{C}$ in the above definition, provided one updates the choice of the set of full measure $\Omega_t$ accordingly. See also \cite{ZP, heida} for further details.
We will tacitly make use of this freedom throughout the paper. Indeed, one is always able to identify the limiting objects resulting from two-scale convergence that appear in our paper by testing the relevant weak identities with countable dense families of (vector-)functions with higher-than-$L^2$ --- typically, smooth --- regularity.
\end{remark}

\

The operator $\A^\mathrm{hom}$ is the stochastic two-scale limit of $\A^\epsilon$ as $\epsilon\to 0$ in the sense of the following theorem.

\begin{theorem}\label{thm 3.3}
Let $\lambda<0$ and let $\{\bdf^\epsilon\}\subset \Lb^2(\R^d)$, $\|\bdf^\epsilon\|_{\Lb^2}\le C$,  be a bounded sequence of square integrable functions.  Suppose $\bdf^\epsilon\wtto\bdf$ ($\bdf^\epsilon\overset{2}{\to}\bdf$). Then the solution $\bu^\epsilon$ of the resolvent problem
\begin{equation}
\label{resolvent problem for Aepsilon}
\A^\epsilon \bu^\epsilon-\lambda \bu^\epsilon=\bdf^\epsilon
\end{equation}
weakly (strongly) stochastically two-scale converges to some $\bu = \bu_1+\overline{\bu_0}\in V$,
$\bu^\epsilon \wtto \bu$ ($\bu^\epsilon \overset{2}{\to} \bu$),  almost surely. Furthermore, $\bu$ is the  solution of 
\begin{equation}
	\label{3.10}
	\A^\hom \bu-\lambda \bu=\textcolor{black}{\mathcal{P}_{\mathbf{H}}}\bdf\,,
\end{equation}
\textcolor{black}{
where $\mathcal{P}_{\mathbf{H}}$ is the orthogonal projection onto $\mathbf{H}$.}
\end{theorem}
\begin{proof}
The argument is standards and we provide only the general sketch for the sake of completeness.

Testing \eqref{resolvent problem for Aepsilon} with $\bu^\e$ one gets the following energy bound:
\begin{equation*}
	\|\sym \nabla \bu^\e\|_{\Lb^2(\M^\e)} + \e	\|\sym \nabla \bu^\e\|_{\Lb^2(\I^\e)} + 	\|\bu^\e\|_{\Lb^2(\R^d)} \leq C.
\end{equation*}
Applying the Extension Theorem~\ref{theorem extension} (modulo a rescaling argument) together with Korn's inequality,  we obtain the extension $\tilde \bu^\e$ to the whole of $\R^d$ of the restriction $\bu^\e\vert_{\M^\e}$ satisfying the estimate
\begin{equation*}
	\| \tilde \bu^\e\|_{\Hb^1(\R^d)}  \leq C.
\end{equation*}
Moreover, the difference $\bu^\e - \tilde \bu^\e \in \Hb^1_0(\I^\e)$ satisfies (via the two preceding bound, Korn's and the Poincar\'e inequalities)
\begin{equation*}
\e	\| \nabla (\bu^\e - \tilde \bu^\e )\|_{\Lb^2(\R^d)} + 	\|\bu^\e - \tilde \bu^\e\|_{\Lb^2(\R^d)} \leq C.
\end{equation*}
Then, by the properties of stochastic two-scale convergence (see Appendix~\ref{appendix: two-scale}), the following convergences (up to extracting a subsequence) take place,
\begin{equation*}
\begin{aligned}
	\bu^\e & \rightharpoonup && \bu_1 \mbox{ weakly in } \Lb^2(\R^d),   
 \\
\bu^\e - \tilde \bu^\e &\wtto&& \overline{\bu_0},
	\\
	 \nabla \tilde\bu^\e &\wtto&& \nabla\bu_1 + \overline{ \boldsymbol{p}},
  \\
  \e\nabla (\bu^\e - \tilde \bu^\e) &\wtto&& \overline{\nabla_y \bu_0},
\end{aligned}
\end{equation*}
where $ \overline{ \boldsymbol{p}} \in  \Lb^2(\R^d; \boldsymbol{\mathcal{V}}^2_\mathrm{pot}(\Omega))$. 

Testing the equation \eqref{resolvent problem for Aepsilon} with $\e a(x) \bvarphi(x/\e,\o)$, $a\in C^\infty_0(\R^d)$, $\overline{\bvarphi}\in \mathcal{C}$, and passing to the limit as $\e \to 0$, we see that $ \overline{ \boldsymbol{p}} $ is the corrector term, i.e. for every $x\in \R^d$ the random variable $\overline{ \boldsymbol{p}}(x,\cdot)$ is a minimiser of the problem \eqref{definition Chom} with $\xi = \nabla \bu_1(x)$. In particular, 
\begin{equation*}
 \E\left[\mathbbm{1}_{\R^d\setminus\omega}  \C_1 ( \nabla \bu_1 + \boldsymbol{p})  \right] =\C^\mathrm{hom} \nabla \bu_1.
\end{equation*}
Then passing to the limit in \eqref{resolvent problem for Aepsilon} with a test function $ \bvarphi_1(x)+a(x) \bvarphi(x/\e,\o)$, $ \bvarphi_1\in \Hb^1(\R^d)$, $a\in C^\infty_0(\R^d)$, $\overline{\bvarphi}\in \{\text{countable dense subset of }\Hb^1_0(\Omega)\}$, and using the density argument, one recovers the equation \eqref{3.10}.

Finally, the strong stochastic two-scale convergence of $\bu^\e$ under the assumption $\bdf^\epsilon\overset{2}{\to}\bdf$ follows from the weak stochastic two-scale convergence by a standard duality argument, see e.g. \cite{zhikov2000}.
\end{proof}

\

The remainder of this section will be devoted to the examination of the spectrum of $\mathcal{A}^\mathrm{hom}$. To this end,   
for $\lambda\not \in \sigma(\A_0)$,
let $\overline{\mathbf{b}^{(i)}_\lambda}\in \Hb^1_0(\Omega)$, $i=1,\ldots, d$, be the unique solution of 
\begin{equation}
\label{equation for bi}
\A_0 \overline{\mathbf{b}^{(i)}_\lambda} = \lambda \,  \overline{\mathbf{b}^{(i)}_\lambda}+\be_i, 
\end{equation}
where $\be_i\in \R^d$, $(\be_i)_j=\delta_{ij}$.

\begin{lemma}
\label{lemma b in physical space}
For every nonnegative $\lambda\textcolor{black}{\not \in \sigma(\A_0)}$, the stationary extension $\mathbf{b}^{(i)}_\lambda$ of $\overline{\mathbf{b}^{(i)}_\lambda}$ satisfies
\begin{equation}
	\label{3.18}
	-\dive_y\C_0\nabla_y\, \mathbf{b}^{(i)}_\lambda = \lambda \,  \mathbf{b}^{(i)}_\lambda+\be_i\qquad \text{in }\ \omega\ \textcolor{black}{\text{almost surely}}.
\end{equation}
\end{lemma}

\begin{proof}
The proof below relies on a classical mollification argument which can be found, e.g., in \cite[\S~7.2, p.~232 ss.]{ZKO}. Let $K\in C_0^\infty(\R^d)$ be a non-negative even function such that $\int_{\R^d} K(x)\,dx=1$, and, for $\delta>0$, set $K^\delta(x):=\delta^{-d}K(\delta^{-1}x)$. For a random variable $\overline{{f}}\in L^2(\Omega)$ (and, analogously, for vector/tensor-valued random variables) we define its mollification to be
\begin{equation*}
\overline{ {f}}^\delta(\omega):=\int_{\R^d} K^\delta(y)\,  {f}(y,\omega)\,dy\,.
\end{equation*}
It is easy to check that the mollification $\overline{ {f}}^\delta$ possesses the following properties: $\overline{ {f}}^\delta\in C^\infty(\Omega)$, $ {f}^\delta(y,\omega)=( {f}\ast K^\delta)(y,\omega)$,  $\partial_{y_i} {f}^\delta=   {f}\ast (\partial_{y_i} K^\delta)$, $\E[\overline{ {f}} \, \overline{ {g}}^\delta] = \E[\overline{ {f}}^\delta \, \overline{ {g}}]$, and $\E[\overline{ {f}}\,  \overline{ \partial^{\boldsymbol{\alpha}}{g}}^\delta] = (-1)^{|{\boldsymbol{\alpha}}|}\E[\overline{\partial^{\boldsymbol{\alpha}} {f}}^\delta \, \overline{ {g}}]$, where $*$ stands for convolution in the variable $y$ and $\overline{{g}}\in L^2(\Omega)$.

Let us define the measurable subset $\mathcal{O}\subset \Omega$ as
\begin{equation*}
\label{definition of O}
\mathcal{O}:=\{\omega \in \Omega \ | \ 0\in \omega\}
\end{equation*}
and let us introduce the random variable $\mathfrak{d}:\Omega \to\mathbb{R}$ defined as
\begin{equation*}
\mathfrak{d}(\omega):= \operatorname{dist}(0,\R^d\setminus \omega)\,,
\end{equation*}
where $0$ is the origin in $\R^d$. Observe that the map $\mathfrak{d}$ is measurable, cf., e.g., \cite[Lemma~B.3]{CCV2}. Furthermore, for $\delta>0$ let $\O^\delta$ be the (measurable, in view of the measurability of the translation group action) set
\begin{equation*}
\label{definition O delta}
\O^\delta:=\{\omega\in \O \ |\ \mathfrak{d}(\omega)>\delta\}\,.
\end{equation*}
Clearly, $\O^{c\delta}\subset \O^\delta$ for $c>1$.

Let $\overline{\bv}$ be an element in $\Lb^2(\Omega)$ vanishing outside $\mathcal{O}^{\delta}$. Then $\overline{\bv}^{\delta} \in \Cb^{\infty} (\Omega)$ and it vanishes outside $\mathcal{O}$, therefore it can be used as a test function in the weak formulation of \eqref{equation for bi} to obtain 
\begin{equation*}
\E 
\left[
\C_0 \nabla_y  \mathbf{b}_{\lambda}^{(i)} \cdot \nabla_y \mathbf{v}^\delta
\right]
=
\E
\left[
\lambda \,\overline{\mathbf{b}_{\lambda}^{(i)}} \cdot\overline{\mathbf{v}}^{\delta}+(\overline{\mathbf{v}}^{\delta})_i
\right].
\end{equation*}
Using the properties of the mollification listed above we obtain
\begin{equation*}
\E 
\left[
-\dive_y\left(\C_0 \nabla_y  [\mathbf{b}_{\lambda}^{(i)}]^\delta\right) \cdot \mathbf{v}
\right]
=
\E
\left[
\lambda \,\overline{\mathbf{b}_{\lambda}^{(i)}}^\delta \cdot\overline{\mathbf{v}}+\be_i^\delta\cdot \overline{\mathbf{v}}
\right],
\end{equation*}
which, since $\overline{\bv}$ is arbitrary, in turn implies
\begin{equation*}
-\overline{\dive_y\left(\C_0 \nabla_y  [\mathbf{b}_{\lambda}^{(i)}]^\delta\right)}=\lambda \,\overline{\mathbf{b}_{\lambda}^{(i)}}^\delta+\be_i^\delta \qquad \text{in}\quad \O^{\delta}\,.
\end{equation*}
Upon taking the stationary extension of all quantities, the above equation is valid without the bar in $\omega^\delta \subset \R^d$ almost surely, where $\omega^\delta:=\{x\in \omega\,|\, \dist(x,\R^d\setminus\omega) >\delta\}$.

Let us fix an element $\omega \in \Omega$ and choose $\mathbf{g}\in \Cb_0^\infty(\R^d)$ compactly supported in one connected component of $\omega$ (i.e.,~in a single inclusion). Then for sufficiently small $\delta$ we have
\begin{equation*}
\int_{\R^n} \C_0\nabla  \mathbf{b}_{\lambda}^{(i)} \cdot [\nabla \mathbf{g}]^{\delta}=\lambda  \int_{\R^n}\mathbf{b}_{\lambda}^{(i)}\cdot \mathbf{g}^{\delta} +\int_{\mathbf{R}^n} \mathbf{e}_i\cdot \mathbf{g}^\delta\,.
\end{equation*}
Passing to  the limit in the above equation as $\delta\to0$ and using the fact that $\bg$ is smooth, we arrive at
\begin{equation*}
\int_{\R^n} \C_0\nabla  \mathbf{b}_{\lambda}^{(i)} \cdot \nabla \mathbf{g}=\lambda  \int_{\R^n}\mathbf{b}_{\lambda}^{(i)}\cdot \mathbf{g} +\int_{\R^n} \mathbf{e}_i\cdot \mathbf{g}.
\end{equation*}
Since $\bg$ is arbitrarily and chosen from a dense subspace of $\Hb^1(\R^d)$, the latter implies \eqref{3.18}.
\end{proof}

Denote
\begin{equation}
\label{definition of b}
\overline{\mathbf{b}_\lambda}:=\left(\overline{\mathbf{b}^{(1)}_\lambda} \ | \ \overline{\mathbf{b}^{(2)}_\lambda}\ |\ \ldots \ | \ \overline{\mathbf{b}^{(d)}_\lambda}\right) \in [\Hb^1_0(\Omega)]^d.
\end{equation}
Clearly, $\overline{\mathbf{b}_\lambda}$ satisfies
\begin{equation}
\label{definition of b formula 2}
	\A_0   \overline{\mathbf{b}_\lambda}=\lambda\, \overline{\mathbf{b}_\lambda}+I\,.
\end{equation}

Note that, unlike in the periodic setting,  the random operator $\A_0$ does not have, in general,  discrete spectrum. For example, in the case \textcolor{black}{when} inclusions are allowed to ``continuously'' change size (e.g., being homothetic to a fixed shape), the spectrum of $\A_0$ is a union of intervals. The following statement asserts that $\sigma(\A_0)$ can be recovered from a single $\o$ almost surely.

\begin{theorem}
\label{thm:DspecUnionSpecIncl}
We have
\begin{equation*}
\sigma(\A_0)=\overline{\bigcup_{k\in \N} \sigma(\A_0^{\o^k})}
\end{equation*}
for almost every $\o\in \Omega$.
\end{theorem}
\begin{proof}
The claim is proven analogously to \cite[Theorem~5.6]{CCV1}, with account of \cite[Remark~3.12]{CCV2}.
\end{proof}

Using Theorem~\ref{thm:DspecUnionSpecIncl},  one can estimate $\mathbf{b}_\lambda$ on hypercubes.
\begin{lemma}
\label{lemma properties b lambda}
There exists a positive constant $C=C_{\lambda,d}$ uniform in $\epsilon$ such that
\begin{equation}
\label{lemma properties b lambda equation 1}
\|\mathbf{b}_\lambda(\cdot/\epsilon, \omega)\|_{[\Lb^2(\square_z^L)]^{d}}+\|\epsilon \,\nabla \mathbf{b}_\lambda(\cdot/\epsilon, \omega)\|_{[\Lb^2(\square_z^L)]^{d}}\le C\, L^{d/2}
\end{equation}
for all $z\in \R^d$ and $L>0$ almost surely.
\end{lemma}
\begin{proof}
In each inclusion $\o^k$ the extension $\mathbf{b}_\lambda^{(i)}(\cdot, \omega)$, $i=1, \ldots, d$, solves the problem (cf.~\eqref{3.18})
\begin{equation}
\label{proof lemma properties b lambda equation 1}
\A_0^{\o^k}\mathbf{b}^{(i)}_\lambda  - \lambda \,  \mathbf{b}^{(i)}_\lambda =\be_i.
\end{equation}
Since $\lambda\not \in \sigma(\A_0)$, in view of Theorem~\ref{thm:DspecUnionSpecIncl}, the resolvent $(\A_0^{\o^k}-\lambda)^{-1}$ is bounded. Hence,  \eqref{proof lemma properties b lambda equation 1} implies
\begin{equation}
\label{proof lemma properties b lambda equation 2}
\|\mathbf{b}^{(i)}_\lambda \|_{\Lb^2(\o^k)}\le C \,|\o^k|^{1/2}\,.
\end{equation}
Moreover,  taking the inner product of \eqref{proof lemma properties b lambda equation 1} with $\mathbf{b}^{(i)}_\lambda$,  integrating by parts over $\o^k$, and invoking Korn's inequality, one obtains
\begin{equation*}
\label{proof lemma properties b lambda equation 3}
\begin{split}
\|\nabla \mathbf{b}^{(i)}_\lambda \|_{\Lb^2(\o^k)}^2
\le 
\lambda \|\mathbf{b}^{(i)}_\lambda \|_{\Lb^2(\o^k)}^2+\|\mathbf{b}^{(i)}_\lambda \|_{\Lb^2(\o^k)} \,|\o^k|^{1/2}
\overset{\eqref{proof lemma properties b lambda equation 2}}{\le} 
C\,|\o^k|\,.
\end{split}
\end{equation*}
Applying an elementary rescaling argument and  observing that the $L^2$-norm of $\mathbf{b}^{(i)}_\lambda(\cdot/\e,\o)$ and $\nabla \mathbf{b}^{(i)}_\lambda(\cdot/\e,\o)$ over any given box of size $L$ is bounded above by the said $L^2$-norms over all inclusions $\e \o^k$ fully contained in a box of size $L+2\e$ (see Assumption~\ref{main assumption}) we finally arrive at \eqref{lemma properties b lambda equation 1}.
\end{proof}

The matrix-function \eqref{definition of b} allows us to define a stochastic matrix-valued version of the Zhikov $\beta$-function \cite{zhikov2000}.

\begin{definition}
\label{definition beta matrix}
We define the \emph{Zhikov $\beta$-matrix} to be the matrix-function 
\[
\bbeta:[0,+\infty)\setminus \sigma(\A_0) \to \mathrm{Mat}(d;\mathbb{R})\simeq \R^{d^2}
\] 
given by
\begin{equation}
\label{beta matrix}
\bbeta(\lambda):=\E\left[\lambda\, I+\lambda^2 \overline{\mathbf{b}_\lambda}\right],
\end{equation}
where $I$ is the $d\times d$ identity matrix.
\end{definition}

\begin{lemma}
\label{lemma beta matrix symmetric}
The Zhikov $\beta$-matrix \eqref{beta matrix} is symmetric.
\end{lemma}
\begin{proof}
	The symmetry of $ \E[ \overline{\mathbf{b}_\lambda}]$ follows from the self-adjointness of $\A_0$, namely,
	\begin{equation*}
	\E\left[  \overline{\mathbf{b}_\lambda^{(i)}}\cdot \be_j \right] = \E\left[  \overline{\mathbf{b}_\lambda^{(i)}}\cdot (\A_0 - \l) \overline{\mathbf{b}_\lambda^{(j)}}\right] = 	\E\left[  (\A_0 - \l)\overline{\mathbf{b}_\lambda^{(i)}}\cdot \overline{\mathbf{b}_\lambda^{(j)}}\right] = \E\left[  \be_i \cdot\overline{\mathbf{b}_\lambda^{(j)}}  \right].
\end{equation*}
\end{proof}

Very much like its scalar counterpart \cite[Theorem~4.3]{CCV2}, the $\bbeta$-matrix determines the spectrum of $\A^\mathrm{hom}$.
\begin{theorem}
\label{propositon spectrum Ahom}
We have
\begin{equation}
\label{spectrum Ahom}
\sigma(\A^\mathrm{hom})= \sigma(\A_0) \cup \overline{\{\lambda \ | \ \bbeta(\lambda)<0\}^\complement}\,,
\end{equation}
where $\bbeta(\lambda)<0$ means that the symmetric matrix $\bbeta(\lambda)$ is negative definite.
\end{theorem}

\begin{remark} \label{defbe}
Observe that if we denote by $\beta$ (non-bold beta) the scalar function
\begin{equation}
\label{non-bold beta}
\beta: [0,+\infty)\setminus \sigma(\A_0) \to \mathbb{R}, \qquad \beta(\lambda):= \max_{\hat{k} \in \mathbb{S}^{d-1}} \hat{k} \cdot \bbeta(\lambda) \hat{k},
\end{equation}
returning, for each $\lambda$ in its domain, the largest eigenvalue of the matrix $\bbeta(\lambda)$, 
then
\begin{equation}
\label{bold beta vs non-bold beta}
\overline{\{\lambda \ | \ \bbeta(\lambda)<0\}^\complement}=\{\lambda\ | \beta(\lambda)\geq 0\}. 
\end{equation}
The RHS of \eqref{bold beta vs non-bold beta} is precisely the expression appearing in the scalar analogue of Theorem~\ref{propositon spectrum Ahom} (although the two $\beta$'s are, in general, different functions).
\end{remark} 

\begin{proof}[Proof of Theorem~\ref{propositon spectrum Ahom}]
To begin with, let us show that
\begin{equation}
\label{proof spectrum Ahom equation 1}
\mathbb{R}\setminus \left( \sigma(\A_0) \cup \overline{\{\lambda \ | \ \bbeta(\lambda)<0\}^\complement} \right) \subset \rho(\A^\mathrm{hom})\,.
\end{equation}
The resolvent problem \eqref{3.10} can be recast as the following coupled system
\begin{subequations}
\begin{equation}
\label{proof spectrum Ahom equation 2a}
\int_{\R^d} \C^\mathrm{hom}\nabla \bu_1 \cdot \nabla \bvarphi_1-\lambda \int_{\R^d}(\bu_1+\E\left[ \bu_0\right])\cdot \bvarphi_1\\
= \int_{\R^d} \E\left[ \bdf\right]\cdot \bvarphi_1,
\end{equation}
\begin{equation}
\label{proof spectrum Ahom equation 2b}
 \E\left[ \C_0 \nabla_y \bu_0(x,\cdot) \cdot \nabla_y \bvarphi_0(x,\cdot)-\lambda (\bu_1(x)+\bu_0(x,\cdot))\cdot\bvarphi_0(x,\cdot)\right]\\
=  \E\left[\bdf(x,\cdot)\cdot \bvarphi_0(x,\cdot)\right], \forall x\in \R^d,
\end{equation}
\end{subequations}
for every $ \bvarphi_1+\bvarphi_0 \in \Vb$.
Since $\lambda \not \in \sigma(\A_0)$,  equation \eqref{proof spectrum Ahom equation 2b} is solvable with the unique solution given by
\begin{equation}
\label{proof spectrum Ahom equation 3}
\bu_0=(\mathcal{A}_0 -\lambda\operatorname{Id})^{-1}\bdf+\lambda \,\overline{\mathbf{b}_\l} \, \bu_1
\end{equation}
(recall that $\mathbf{b}_\l$ is the matrix-valued random variable defined in \eqref{definition of b}).
Substituting \eqref{proof spectrum Ahom equation 3} into \eqref{proof spectrum Ahom equation 2a} and using \eqref{beta matrix} we obtain
\begin{equation}
\label{proof spectrum Ahom equation 4}
\mathcal{A}_1 \bu_1-\bbeta(\lambda)\bu_1= \E[ \bdf ]+\lambda \E[ (\mathcal{A}_0 -\lambda\operatorname{Id})^{-1}\bdf ].
\end{equation}
Since $\mathcal{A}_1$ is non-negative and $\bbeta(\lambda)<0$, the operator $\mathcal{A}_1-\bbeta(\lambda)$ is positive and \eqref{proof spectrum Ahom equation 4} is uniquely solvable by the Fredholm alternative. This implies \eqref{proof spectrum Ahom equation 1}.

Next, let us assume that $\lambda \not \in \sigma(\A_0)$ and $\bbeta(\lambda)$ has at least one non-negative eigenvalue. The task is to show that $\lambda \in \sigma(\A^\mathrm{hom})$. To this end, it suffices to construct a sequence $\bu_n\in \Cb_0^\infty(\R^d)$, $n\in \mathbb{N}$,  $\|\bu_n\|_{\Lb^2}=1$, such that
\begin{equation}
\label{proof spectrum Ahom equation 5}
\lim_{n\to\infty}\|(\A_1-\bbeta(\lambda))\bu_n\|_{\Lb^2}=0.
\end{equation}
Indeed, in this case the sequence $\bu_n + \overline{\bv_n}$, with $\overline{\bv_n} : = \l   \overline{\mathbf{b}_\l}\bu_n \in \Lb^2(\R^d;\Hb^1_0(\Omega))$, is singular for the limit operator $\A^\hom$. In order to see this we first observe that for two functions $\bu\in\Hb^1(\R^d)$ and $\overline{\bv}  \in \Lb^2(\R^d;\Hb^1_0(\Omega))$  such that $\A_1 \bu=\bdf$ and $\A_0 \overline{\bv} = \overline{\bg}$ one has
\begin{equation*}
	\A^\hom \bu =(1-\overline{1_\Omega}) \E[(1-\overline{1_\Omega})]^{-1}\, \bdf ,
\end{equation*}
\begin{equation*}
	\A^\hom \overline{\bv} = \overline{\bg} - (1-\overline{1_\Omega}) \E[(1-\overline{1_\Omega})]^{-1}\E[\overline{\bg}].
\end{equation*}
Then elementary calculations give
\begin{equation*}
	\lim_{n\to\infty}\|(\A^\hom-\lambda)(\bu_n+\overline{\bv_n})\|_{\Lb^2(\R^d\times \Omega)}=0.
\end{equation*}

We seek $\bu_n$ in the form
\begin{equation}
\label{proof spectrum Ahom equation 6}
\bu_n(x)=\frac{\eta_n(x)\,e^{i r \, k  \cdot x}\, c}{\|\eta_n(x) c\|_{\Lb^2(\R^d)}},
\end{equation}
where $ k $ is a given vector in $\mathbb{S}^{d-1}$, the cut-off function $\eta_n$ is defined in  accordance with \eqref{def eta_L},
and $c\in \R^d$, $r\in \mathbb{R}$ are constants chosen as follows.  First, we choose $r$ in such a way that\footnote{Note that $ k \cdot \C \cdot  k $ is a $d\times d$ matrix with matrix elements
\begin{equation*}
\label{proof spectrum Ahom equation 8}
( k \cdot \C^\mathrm{hom} \cdot  k )_{\alpha\beta}=\C^\mathrm{hom}_{\alpha\mu\nu\beta}  k _\mu  k _\nu.
\end{equation*}
Due to the properties of the tensor $\C^\mathrm{hom} $,  the matrix $ k  \cdot\C^\mathrm{hom}\cdot   k $ is clearly symmetric and positive definite.
}
\begin{equation*}
\label{proof spectrum Ahom equation 9}
\det (r^2\, k \cdot \C^\mathrm{hom}  \cdot  k  -\bbeta(\lambda))=0.
\end{equation*}
This is clearly possible, because $\beta(\lambda)$ has at least one non-negative eigenvalue. The vector $c$ is then chosen in the kernel of $r^2\, k \cdot \C^\mathrm{hom}  \cdot  k  -\bbeta(\lambda)$,  i.e., 
\begin{equation}
\label{proof spectrum Ahom equation 10}
\left[r^2\, k \cdot \C^\mathrm{hom} \cdot  k  -\bbeta(\lambda)\right] c=0.
\end{equation}
It should be noted that $r$ depends on the choice of $ k $, whereas $c$ depends on the choice of both $ k $ and $r$.  

Then formulae  \eqref{bound grad eta}, \eqref{proof spectrum Ahom equation 6} and \eqref{proof spectrum Ahom equation 10} imply
\begin{equation*}
\label{proof spectrum Ahom equation 11}
\|(\A_1-\bbeta(\lambda))\bu_n\|_{\Lb^2(\R^d)}\le \frac{C}{n}.
\end{equation*}
which, in turn, gives us \eqref{proof spectrum Ahom equation 5}.

It only remains to show that $\sigma(\A_0)\subset \sigma(\A^\mathrm{hom})$. Let $\lambda\in \rho(\A^\mathrm{hom})$. In order to conclude, by the inverse mapping theorem    it is enough to show that for every $\overline{\bg}\in \Lb^2_0(\Omega)$ the equation
\begin{equation}
\label{proof spectrum Ahom equation 12}
(\A_0-\lambda)\overline{\bv}=\overline{\mathbf{g}}
\end{equation}
admits a solution.

Since $\lambda\not\in \sigma(\mathcal{A}^\mathrm{hom})$,  let $\bu=\bu_1+\overline{\bu_0}$ the solution to \eqref{proof spectrum Ahom equation 2a}, \eqref{proof spectrum Ahom equation 2b}  with $\overline{\bdf}=h \overline{\mathbf{g}}$, for some $h\in L^2(\R^d)$. Then, by \eqref{proof spectrum Ahom equation 2b},  we have
\begin{equation*}
(\A_0-\lambda)\overline{\bu_0}=\lambda \bu_1+h\,\overline{\mathbf{g}},
\end{equation*}
and hence\footnote{Recall that $\overline{\mathbf{b}_\lambda}$ is defined by \eqref{definition of b}. }
\begin{equation}
\label{proof spectrum Ahom equation 14}
(\A_0-\lambda)(\overline{\bu_0}-\lambda \overline{\mathbf{b}_\lambda} \bu_1)=h\,\overline{\mathbf{g}} .
\end{equation}
Equation \eqref{proof spectrum Ahom equation 14} now implies that 
\begin{equation*}
\label{proof spectrum Ahom equation 15}
\overline{\bv}:=\int_{\R^d}\frac{h}{\|h\|_{L^2}}(\overline{\bu_0}-\lambda \overline{\mathbf{b}_\lambda} \bu_1)
\end{equation*}
is a solution of \eqref{proof spectrum Ahom equation 12},
as one can establish by a straightforward calculation. This completes the proof.
\end{proof}

In plain English, the above theorem tells us that the spectrum of $\A^\mathrm{hom}$ is comprised of two parts: (i) those $\lambda$'s for which $\bbeta(\lambda)$ has at least one non-negative eigenvalue and (ii) the `microscopic resonances' $\sigma(\A_0)$.

The following statement is an immediate consequence of Theorem~\ref{thm 3.3}, \textcolor{black}{which can be proved by retracing the argument from \cite[Proposition~2.2]{zhikov2004}}.
\begin{corollary}\label{c 4.12}
We have 
\begin{equation}
\label{eq spectrum Ahom contained in lim spec}
\sigma(\A^\mathrm{hom})\subset \lim_{\epsilon\to 0}\sigma(\A^\epsilon)
\end{equation}
\textcolor{black}{almost surely}.
\end{corollary}

\section{Spectral convergence}
\label{Spectral convergence}

The goal of this section is to characterise the RHS of \eqref{eq spectrum Ahom contained in lim spec}.  

Recall that the spectrum of the homogenised operator is characterised by Zhikov's $\bbeta$-matrix, cf.~Definition~\ref{definition beta matrix} and Theorem~\ref{propositon spectrum Ahom}.  As it turns out, the limiting spectrum $\lim_{\epsilon\to 0}\sigma(\A^\epsilon)$ is determined by a similar quantity, albeit scalar and nonlocal --- the $\beta_\infty$-function --- which ``feels'' the atypical local arrangements of inclusions.

\begin{definition}
\label{definition beta infinity}
We define the \emph{$\beta_\infty$-function} to be the function
\[
\beta_\infty:\R^{+} \backslash \sigma(\A_0) \times \Omega \to \R,
\]
\begin{equation}
\label{eq definition beta infinity}
\beta_\infty(\lambda, \omega):=\liminf_{L\to+\infty}\sup_{x\in \R^d} \sup_{ k \in \mathbb{S}^{d-1}} \frac{1}{|\square_x^L|}\int_{\square_x^L} \left(\lambda+\lambda^2\,  k \cdot \mathbf{b}_\lambda(y,\omega)  k  \right)  dy\,.
\end{equation}
\end{definition}

\begin{remark} 
	Since for every fixed $\l$ the random variable $\overline{\mathbf{b}_\lambda}$ is defined up to a set of probability measure zero (which in general depends on $\l$), in order \textcolor{black}{to} define $\beta_\infty$ as a function of $\l$ almost surely we need to work with an appropriate version (in the probabilistic sense of the word) of the process $(\overline{\mathbf{b}_\lambda})_{\l\notin \sigma(\A_0)}$. This version can be obtained by using the Kolmogorov--\v{C}entsov Theorem (see e.g. \cite[Theorem 2.8]{KS}) and the fact that
	$$\E\left[|\overline{\mathbf{b}_{\lambda_1}}-\overline{\mathbf{b}_{\lambda_2}}|^2 \right] \leq C(\lambda_1-\lambda_2)^2,$$
	whenever one stays away from the spectrum of $\A_0$. Note that in  \cite{CCV2} the authors adopted a different approach, defining an appropriate version of the scalar analogue of $\overline{\mathbf{b}_\lambda}$ through realisations, \textcolor{black}{cf.~\eqref{3.18}}.
\end{remark} 	

The following proposition lists some basic properties of $\beta_\infty$. Their proof  is a straightforward adaptation of the arguments presented in \cite[Lemma 5.9 and~Proposition 5.11]{CCV2} and is not provided here for the sake of brevity.

\begin{proposition}\label{prop properties of beta inf}
\begin{enumerate}[(i)]
	\item The limit inferior in the right hand side of \eqref{eq definition beta infinity} can be replaced by the limit \textcolor{black}{almost surely}. 
	\item 	The function 
	$$(\lambda,\omega) \mapsto \beta_\infty(\lambda,\omega) $$
	is deterministic almost surely, continuous on  $\R^{+} \backslash \sigma(\A_0)$, and strictly increasing on every connected component (interval) thereof.	
\end{enumerate}
\end{proposition}

In view of Proposition~\ref{prop properties of beta inf}, in what follows we suppress the dependence of $\beta_\infty$ on $\omega$ in our notation and write simply $\beta_\infty(\l)$.

\begin{remark} 
Note that for $\lambda \geq 0$ the largest eigenvalue $\beta(\lambda)$ of the matrix $\bbeta(\lambda)$ defined in accordance with \eqref{non-bold beta} satisfies the following bound: 
\begin{equation} 
\label{beta less than beta infinity}
\begin{split}
 \beta(\lambda) &= \max_{\hat k\in \mathbb{S}^{d-1}} \hat k \cdot \bbeta(\lambda) \hat{k}\\ 
 &
 = \max_{\hat k\in \mathbb{S}^{d-1}}\lim_{L\to+\infty}  \frac{1}{|\square_x^L|}\int_{\square_x^L} \left(\lambda+\lambda^2\, \hat k\cdot \mathbf{b}_\lambda(T_y\omega) \hat k \right)\textcolor{black}{dy}\\
 &
 = 
 \lim_{L\to+\infty} \max_{\hat k\in \mathbb{S}^{d-1}} \frac{1}{|\square_x^L|}\int_{\square_x^L} \left(\lambda+\lambda^2\, \hat k\cdot \mathbf{b}_\lambda(T_y\omega) \hat k \right)\textcolor{black}{dy}\\
  &
  \overset{\eqref{eq definition beta infinity}}{\leq} \beta_{\infty} (\lambda).
\end{split}
\end{equation} 		
Here $x$ is an arbitrary point in $\R^d$, and in writing the second equality we used the Ergodic Theorem. 
Formula \eqref{beta less than beta infinity} provides a generalisation to systems of its scalar counterpart \cite[Remark 5.7]{CCV2}, with the Zhikov $\beta$-function being replaced by the largest eigenvalue of the $\bbeta$-matrix (denoted here by $\beta(\lambda)$ by analogy) --- cf.~Remark~\ref{defbe}.
\end{remark}

A key role in the upcoming analysis will be played by the set
\begin{equation}
\label{G}
\mathcal{G}:=\sigma(\A_0) \cup \overline{\{\lambda:\beta_\infty(\lambda)\ge 0\}}\,,
\end{equation}
compare with \eqref{spectrum Ahom} and \eqref{bold beta vs non-bold beta}. 

\subsection{\textcolor{black}{Outer} bound for the limiting spectrum}
\label{Upper bound for the limiting spectrum}

Our first result provides an \textcolor{black}{outer} bound for the limiting spectrum.

\begin{theorem}
\label{theorem limiting spectrum sub G}
We have
\begin{equation*}
\label{theorem limiting spectrum sub G equation 1}
\lim_{\epsilon\to 0}\sigma(\A^\epsilon)\subseteq \mathcal{G}
\end{equation*}
\textcolor{black}{almost surely}.
\end{theorem}

\begin{proof}
Let $\lambda_0\in \lim_{\epsilon\to 0}\sigma(\A^\epsilon)$.  This means that there exist $\lambda_\epsilon\in \sigma(\A^\epsilon)$ and $\bu^\epsilon\in \mathcal{D}(\A^\epsilon)$, $\|\bu^\epsilon\|_{\Lb^2(\R^d)}=1$, such that
\begin{equation*}
\label{proof limiting spectrum sub G equation 1}
\lim_{\epsilon\to 0} \lambda_\epsilon=\lambda_0,
\end{equation*}
\begin{equation}
\label{proof limiting spectrum sub G equation 2}
\|\A^\epsilon \bu^\epsilon-\lambda_\epsilon \bu^\epsilon\|_{\Lb^2(\R^d)}=:\delta_\epsilon=
o(1) \quad \text{as}\quad \epsilon\to 0.
\end{equation}
\textcolor{black}{In what follows, we will assume that $\delta_\varepsilon\neq 0$, i.e., $\lambda_\varepsilon$ is not an eigenvalue for all sufficiently small $\varepsilon$. The modifications to the proof for the case where $\lambda_\varepsilon$ is an eigenvalue for some (or all) $\varepsilon$ are straightforward.}

If $\lambda_0\in \sigma(\A_0)$ there is nothing to prove. Hence, without loss of generality, let us assume that $\lambda_0\not\in \sigma(\A_0)$.  The task at hand is then to show that 
\begin{equation}
	\label{task proof theorem upper bound}
	\beta_\infty(\lambda_0)\ge 0\,.
\end{equation}

For the future reference we note that \eqref{proof limiting spectrum sub G equation 2} immediately implies the following bound on the symmetrised gradient of $\bu^\epsilon$:
\begin{equation}
	\label{proof limiting spectrum sub G equation 6}
	\|\sym \nabla \bu^\epsilon\|_{[\Lb^2(\M^\epsilon)]^{d}}+\epsilon\|\sym \nabla \bu^\epsilon\|_{[\Lb^2(\I^\epsilon)]^{d}}\le C.
\end{equation}

The idea of the proof goes as follows. Intuitively, one can think of the soft inclusions as micro-resonators, and the local average of the term $\lambda^2\,   \mathbf{b}_\lambda(y,\omega)$ in \eqref{eq definition beta infinity} represents a local {\it resonant} (or {\it anti-resonant}, depending on the sign) contribution from said inclusions. Thus, $\beta_\infty(\l)$ quantifies the maximal possible resonant contribution of the soft component on large scales. One expects that this resonant behaviour is present in the structure of the approximate eigenfunctions $\bu^\e$ introduced above. Therefore, we devise a procedure which will bring this feature to the surface in the limit $\epsilon\to 0$, and relate it to  $\beta_\infty(\l)$. We argue that one can restrict the attention to large fixed size cubes which carry significant proportion of the energy of $\bu^\e$, at the same time keeping the error in the quasimode approximation relatively small. Once one has chosen the cubes appropriately, it only remains to shift such cubes to the origin in order to utilise a compactness argument. We proceed with this plan in mind.

Let $L>1$. We claim that there exists $y=y_{L,\epsilon}\in \R^d$ such that
\begin{equation}
\label{proof limiting spectrum sub G equation 3}
\|\bu^\epsilon\|_{\Lb^2(\square_{y}^L)}>0
\end{equation}
and
\begin{equation}
\label{proof limiting spectrum sub G equation 4}
\frac1{3^{d+1}\delta_\epsilon^2}\|\A^\epsilon \bu^\epsilon-\lambda_\epsilon \bu^\epsilon\|^2_{\Lb^2(\square_{y}^{3L})}
+
\frac{1}{3^{d+1}}\frac{\|\mathbbm{1}_{\M^\epsilon}\sym\nabla \bu^\epsilon\|^2_{[\Lb^2(\square_{y}^{3L})]^{d}}}{\|\mathbbm{1}_{\M^\epsilon}\sym \nabla \bu^\epsilon\|^2_{[\Lb^2(\R^d)]^{d}}}
+
\frac{1}{3^{d+1}}\|\bu^\epsilon\|^2_{\Lb^2(\square_{y}^{3L})}
\le \|\bu^\epsilon\|^2_{\Lb^2(\square_{y}^L)}\,.
\end{equation}
Indeed,
\begin{multline}
\label{proof limiting spectrum sub G equation 5}
\frac{1}{3^{d+1}}\sum_{z\in L \mathbb{Z}}\left[\frac{\|\A^\epsilon \bu^\epsilon-\lambda_\epsilon \bu^\epsilon\|^2_{\Lb^2(\square_{z}^{3L})}}{\|\A^\epsilon \bu^\epsilon-\lambda_\epsilon \bu^\epsilon\|^2_{\Lb^2(\R^d)}}
+
\frac{\|\mathbbm{1}_{\M^\epsilon}\sym \nabla \bu^\epsilon\|^2_{[\Lb^2(\square_{z}^{3L})]^{d}}}{\|\mathbbm{1}_{\M^\epsilon}\sym \nabla \bu^\epsilon\|^2_{[\Lb^2(\R^d)]^{d}}}
+
\|\bu^\epsilon\|^2_{\Lb^2(\square_{z}^{3L})}\right]
\\
=
1
=
\sum_{z\in L \mathbb{Z}}\|\bu^\epsilon\|^2_{\Lb^2(\square_{z}^L)}.
\end{multline}
Then, if we restrict the summation in \eqref{proof limiting spectrum sub G equation 5} to the lattice points $z$ such that $\|\bu^\epsilon\|_{\Lb^2(\square_{z}^L)}\ne 0$,  \eqref{proof limiting spectrum sub G equation 4} is an immediate consequence of the equality between real series with strictly positive terms\footnote{Of course, there exists at least one $z\in L\mathbb{Z}$ such that $\|\bu^\epsilon\|_{\Lb^2(\square_{z}^L)}\ne 0$ because $\bu^\epsilon$ is normalised in $\Lb^2$.}. 

We define a new sequence of vector-functions $\bw^\epsilon_L$, normalised in $\Lb^2(\square_0^L)$,  as
\begin{equation}
\label{proof limiting spectrum sub G equation 7}
\bw_L^\epsilon(x):=\frac{\bu^\epsilon(x+y)}{\|\bu^\epsilon\|_{\Lb^2(\square_{y}^L)}}\,.
\end{equation}
Recall that the shift $y$ is $\epsilon$-dependent. At this point it is convenient to define the set of shifted inclusions 
$$
\e\hat{\o} = \bigsqcup_{k\in \N} \e \hat\omega^k : = \e\o - y,
$$
and   the ``shifted'' operator $\hat \A^\epsilon$ as the analogue of $\A^\epsilon$ with shifted coefficients $\C^\e(x+y,\omega)$.  Similarly, we shall denote by $\hat{\mathbbm{1}}_{\M^\epsilon}$ and $\hat{\mathbbm{1}}_{\I^\epsilon}$ the shifted characteristic functions of $\M^\epsilon$ and $\I^\epsilon$.
 Then, the definition \eqref{proof limiting spectrum sub G equation 7} combined with the bounds \eqref{proof limiting spectrum sub G equation 4} and \eqref{proof limiting spectrum sub G equation 6} gives us
\begin{equation}
	\label{proof limiting spectrum sub G equation 16}
	\|\hat{\mathbbm{1}}_{\M^\epsilon}\sym \nabla \bw_L^\epsilon\|^2_{[\Lb^2(\square_{0}^{3L})]^{d}}
	+
	\|\bw^\epsilon_L\|_{\Lb^2(\square_{0}^{3L})}
	\le C
\end{equation}
and
\begin{equation}
	\label{proof limiting spectrum sub G equation 17}
	\|\hat \A^\epsilon \bw_L^\epsilon-\lambda_\epsilon \bw_L^\epsilon\|_{\Lb^2(\square_0^{3L})} \le  C\,\delta_\epsilon\,.
\end{equation}

Let $\tilde \bu^\epsilon$ be the extension of $\bu^\epsilon|_{\M^\epsilon}$ into $\I^\epsilon$ as per Theorem~\ref{theorem extension}. Clearly, the corresponding extension $\tilde \bw_L^\epsilon$ of $\bw_L^\epsilon$ is given by $\frac{\tilde \bu^\epsilon(x+y)}{\| \bu^\epsilon\|_{\Lb^2(\square_{y}^L)}}$. We denote  $\bv^\epsilon_L:=\bw^\epsilon_L-\tilde{\bw}^\epsilon_L$ --- this function is supported only on the (shifted) inclusions and  carries the information about the resonant effect of the soft component.

First, we will derive some basic bounds for $\tilde{\bw}^\epsilon_L$ and $\bv^\epsilon_L$. Recalling that the extension $\tilde{\bw}^\epsilon$ is chosen to be $\C_0$-harmonic on the set of inclusions, we see that $\bv^\epsilon_L$  satisfies the identity
\begin{equation}
	\label{4.14}
	(-\e^2\operatorname{div} \C_0 \nabla-\lambda_\epsilon)\bv^\epsilon_L=\lambda_\epsilon \tilde \bw^\epsilon_L +\left(-\e^2\operatorname{div} \C_0 \nabla-\lambda_\e \right)\bw^\epsilon_L, \qquad \mbox{in } \e\hat\omega,
\end{equation}
whose rescaled version reads
\begin{equation}
\label{proof limiting spectrum sub G equation 7bis}
(-\operatorname{div} \C_0 \nabla-\lambda_\epsilon)\bv^\epsilon_L(\epsilon x)=\lambda_\epsilon \tilde \bw^\epsilon_L(\epsilon x) +\left(-\operatorname{div} \C_0 \nabla-\lambda_\e \right)\bw^\epsilon_L(\epsilon x), \qquad x\in \hat\omega.
\end{equation}
The latter immediately implies the bound
\begin{equation}
\label{proof limiting spectrum sub G equation 7ter}
\|\bv^\epsilon_L\|_{\Lb^2(\epsilon\o^k)}\le \frac{\|\lambda_\epsilon \tilde \bw^\epsilon_L\|_{\Lb^2(\epsilon\o^k)}+\delta_\epsilon}{\operatorname{dist}(\lambda_\e, \sigma(\A_0))}, \qquad \forall k=1,2,\ldots.
\end{equation}
(Recall the definition of $\delta_\epsilon$ \eqref{proof limiting spectrum sub G equation 2}.)
With  account of \eqref{proof limiting spectrum sub G equation 16}, the extension Theorem~\ref{theorem extension} gives us
\begin{equation}
	\label{proof limiting spectrum sub G equation 18}
	\|\tilde \bw^\epsilon_L\|_{\Hb^1(\square_0^{5L/2})}\le C.
\end{equation}
Furthermore, \eqref{proof limiting spectrum sub G equation 7ter} and \eqref{proof limiting spectrum sub G equation 2} imply that $\tilde \bw^\epsilon_L$ does not vanish in the limit:
\begin{equation}
	\label{proof limiting spectrum sub G equation 19}
	0<C\le \| \tilde \bw^\epsilon_L\|_{\Lb^2(\square_0^{2L})}\,.
\end{equation}
Finally, testing \eqref{4.14} with $\bv^\epsilon_L$ on each inclusion and taking into account  \eqref{proof limiting spectrum sub G equation 7ter},  \eqref{proof limiting spectrum sub G equation 18} and  \eqref{proof limiting spectrum sub G equation 17} we arrive at 
\begin{equation}
	\label{proof limiting spectrum sub G equation 23}
	\epsilon \| \sym \nabla \bw^\epsilon_L\|_{[\Lb^2(\square_0^{2L})]^{d}}\le C.
\end{equation}

Let $\bw_L^0$ and $\bg_\lambda$ be defined as the weak limits
\begin{equation}
\label{proof limiting spectrum sub G equation 8}
\tilde \bw_L^\epsilon \rightharpoonup \bw_L^0\ne 0 \quad \text{in} \quad \Hb^1(\square_0^{5L/2})
\end{equation}
and
\begin{equation}
\label{proof limiting spectrum sub G equation 9}
\mathbf{b}_{\lambda_0}(\epsilon^{-1}(\cdot+y), \omega) \rightharpoonup \bg_{\lambda_0} \quad \text{in} \quad [\Lb^{2}(\square_0^{2L})]^{d},
\end{equation}
respectively, \textcolor{black}{possibly up to the extraction of a subsequence}.

We claim that
\begin{enumerate}[(a)]
\item
the limits \eqref{proof limiting spectrum sub G equation 8} and \eqref{proof limiting spectrum sub G equation 9} exist (indeed, the first follows from \eqref{proof limiting spectrum sub G equation 18} and \eqref{proof limiting spectrum sub G equation 19}, and the second is a consequence of Lemma~\ref{lemma properties b lambda});

\item
we have (recall definition \eqref{def eta_L}) 
\begin{equation}
\label{proof limiting spectrum sub G equation 11}
\lim_{\epsilon\to 0} \int_{\square_0^{2L}} \lambda_\epsilon \bw_L^\epsilon \cdot \tilde \bw_L^\epsilon \eta_{2L} \ge -C L^{-1}
\end{equation}
and

\item
we have\footnote{Note that here the order of the terms matters.}
\begin{equation}
\label{proof limiting spectrum sub G equation 10}
\bv^\epsilon_L = \bw^\epsilon_L-\tilde \bw^\epsilon_L \rightharpoonup \lambda_0 \,\bg_{\lambda_0}\,\bw_L^0 \quad \text{in} \quad \Lb^2(\square_0^{2L}).
\end{equation}
\end{enumerate}
 Suppose this is the case. Then we have
\begin{equation}
\label{proof limiting spectrum sub G equation 12}
\begin{split}
-\frac{C}{L}
&
\overset{(b)}{\le} \lim_{\epsilon\to 0} \int_{\square_0^{2L}} \lambda_\epsilon \bw_L^\epsilon \cdot \tilde \bw_L^\epsilon \eta_{2L}
\\
&
=
\lim_{\epsilon\to 0} \int_{\square_0^{2L}} \lambda_\epsilon \left[\tilde \bw_L^\epsilon+\bv_L^\epsilon \right] \cdot \tilde \bw_L^\epsilon \eta_{2L} 
\\
&
\overset{(a)+(c)}{=}
\int_{\square_0^{2L}} \lambda_0 \left(1+\lambda_0 \,\bg_{\lambda_0}  \right)\bw_L^0 \cdot \bw_L^0 \eta_{2L} 
\\
&
\le
 C\,\beta_\infty(\lambda_0).
\end{split}
\end{equation}
Here we have used that \eqref{proof limiting spectrum sub G equation 8} implies strong convergence in $\Lb^2$.
Furthermore, in the last step we applied the estimate
\begin{equation*}
\label{proof limiting spectrum sub G equation 13}
\lambda_0+\lambda_0^2\,  k \cdot \bg_{\lambda_0}(x)\, k \le \beta_\infty(\lambda_0) \quad \text{for a.e.}\quad x\in \square_0^{2L}, \ \forall  k \in \mathbb{S}^{d-1},
\end{equation*}
which easily follows\footnote{Indeed, taking into account Proposition \ref{prop properties of beta inf} the assertion follows by integrating $\lambda_0+\lambda_0^2\,  k \cdot \mathbf{b}_{\lambda_0}(\epsilon^{-1}(\cdot+y), \omega) \, k$ over an arbitrary fixed cube contained in
$\square^{2L}_0$ and passing to the limit.} from \eqref{eq definition beta infinity},  to 
\begin{multline*}
\label{proof limiting spectrum sub G equation 14}
	\int_{\square_0^{2L}} \lambda_0 \left(1+\lambda_0 \,\bg_{\lambda_0} \right)\bw_L^0 \cdot \bw_L^0 \,\eta_{2L} 
	\\
	=\int_{\square_0^{2L}}\mathbbm{1}_{\{\bw^0_L\ne 0\}}\left(\lambda_0 + \lambda_0^2 \frac{\bw^0_L}{|\bw^0_L|}\cdot \bg_{\lambda_0}\frac{\bw^0_L}{|\bw^0_L|}\right) |\bw^0_L|^2\, \eta_{2L}
	\\
	\leq \beta_\infty(\l_0)\int_{\square_0^{2L}}|\bw^0_L|^2\, \eta_{2L}.
\end{multline*}
Now, the integral on the RHS of the latter is bounded uniformly in $L$ due to \eqref{proof limiting spectrum sub G equation 18} and \eqref{proof limiting spectrum sub G equation 8}.
Since  \eqref{proof limiting spectrum sub G equation 12} holds for an arbitrary $L$ we conclude the validity of \eqref{task proof theorem upper bound}. 

\

The remainder of the proof will be concerned with demonstrating  (b) and (c).


\

Let us  prove (b). Multiplying
\begin{equation}
\label{proof limiting spectrum sub G equation 21}
(\hat \A^\epsilon -\lambda_\epsilon)\bw_L^\epsilon=:\mathbf{f}_L^\epsilon
\end{equation}
by $\eta_{2L} \tilde \bw_L^\epsilon$ and integrating over $\square_0^{2L}$,
one gets
\begin{multline}
\label{proof limiting spectrum sub G equation 22}
\int_{\square_0^{2L}}  \hat{\mathbbm{1}}_{\M^\epsilon}\, \C_1 \nabla \bw_L^\epsilon\cdot \nabla( \eta_{2L}  \bw_L^\epsilon)+\epsilon^2\int_{\square_0^{2L}} \hat{\mathbbm{1}}_{\I^\epsilon} \,\C_0 \nabla \bw_L^\epsilon\cdot \nabla( \eta_{2L} \tilde \bw_L^\epsilon)
\\
=\lambda_\epsilon\int_{\square_0^{2L}} \eta_{2L}\,\bw_L^\epsilon \cdot \tilde \bw_L^\epsilon
+
\int_{\square_0^{2L}} \eta_{2L} \,\mathbf{f}_L^\epsilon \cdot \tilde \bw_L^\epsilon\,.
\end{multline}

We estimate the second integral on LHS  via \eqref{proof limiting spectrum sub G equation 18} and \eqref{proof limiting spectrum sub G equation 23}
to obtain
\begin{equation}
\label{proof limiting spectrum sub G equation 24}
\lim_{\epsilon\to 0}\left| \epsilon^2\int_{\square_0^{2L}} \hat{\mathbbm{1}}_{\I^\epsilon}\,\C_0 \nabla \bw_L^\epsilon\cdot \nabla( \eta_{2L} \tilde \bw_L^\epsilon)\right|=0.
\end{equation}
Furthermore,  \eqref{proof limiting spectrum sub G equation 18} and the properties of $\eta_{2L}$ imply
\begin{equation*}
\label{proof limiting spectrum sub G equation 25}
\liminf_{\epsilon\to 0}\int_{\square_0^{2L}} \hat{\mathbbm{1}}_{\M^\epsilon} \, \C_1 \nabla \bw_L^\epsilon\cdot \nabla( \eta_{2L}  \bw_L^\epsilon) \ge -\frac{C}{L}\,
\end{equation*}
for some $C$ independent of $L$.
Finally, \eqref{proof limiting spectrum sub G equation 17} implies
\begin{equation}
\label{proof limiting spectrum sub G equation 26}
\lim_{\epsilon\to 0}\left|\int_{\square_0^{2L}}  \eta_{2L} \mathbf{f}_L^\epsilon\cdot \tilde \bw_L^\epsilon\right|=0.
\end{equation}
Combining \eqref{proof limiting spectrum sub G equation 22} and \eqref{proof limiting spectrum sub G equation 24}--\eqref{proof limiting spectrum sub G equation 26} we arrive at \eqref{proof limiting spectrum sub G equation 11}.

\

Next, let us show (c).  Following \cite{CCV2}, for the given $L$ we introduce a one-parameter family of averaging operators $P^\epsilon: \Lb^2(\R^d)\to \Lb^2(\R^d)$ defined by
\begin{equation*}
\label{proof limiting spectrum sub G equation 27}
P^\epsilon \mathbf{g}(x):=
\begin{cases}
\frac1{|\e\hat\omega^k|}\int_{\e\hat\omega^k} \mathbf{g}(z)\,dz & \text{for }x\in \e\hat\omega^k, \, k\in \N,
\\
\mathbf{g}(x) & \text{otherwise}.
\end{cases}
\end{equation*}
What the operator $P^\epsilon$ does is that on every shifted inclusion $\e\hat\omega^k$  it replaces the function $\mathbf{g}$ with its average over said inclusion. Note that the set $\epsilon\hat\omega$ changes with $\epsilon$ in a nontrivial way, because for different $\epsilon$'s we get different shifts $y=y_{L,\epsilon}$; moreover, $\epsilon\hat\omega$ comprises disjoint open sets whose diameters tend to zero as $\epsilon$ goes to zero. A slavish adaptation of \cite[Lemma~E.1]{CCV2} to the case of vector-valued functions gives us
\begin{equation}
\label{proof limiting spectrum sub G equation 28}
\lim_{\epsilon\to 0}\|P^\epsilon \mathbf{g}-\mathbf{g}\|_{\Lb^2(\square_0^{3L})}=0
\end{equation}
for every $\mathbf{g}\in \Lb^2(\R^d)$.

With account of \eqref{proof limiting spectrum sub G equation 21}, equation \eqref{4.14} reads
\begin{equation*}
\label{proof limiting spectrum sub G equation 29}
-\epsilon^2 \dive \C_0 \nabla \bv^\epsilon_L-\lambda_\epsilon\,\bv^\epsilon_L=\lambda_\epsilon \tilde \bw_L^\epsilon +\mathbf{f}_L^\epsilon\,.
\end{equation*}
Let us decompose $\bv^\epsilon_L$ as $\bv^\epsilon_L=\hat \bv^\epsilon_L+\mathring \bv^\epsilon_L$, where $\hat \bv^\epsilon_L$ and $\mathring \bv^\epsilon_L$ are defined as solutions of
\begin{equation}
\label{proof limiting spectrum sub G equation 30}
-\epsilon^2 \dive \C_0 \nabla \hat \bv^\epsilon_L-\lambda_0\,\hat \bv^\epsilon_L=\lambda_0 \, P^\epsilon (\tilde \bw_L^\epsilon)\,
\end{equation}
and
\begin{equation}
\label{proof limiting spectrum sub G equation 31}
-\epsilon^2 \dive C_0 \nabla \mathring \bv^\epsilon_L-\lambda_\epsilon\mathring \bv^\epsilon_L=
(\lambda_\e-\lambda_0)  \hat \bv^\epsilon_L
+\lambda_\epsilon \tilde \bw^\epsilon_L-\lambda_0 P^\epsilon(\bw_L^0)+\mathbf{f}^\epsilon_L\,.
\end{equation}
Arguing as in \eqref{proof limiting spectrum sub G equation 7bis}--\eqref{proof limiting spectrum sub G equation 7ter}, one obtains from \eqref{proof limiting spectrum sub G equation 30} the estimate
\begin{equation}
\label{proof limiting spectrum sub G equation 32}
\|\hat \bv^\epsilon_L\|_{L^2(\square_0^{2L})}\le C;
\end{equation}
similarly,  with the help of \eqref{proof limiting spectrum sub G equation 32},
\eqref{proof limiting spectrum sub G equation 28}, 
\eqref{proof limiting spectrum sub G equation 21}, \eqref{proof limiting spectrum sub G equation 17}, 
the identity
\eqref{proof limiting spectrum sub G equation 31} implies
\begin{equation}
\label{proof limiting spectrum sub G equation 33}
\lim_{\epsilon\to 0}\|\mathring \bv^\epsilon_L\|_{\Lb^2(\square_0^{2L})}=0.
\end{equation}
Finally, we observe that 
\begin{equation}
\label{proof limiting spectrum sub G equation 34}
\hat \bv^\epsilon_L=\lambda_0\, \mathbf{b}_{\lambda_0}(x/\epsilon+y,\omega) P^\epsilon(\bw_L^0)
\end{equation}
solves \eqref{proof limiting spectrum sub G equation 30} (recall that the shift $y$ depends on $L$ and $\e$, see \eqref{proof limiting spectrum sub G equation 3} and the line above). 

Formulae \eqref{proof limiting spectrum sub G equation 32},  \eqref{proof limiting spectrum sub G equation 33}, \eqref{proof limiting spectrum sub G equation 34} and \eqref{proof limiting spectrum sub G equation 28} imply that the weak limit of $\bv^\epsilon_L=\hat \bv^\epsilon_L+\mathring \bv^\epsilon_L$ is \eqref{proof limiting spectrum sub G equation 10}, as claimed. This concludes the proof.
\end{proof}

\subsection{\textcolor{black}{Inner} bound on the limiting spectrum}
\label{Lower bound on the limiting spectrum}

In this subsection we will prove an \textcolor{black}{inner} bound for the limiting spectrum under the additional assumption of finite \textcolor{black}{range of dependence}. Combined with Theorem~\ref{theorem limiting spectrum sub G}, this will establish equality between $\mathcal{G}$ and the limiting spectrum.

In order to formulate the \textcolor{black}{assumption on the finite range of dependence} we first need to introduce some notation. Let a set $K\subset \R^d$ be compact. We denote by $\widetilde{\mathcal{F}}_{K}$  the $\sigma$-algebra (in $\R^d$) generated by all compact subsets of $K$ with respect to the Hausdorff distance and
 define a set-valued map $\mathcal{H}_{K}:\Omega\to \mathcal{P}(\R^d)$ by
\[
\mathcal{H}_{K}(\omega):=\omega \cap K.
\]
We also define  a $\sigma$-algebra  $\mathcal{F}_{K}$  (in $\Omega$) as
\[
\mathcal{F}_{K}:=\{ \mathcal{H}_{K}^{-1}(F_{K})\,|\, F_{K}\in \widetilde{\mathcal{F}}_{K}\}.
\]

\begin{assumption}[Finite range of dependence]
\label{assumption finite correlation}
There exists a positive real number $\kappa>0$ such that for every pair of compact sets $K_j\subset \R^d$, $j=1,2$, such that $\operatorname{dist}(K_1,K_2)>\kappa$ the $\sigma$-algebras $\mathcal{F}_{K_1}$ and $\mathcal{F}_{K_2}$ are independent, where \color{black}
\[
\operatorname{dist}(K_1,K_2):=\min\{|x_1-x_2|\,|\,x_1\in K_1, \ x_2\in K_2\}\,.
\]
\end{assumption}

\begin{theorem}
\label{theorem G sub limiting spectrum}
Under Assumption~\ref{assumption finite correlation} we have
\begin{equation*}
\label{theorem G sub limiting spectrum equation 1}
 \mathcal{G} \subseteq \lim_{\epsilon\to 0}\sigma(\A^\epsilon)\,,
\end{equation*}
which,  in conjunction with Theorem~\ref{theorem limiting spectrum sub G},  implies
\begin{equation*}
\label{theorem G sub limiting spectrum equation 2}
\lim_{\epsilon\to 0}\sigma(\A^\epsilon)= \mathcal{G},
\end{equation*}
\textcolor{black}{almost surely}.
\end{theorem}

The proof of Theorem~\ref{theorem G sub limiting spectrum} relies on the observation that given a $\lambda \not \in \sigma(\A_0)$,  one can find an arbitrarily large number of cubes of the same size such that
\begin{enumerate}[(i)]
\item \textcolor{black}{(the interiors of)} these cubes are disjoint and their union is a (larger) cube,
\item the configuration of the set of inclusions in each of these cubes, apart from the boundary layers of thickness $\kappa$ (the range of \textcolor{black}{dependence}), is almost the same, and
\item the quantity $\beta_\infty(\lambda)$ is well approximated by the integral on the RHS of \eqref{eq definition beta infinity} computed for each of these cubes.
\end{enumerate}
\textcolor{black}{This guarantees the almost sure existence of large nearly-periodic areas where one can construct quasi-modes using the simpler ``periodic toolkit'', committing only a small controllable error due to the near-periodicity.}

The \textcolor{black}{existence of nearly-periodic cubes} is formalised by the following proposition, whose proof was given in \cite{CCV2}. \color{black}
Henceforth, we will denote by $\operatorname{dist}_H$ the Hausdorff distance.
\color{black} 

\begin{proposition}[{\cite[Thms.~5.18 and~5.19]{CCV2}}]
\label{proposition cubes}
Under Assumption~\ref{assumption finite correlation}, there exists a set of full measure $\Omega_1\subset \Omega$ such that for every $\omega \in \Omega_1$ the following holds.  Suppose we are given $R_0>0$, $\delta>0$ and  $\lambda \not \in \sigma(\A_0)$. Then there exists $R\geq R_0$, and for any given $N\in \N$ there exists a collection of points $\{x_n\}_{n=1}^{N^d}\subset \R^d$ and a point $\overline{x}\in \R^d$ such that
\begin{equation}
\label{proposition cubes equaiton 1}
\operatorname{Int}\square_{x_j}^{R+\kappa} \cap \operatorname{Int} \square_{x_k}^{R+\kappa}=\emptyset \qquad \forall j\ne k, \ j,k\in \{1,\ldots, N^d\},
\end{equation}
\begin{equation}
\label{proposition cubes equaiton 2}
\bigcup_{n=1}^{N^d}\square_{x_n}^{R+\kappa} =\square_{\overline{x}}^{N(R+\kappa)} ,
\end{equation}
\begin{equation}
\label{proposition cubes equaiton 3}
\operatorname{dist}_{\textcolor{black}{H}}\left(\mathcal{H}_{\square_{0}^R}(\omega-x_j),\mathcal{H}_{\square_{0}^R}(\omega-x_k)\right)<\delta\qquad \forall j,k\in\{1,\ldots, N^d\},
\end{equation}
\begin{equation}
\label{proposition cubes equaiton 4}
\left|\beta_\infty(\lambda)-\sup_{ k \in \mathbb{S}^{d-1}} \frac{1}{|\square_{x_n}^{R+\kappa}|}\int_{\square_{x_n}^{R+\kappa}} \left(\lambda+\lambda^2\,  k \cdot \mathbf{b}_\lambda(y,\omega)  k  \right)  dy\right|<\delta \qquad \forall n\in \{1,\dots,N^d\},
\end{equation}
\begin{multline}
\label{proposition cubes equaiton 5}
\left|\frac{1}{|\square_{x_n}^{R+\kappa}|}\int_{\square_{x_n}^{R+\kappa}} \left(\lambda\, \mathrm{I}+\lambda^2\, \mathbf{b}_\lambda(y,\omega) \right)  dy
-
\frac{1}{|\square_{x_m}^{R+\kappa}|}
\int_{\square_{x_m}^{R+\kappa}} \left(\lambda\, \mathrm{I}+\lambda^2\, \mathbf{b}_\lambda(y,\omega) \right)  dy\right|<\delta \\\qquad \forall n,m\in \{1,\dots,N^d\}\,.
\end{multline}
\end{proposition}

\begin{remark}
Note that the effect of \textcolor{black}{the range of dependence} manifests itself explicitly in the bound \eqref{proposition cubes equaiton 3}, which is only true when one removes a boundary layer of width $\kappa$ from our cubes $\square_{x_n}^{R+\kappa}$. Indeed, distributions of inclusions in $\square_{x_j}^{R}$ and $\square_{x_k}^{R}$  are independent for all $j\ne k$, $j,k\in \{1,\ldots, N^d\}$ by Assumption~\ref{assumption finite correlation}.
\end{remark}

\begin{proof}[Proof of Theorem~\ref{theorem G sub limiting spectrum}]
Suppose we are given $\lambda_0 \in \mathcal G$.   The task at hand is to show that 
\begin{equation}
\label{proof G sub limiting spectrum equation 1}
\lim_{\epsilon\to 0} \operatorname{dist}(\lambda_0,\sigma(\A^\epsilon))=0.
\end{equation}
Due to Corollary \ref{c 4.12} and the continuity of $\beta_\infty$, see Proposition \ref{prop properties of beta inf}, it is sufficient to prove this for $\l_0$ such that $\beta_\infty(\lambda_0)>0$.
The proof, rather long and technical, will be broken into several steps for the sake of clarity.

\

\textbf{Step 1}: \emph{Approximating $\beta_\infty(\lambda_0)$ on \textcolor{black}{nearly}-periodic cubes}.\\

Let us fix sufficiently small $\delta>0$ and let $R_0, L>0$ be sufficiently large positive numbers and suppose we have fixed $\epsilon>0$. Then by Proposition \ref{proposition cubes} there exists $R\geq R_0$ such that the following holds: there exist $\overline x\in \R^d$ and $\{x_k\}_{k=1}^{N^d}\subset \R^d$ for which \eqref{proposition cubes equaiton 1}--\eqref{proposition cubes equaiton 4} are satisfied, where $N=N(\epsilon)$ is chosen to be the smallest positive integer such that
\begin{equation}
\label{proof part2 equation 12}
\epsilon N(R+\kappa)\ge L.
\end{equation}
Note that $N(\epsilon)=O(\epsilon^{-1})$ as $\epsilon\to 0$, and the choice of the centres of the cubes $\overline x$, $\{x_k\}_{k=1}^{N^d}$ depends on $\e$, which we, however, do not reflect in the notation for the sake of brevity.
The above condition \eqref{proof part2 equation 12} ensures that
\begin{equation}
\label{proof part2 equation 13}
\square_{\epsilon \overline{x}}^{L}\subseteq \square_{\epsilon \overline{x}}^{\epsilon N(R+\kappa)} = \bigcup_k \square_{\epsilon {x_k}}^{\epsilon(R+\kappa)}.
\end{equation}

We pick one of the cubes as the approximation basis for the construction that follows. Since it is not important which cube we use, we pick $\square_{x_1}^{R+\kappa}$. Let 
\begin{equation*}
\label{proof part2 equation 3}
\boldsymbol{\ell}:=\frac{1}{|\square_{x_1}^{R+\kappa}|}\int_{\square_{x_1}^{R+\kappa}} (\lambda_0 \, \mathrm{I}+\lambda_0^2\,  \mathbf{b}_{\lambda_0}(y,\omega))\, dy\,.
\end{equation*}
Then \eqref{proposition cubes equaiton 4} can be recast as
\begin{equation*}
\label{proof part2 equation 4}
|\beta_\infty(\lambda_0)-\lambda_\mathrm{max}(\boldsymbol{\ell})|< \delta,
\end{equation*}
where $\lambda_\mathrm{max}(\boldsymbol{\ell})$ denotes the greatest eigenvalue of the $d\times d$ matrix $\boldsymbol{\ell}$. Without loss of generality, we can assume that $\lambda_\mathrm{max}(\boldsymbol{\ell})>0$. This can always be achieved by choosing $\delta$ small enough.

\

\textbf{Step 2}: \emph{Auxiliary periodic problem}.\\
Let 
\begin{equation}
\label{proof part2 equation 5}
\tilde\square_{x_1}^{R+\kappa}:=\square_{x_1}^{R+\kappa} \setminus \left( \bigcup_{\o^k \subset\subset \square_{x_1}^{R}} \o^k \right)
\end{equation}
be the set obtained from $\square_{x_1}^{R+\kappa}$ by removing the inclusions fully contained within the smaller cube $\square_{x_1}^{R}$.

For $j,k\in\{1,\ldots,d\}$, let $\hat \Nb^{(j,k)}\in \Hb^1_{\mathrm{per}}(\tilde\square_{x_1}^{R+\kappa})$ be the unique zero-mean solution to
\begin{equation}
\label{proof part2 equation 6}
\int_{\tilde\square_{x_1}^{R+\kappa}} (\C_1)_{\alpha\beta\mu\nu}\left(\delta_{\mu j}\delta_{\nu k}+\nabla_\mu \hat N_\nu^{(j,k)} \right)\,\nabla_\alpha \varphi_\beta =0, \qquad \forall \bvarphi\in \Hb^1_{\mathrm{per}}(\tilde\square_{x_1}^{R+\kappa})\,.
\end{equation}
The quantity $\hat \Nb$ is a tensor of order $d^2$ valued in vector functions over $\R^d$, and it has the meaning of \emph{homogenisation corrector} for the auxiliary periodic problem obtained by covering $\R^d$ with copies of the perforated cubical domain $\tilde\square_{x_1}^{R+\kappa}$.  
Let us first extend $\hat \Nb^{(j,k)}$ to the whole of $\square_{x_1}^{R+\kappa}$ by Theorem~\ref{theorem extension}, and then to the whole of $\R^d$ by periodicity. With slight abuse of notation, we denote the resulting object $\hat \Nb^{(j,k)}\in \Hb^1_{\mathrm{per}}(\square_{x_1}^{R+\kappa})$ by the same symbol.

We define the periodic homogenised tensor $\hat \C^\mathrm{hom}$ via
\begin{equation}
\label{proof part2 equation 7}
\hat{\C}^\mathrm{hom}\xi :=\frac{1}{|\square_{x_1}^{R+\kappa}|}\int_{\tilde \square_{x_1}^{R+\kappa}} \C_1(\xi+\xi_{jk}\nabla \hat \Nb^{(j,k)}) \qquad \forall \xi\in \R^{d^2},  \ \xi=\sym \xi\,.
\end{equation}

\

\textbf{Step 3}: \emph{Estimates for the periodic corrector}.\\
Formula \eqref{proof part2 equation 6} immediately gives us
\begin{equation*}
\label{proof part2 equation 8}
\int_{\square_{x_1}^{R+\kappa}} \C_1\,\nabla \hat \Nb^{(j,k)}\cdot \nabla \hat \Nb^{(j,k)}=-\int_{\square_{x_1}^{R+\kappa}} (\C_1)_{\alpha\beta jk}\,\nabla_\alpha \hat N_\beta^{(j,k)}, \qquad j,k\in\{1,\ldots,d\},
\end{equation*}
(in the above formula there is no summation over $j$ and $k$). The latter and \eqref{ellipticity} imply
\begin{equation}
\label{proof part2 equation 9}
\|\sym \nabla \hat \Nb^{(j,k)}\|_{[\Lb^2(\square_{x_1}^{R+\kappa})]^d}\le C (R+\kappa)^{\frac{d}{2}}.
\end{equation}
The Korn and Poincar\'e inequalities then give us
\begin{equation}
\label{proof part2 equation 10}
\| \hat \Nb^{(j,k)}\|_{\Lb^2(\square_{x_1}^{R+\kappa})}\le C (R+\kappa)^{\frac{d}{2}+1}.
\end{equation}

The above estimates are rather straightforward. However,  further on in the proof, in order to close the estimates, we will need higher (than $\Hb^1$) regularity of the periodic corrector.  The key technical result in this regard is stated and proved in Appendix~\ref{Appendix higher regularity} in the form of Theorem~\ref{theorem higher regularity corrector}, which states that there exists a constant $C>0$ and $p>2$, both independent of $R$,  such that 
\begin{equation}
\label{proof part2 equation 11}
\|\nabla \hat \Nb^{(j,k)}\|_{[\Lb^p(\square_{x_1}^{R+\kappa})]^d}\le C (R+\kappa)^{\frac{d}{p}}.
\end{equation}

\

\textbf{Step 4}: \emph{Construction of ``quasimodes''}.\\
Let $ k_0\in \mathbb{S}^{d-1}$ be given and let us define a $d\times d$ matrix $\tilde \C=\tilde \C( k_0)$ component-wise as
\begin{equation}
\label{proof part2 equation 14}
\tilde \C_{\alpha\beta}:=( k_0)_\mu \,\hat \C^\mathrm{hom}_{\alpha\mu\nu\beta}\, ( k_0)_\nu\, , \qquad \alpha,\beta\in\{1,\dots,d\}\,.
\end{equation}
Observe that $\tilde \C$ is symmetric and positive definite.  In the wave equation context the direct analogue of $k_0$ has the meaning of the direction of propagation of a flat wave, cf. \eqref{proof part2 equation 17} below. Its particular choice is unimportant, and it won't affect the final conclusion.

Let us choose $\varpi=\varpi( k_0)\in \R$ and $A=A( k_0)\in \mathbb{S}^{d-1}$ in accordance with
\begin{equation*}
\label{proof part2 equation 15}
\varpi^2=\lambda_\mathrm{max}(\tilde{\C}^{-1/2}\,\boldsymbol{\ell}\,\tilde \C^{-1/2}), 
\end{equation*}
\begin{equation*}
\label{proof part2 equation 16}
A\in \ker(\varpi^2 \tilde \C-\boldsymbol{\ell}),
\end{equation*}
and let us put
\begin{equation}
\label{proof part2 equation 17}
\bu(x):=A\cos(\varpi \,  k_0\cdot x).
\end{equation}
A direct calculation shows that formulae \eqref{proof part2 equation 14}--\eqref{proof part2 equation 17} imply
\begin{equation}
\label{proof part2 equation 18}
-\dive \hat \C^\mathrm{hom} \nabla \bu=\boldsymbol{\ell}\,\bu\,.
\end{equation}

Let $\eta_{L,\epsilon}(x):=\eta(x/L+\epsilon \overline{x})$, where the cut-off $\eta$ is defined as in \eqref{def eta}. Then the normalised vector-functions
\begin{equation*}
\label{proof part2 equation 19}
\bu_{L,\epsilon}:= \frac{\eta_{L,\epsilon}\bu}{\|\eta_{L,\epsilon} \bu\|_{\Lb^2(\R^d)}}
\end{equation*}
are approximate solutions of \eqref{proof part2 equation 18}, that is, they satisfy
\begin{equation}
\label{proof part2 equation 20}
\|\dive \hat \C^\mathrm{hom} \nabla \bu_{L,\epsilon}+\boldsymbol{\ell}\,\bu_{L,\epsilon}\|_{\Lb^2(\R^d)}\le \frac{C}{L}.
\end{equation}
Note that the constant $C$ in the above estimate is independent of $\epsilon$. The vector-functions $\bu_{L,\epsilon}$ can be thought of as ``quasimodes'' as $L\to\infty$, where quotation marks are needed because the ``spectral parameter'' here is a matrix. Also note that (cf. \eqref{proof part2 equation 13}) 
\begin{equation}
\label{proof part2 equation 21}
\operatorname{supp}\bu_{L,\epsilon}\subseteq\square_{\epsilon \overline{x}}^{L}\subseteq \square_{\epsilon \overline{x}}^{\epsilon N(R+\kappa)}\,.
\end{equation}

Out of $\bu_{L,\epsilon}$ one can construct candidates for ``quasimodes'' for the operator $\mathcal{A}^\epsilon$ as follows. Let $\mathbf{b}_{\lambda_0}^\epsilon(x):=\mathbf{b}_{\lambda_0}(x/\epsilon,\omega)$ (recall that $\mathbf{b}_{\lambda_0}$ is defined by \eqref{definition of b}) and define
\begin{equation}
\label{proof part2 equation 22}
\bu_L^{\epsilon}:=(1+\lambda_0 \bb^\epsilon_{\lambda_0})\bu_{L,\epsilon}.
\end{equation}
While the functions $\bu_L^{\epsilon}$ are in the domain of the form \eqref{form of the operator}, they are not in the domain of $\A^\epsilon$. Therefore, we define
\begin{equation*}
\label{proof part2 equation 23}
\hat \bu^\epsilon_L:=(\lambda_0+1)(\A^\epsilon+1)^{-1}\bu_L^{\epsilon}
\end{equation*}
so that
\begin{equation}
\label{proof part2 equation 24}
(\A^\epsilon-\lambda_0)\hat \bu^\epsilon_L=(\lambda_0+1)(\bu^\epsilon_L-\hat \bu^\epsilon_L).
\end{equation}

The task at hand reduces to show that for every $\sigma>0$ one can choose $\epsilon_0>0$, $\delta>0$ and $L>0$ such that
\begin{equation}
\label{proof part2 equation 25}
\|\bu^\epsilon_L-\hat \bu^\epsilon_L\|_{\Lb^2(\R^d)}<\sigma \qquad \forall \epsilon<\epsilon_0.
\end{equation}
The latter, combined with \eqref{proof part2 equation 24} and the fact that, clearly,
\begin{equation*}
\label{proof part2 equation 26}
\|\bu^\epsilon_L\|_{\Lb^2(\R^d)}\ge C>0
\end{equation*}
uniformly in $\epsilon$ and $L$, gives us \eqref{proof G sub limiting spectrum equation 1}.

\

\textbf{Step 6}: \emph{Adding the corrector}.\\
Rather than estimating $\|\bu^\epsilon_L-\hat \bu^\epsilon_L\|_{\Lb^2(\R^d)}$ directly, it is convenient to estimate
$\|\bu^\epsilon_{LC}-\hat \bu^\epsilon_L\|_{\Lb^2(\R^d)}$
instead, where
\begin{equation}
\label{proof part2 equation 28}
\bu^\epsilon_{LC}(x):=\bu_L^\epsilon(x)+\epsilon \partial_k [\bu_{L,\epsilon}]_j(x) \hat\Nb^{(j,k)}(x/\epsilon)
\end{equation}
is a modification of $\bu_L^\epsilon$ by adding the homogenisation corrector.  (Recall that $\hat\Nb^{(j,k)}$ is not the actual stochastic corrector for the operator $\A^\e$, but the auxiliary periodic corrector introduced in Step 2.) 
Doing so will yield the result we are after because, on account of the estimate \eqref{proof part2 equation 10} for the corrector, we have
\begin{equation}
\label{proof part2 equation 29}
\left|
\|\bu^\epsilon_L-\hat \bu^\epsilon_L\|_{\Lb^2(\R^d)}
-
\|\bu^\epsilon_{LC}-\hat \bu^\epsilon_L\|_{\Lb^2(\R^d)}
\right|
\le
\|\bu^\epsilon_L-\bu^\epsilon_{LC}\|_{\Lb^2(\R^d)}\le C\epsilon (R+\kappa).
\end{equation}

\

\textbf{Step 7}: \emph{Key bilinear form estimate}.\\
Let us denote by $a^\epsilon$ the bilinear form associated with the operator $\A^\epsilon+\mathrm{Id}$:
\begin{equation*}
\label{proof part2 equation 30}
a^\epsilon(\bu,\bv):=\int_{\R^d} \left(\C^\epsilon \nabla \bu \,\cdot \nabla \bv +\bu\cdot \bv\right)\,.
\end{equation*}
Following a strategy  proposed in \cite{KamSm1} (which exploits, in turn, a general idea found, e.g., in \cite{vishik}), we claim that the task at hand reduces to proving the key estimate
\begin{equation}
\label{proof part2 equation 31}
|a^\epsilon(\bu^\epsilon_{LC}-\hat \bu^\epsilon_L,\bv)|\le C \,\mathcal{E}(\epsilon,\delta, L) \sqrt{a^\epsilon(\bv,\bv)} \qquad \forall \bv \in \Hb^1(\R^d)
\end{equation}
for some `error' $\mathcal{E}(\epsilon,\delta, L)$ satisfying
\begin{equation}
\label{proof part2 equation 32}
\lim_{\delta \to 0} \lim_{L\to+\infty}\lim_{\epsilon\to 0}\mathcal{E}(\epsilon,\delta, L)=0.
\end{equation}

Suppose \eqref{proof part2 equation 31} holds. Then for $\bv=\bu^\epsilon_{LC}-\hat \bu^\epsilon_L$ one obtains 
\begin{equation}
\label{proof part2 equation 33}
\|\bu^\epsilon_{LC}-\hat \bu^\epsilon_L\|_{\Lb^2(\R^d)} \le\mathcal{E}(\epsilon,\delta, L).
\end{equation}
Formulae \eqref{proof part2 equation 33}, \eqref{proof part2 equation 32} and \eqref{proof part2 equation 29} imply \eqref{proof part2 equation 25}. The remainder of the proof is devoted to proving \eqref{proof part2 equation 31}.

\

Given an arbitrary $\bv\in \Hb^1(\R^d)$, let $\tilde \bv^\epsilon$ be the extension of $\bv|_{\M^\epsilon}$ into $\I^\epsilon$ via Theorem~\ref{theorem extension} and put $\bv_0^\epsilon=\bv-\tilde \bv^\epsilon$.

By the properties of the extension, we have:
\begin{equation}
\label{estimates v extension equation 1}
\|\bv_0^\epsilon\|_{\Lb^2(\R^d)}\le C \epsilon \|\sym \nabla \bv\|_{\Lb^2(\R^d)}\,,
\end{equation}
\begin{equation}
\label{estimates v extension equation 2}
\|\sym \nabla \bv_0^\epsilon\|_{\Lb^2(\R^d)} \le C\,\|\sym \nabla \tilde \bv^\epsilon\|_{\Lb^2(\R^d)}
\end{equation}
and
\begin{equation}
\label{estimates v extension equation 3}
\|\tilde \bv^\epsilon\|_{\Lb^2(\R^d)}+\|\sym \nabla\tilde \bv^\epsilon\|_{\Lb^2(\R^d)}+\epsilon\|\sym \nabla \bv\|_{\Lb^2(\R^d)}\le C \sqrt{a^\epsilon(\bv,\bv)}\,.
\end{equation}

Let us consider the quantity
\begin{equation}
\label{proof part2 equation 34}
a^\epsilon(\bu^\epsilon_{LC},\bv)=
\int_{\M^\epsilon} \C_1 \nabla \bu^\epsilon_{LC}\cdot \nabla \bv 
+
\epsilon^2\int_{\I^\epsilon} \C_0 \nabla \bu^\epsilon_{LC}\cdot \nabla \bv
+
\int_{\R^d} \bu^\epsilon_{LC}\cdot \bv.
\end{equation}

\

\textbf{Step 8}: \emph{Estimates: part 1}.\\
By adding and subtracting $\hat \C^\mathrm{hom}\nabla \bu_{L,\e} \cdot \nabla \tilde \bv$, the first term in the RHS of \eqref{proof part2 equation 34} can be equivalently recast as
\begin{multline}
\label{proof part2 equation 35}
\int_{\M^\epsilon} \C_1 \nabla \bu^\epsilon_{LC}\cdot \nabla \bv
=
\int_{\M^\epsilon} \left[
\hat \C^\mathrm{hom} \nabla \bu_{L,\epsilon}+ (\mathbbm{1}_{\M^\epsilon} \C_1(\delta_{j\cdot}\delta_{k\cdot}+\nabla \hat\Nb^{(j,k)}(\cdot/\e)) -\hat \C^\mathrm{hom} \delta_{j\cdot}\delta_{k\cdot})\partial_k [\bu_{L,\epsilon}]_j
\right.
\\
\left.
+
\epsilon\, \C_1 (\nabla\partial_k [\bu_{L,\epsilon}]_j) \otimes\hat\Nb^{(j,k)}(\cdot/\e)\right]\cdot \nabla \tilde \bv^\epsilon.
\end{multline}
In view of \eqref{proof part2 equation 10}, the elementary estimate 
\begin{equation}
\label{proof part2 equation 35bis}
\|\nabla^2 \bu_{L,\epsilon}\|_{[L^\infty(\R^d)]^{d^3}}\le C L^{-d/2}
\end{equation}
 and \eqref{proof part2 equation 12}, the last term in the RHS of \eqref{proof part2 equation 35} can be estimated as
\begin{equation}
\label{proof part2 equation 36}
\left|\int_{\M^\epsilon} \left[\epsilon\, \C_1 (\nabla\partial_k [\bu_{L,\epsilon}]_j) \otimes\hat\Nb^{(j,k)}(\cdot/\e)\right]\cdot \nabla \tilde \bv\right|\le C\epsilon R\, \|\operatorname{sym}\nabla \tilde \bv^\epsilon\|_{\Lb(\R^d)}\,.
\end{equation}

Next, let us denote by $\hat{\mathbbm{1}}_{\M^\epsilon}$ the characteristic function of the set $\tilde\square_{x_1}^{R+\kappa}$,
cf.~\eqref{proof part2 equation 5}, extended by periodicity to the whole of $\R^d$, and let us rewrite
\begin{multline}
\label{proof part2 equation 38}
\mathbbm{1}_{\M^\epsilon} \C_1(\delta_{j\cdot}\delta_{k\cdot}+\nabla \hat\Nb^{(j,k)}(\cdot/\e)) -\hat \C^\mathrm{hom} \delta_{j\cdot}\delta_{k\cdot}\\
=
(\mathbbm{1}_{\M^\epsilon}-\hat{\mathbbm{1}}_{\M^\epsilon}) \C_1(\delta_{j\cdot}\delta_{k\cdot}+\nabla \hat\Nb^{(j,k)}(\cdot/\e))
+\hat{\mathbbm{1}}_{\M^\epsilon} \C_1(\delta_{j\cdot}\delta_{k\cdot}+\nabla \hat\Nb^{(j,k)}(\cdot/\e))-\hat \C^\mathrm{hom} \delta_{j\cdot}\delta_{k\cdot}\,.
\end{multline}
In view of \eqref{proof part2 equation 21}, we will need to estimate $\mathbbm{1}_{\M^\epsilon}-\hat{\mathbbm{1}}_{\M^\epsilon}$ in $\square_{\epsilon \overline{x}}^{L}$. We claim that, for $R$ large enough, we have
\begin{equation}
\label{proof part2 equation 39}
\int_{\square_{\epsilon \overline{x}}^{L}}|\mathbbm{1}_{\M^\epsilon}-\hat{\mathbbm{1}}_{\M^\epsilon}|\le C \delta L^d.
\end{equation}

The inequality \eqref{proof part2 equation 39} follows from:
\begin{enumerate}[(i)]
\item Formula \eqref{proof part2 equation 13};
\item The estimate
\[
\int_{\square_{\epsilon x_j}^{\epsilon R}}|\mathbbm{1}_{\M^\epsilon}-\hat{\mathbbm{1}}_{\M^\epsilon}|\le C \delta (\epsilon \,R)^d, \qquad j\in\{1, \ldots, N^d\},
\]
which is a consequence of \eqref{proposition cubes equaiton 3} and the fact that the surface area of inclusions in each $\square_{\epsilon x_j}^{\epsilon R}$, $j\in\{1, \ldots, N^d\}$, is bounded above by the volume of the cube, up to a constant (this follows immediately from Assumption \ref{main assumption 2});
\item
The fact that
\[
\frac{\left|\square_{\epsilon x_j}^{\epsilon (R+\kappa)}\setminus \square_{\epsilon x_j}^{\epsilon R}\right|}{|\square_{\epsilon \overline{x}}^{L}|}
\]
can be made arbitrarily small by choosing $R$ (at the start of the proof) sufficiently large.
\end{enumerate}

When substituting \eqref{proof part2 equation 38} into \eqref{proof part2 equation 35}, the first term on the RHS of \eqref{proof part2 equation 38} gives rise to two contributions: 
$(\mathbbm{1}_{\M^\epsilon}-\hat{\mathbbm{1}}_{\M^\epsilon}) \C_1 \nabla \bu_{L,\epsilon} \cdot \nabla \tilde \bv^\epsilon$
and
$(\mathbbm{1}_{\M^\epsilon}-\hat{\mathbbm{1}}_{\M^\epsilon}) \C_1(\nabla \hat\Nb^{(j,k)}(\cdot/\e)\,\partial_k (\bu_{L,\epsilon})_j)\cdot \nabla \tilde \bv^\epsilon$. The former can be estimated via \eqref{proof part2 equation 35bis} and \eqref{proof part2 equation 39} as
\begin{equation}
\label{proof part2 equation 40}
\left|\int_{\R^d} (\mathbbm{1}_{\M^\epsilon}-\hat{\mathbbm{1}}_{\M^\epsilon}) \C_1 \nabla \bu_{L,\epsilon} \cdot \nabla \tilde \bv \right|\le C \,\delta^{\nicefrac{1}{2}}\|\sym\nabla \tilde \bv^\epsilon\|_{[\Lb^2(\R^d)]^2}\,.
\end{equation}
Combining again \eqref{proof part2 equation 35bis}, \eqref{proof part2 equation 39} and the higher regularity of the corrector \eqref{proof part2 equation 11} by means of H\"older's inequality, and applying  Korn's inequality to  $\nabla\tilde \bv^\epsilon$,  we can estimate the latter as
\begin{equation}
\label{proof part2 equation 41}
\left|\int_{\R^d} (\mathbbm{1}_{\M^\epsilon}-\hat{\mathbbm{1}}_{\M^\epsilon}) \C_1(\nabla \hat\Nb^{(j,k)}(\cdot/\e)\,\partial_k [\bu_{L,\epsilon}]_j)\cdot \nabla \tilde \bv^\epsilon \right|\le C \,\delta^{\frac{p-2}{2p}}\, \|\sym \nabla \tilde \bv^\epsilon\|_{[\Lb^2(\R^d)]^2}\,
\end{equation}
for some $p>2$. 

We are left to deal with the last two terms on the RHS of \eqref{proof part2 equation 38}. To this end, we observe that, for fixed $\beta$, the vector field
\begin{equation*}
\label{proof part2 equation 42}
\hat{\mathbbm{1}}_{\M^\epsilon} (\C_1)_{\alpha\beta \mu\nu}(\delta_{j\mu}\delta_{k\nu}+\nabla_\mu N^{(j,k)}_\nu)-(\hat \C^\mathrm{hom})_{\alpha\beta \mu\nu}\,\delta_{j\mu}\delta_{k\nu}
\end{equation*}
(whose components are labelled by the index $\alpha$) is divergence free --- because of \eqref{proof part2 equation 6} --- and has zero mean --- because of \eqref{proof part2 equation 7}. Hence, by a classical construction, see e.g. \cite{ZKO}, there exist $d^3$ skew-symmetric,  zero-mean matrix functions $\Gb_{\cdot\,\cdot;\beta}^{(j,k)}\in [\Hb^1_{\mathrm{per}}(\square_{x_1}^{R+\kappa})]^2$, $\beta,j,k\in\{1,\dots,d\}$, such that
\begin{equation}
\label{proof part2 equation 43}
\partial_\gamma \Gb_{\gamma \alpha;\beta}^{(j,k)}
=
\hat{\mathbbm{1}}_{\M^\epsilon} (\C_1)_{\alpha\beta \mu\nu}(\delta_{j\mu}\delta_{k\nu}+\nabla_\mu N^{(j,k)}_\nu(\cdot/\e))-(\hat \C^\mathrm{hom})_{\alpha\beta \mu\nu}\,\delta_{j\mu}\delta_{k\nu}
\,.
\end{equation}
Observe that the RHS of \eqref{proof part2 equation 43} is symmetric in the pair of indices $\alpha$ and $\beta$.
Moreover, the explicit construction of $\Gb_{\gamma \alpha;\beta}^{(j,k)}$ implies
\begin{multline}
\label{proof part2 equation 44}
\|\Gb_{\cdot\, \cdot;\beta}^{(j,k)}\|_{[\Lb^2(\square_{x_1}^{R+\kappa})]^2}
\\
\le C R \|\hat{\mathbbm{1}}_{\M^\epsilon} (\C_1)_{\cdot\,\cdot \mu\nu}(\delta_{j\mu}\delta_{k\nu}+\nabla_\mu N^{(j,k)}_\nu(\cdot/\e)-(\hat \C^\mathrm{hom})_{\cdot\,\cdot \mu\nu}\,\delta_{j\mu}\delta_{k\nu}\|_{[\Lb^2(\square_{x_1}^{R+\kappa})]^2}
\\
\overset{\eqref{proof part2 equation 9}}{\le} C R^{\nicefrac{d}2 + 1}\,.
\end{multline}

In view of \eqref{proof part2 equation 43}, integrating by parts and using the antisymmetry of $\Gb_{\cdot\,\cdot;\beta}^{(j,k)}$, we get 
\begin{multline}
\label{proof part2 equation 45}
\int_{\R^d} \left[
(\hat{\mathbbm{1}}_{\M^\epsilon} \C_1(\delta_{j\cdot}\delta_{k\cdot}+\nabla \hat\Nb^{(j,k)}(\cdot/\e))-\hat \C^\mathrm{hom} \delta_{j\cdot}\delta_{k\cdot})\,\partial_k [\bu_{L,\epsilon}]_j\right]\cdot \nabla \tilde \bv^\epsilon
\\
=
\int_{\R^d} \left[
\epsilon\,\partial_\gamma \Gb_{\gamma \alpha;\beta}^{(j,k)}\,\partial_k [\bu_{L,\epsilon}]_j\right]\, \partial_\alpha \tilde \bv^\epsilon_\beta
\\
=
-\int_{\R^d}
\epsilon\,\Gb_{\gamma \alpha;\beta}^{(j,k)}\,\partial_\gamma\partial_k [\bu_{L,\epsilon}]_j\,\partial_\alpha \tilde\bv^\epsilon_\beta
-\int_{\R^d}
\underset{=0}{\underbrace{\epsilon\,\Gb_{\gamma \alpha;\beta}^{(j,k)}\,\partial_k [\bu_{L,\epsilon}]_j\, \partial_\gamma\partial_\alpha \tilde\bv^\epsilon_\beta}}
\\
=
-\frac12\int_{\R^d}
\epsilon\,\Gb_{\gamma \alpha;\beta}^{(j,k)}\,\partial_\gamma\partial_k [\bu_{L,\epsilon}]_j\,\left(\partial_\alpha \tilde\bv^\epsilon_\beta+\partial_\beta \tilde\bv^\epsilon_\alpha\right)\,.
\end{multline}

Hence, by combining \eqref{proof part2 equation 36}, \eqref{proof part2 equation 38}, \eqref{proof part2 equation 40}, \eqref{proof part2 equation 41}, \eqref{proof part2 equation 45}, \eqref{proof part2 equation 35bis},  \eqref{proof part2 equation 44}, and \eqref{estimates v extension equation 3}, we obtain that \eqref{proof part2 equation 35} can be rewritten as
\begin{equation}
\label{proof part2 equation 46}
\int_{\M^\epsilon} \C_1 \nabla \bu^\epsilon_{LC}\cdot \nabla \bv
=
\int_{\R^d}
\hat \C^\mathrm{hom} \nabla \bu_{L,\epsilon}\cdot \nabla \bv + \mathcal{R}_1,
\end{equation}
where
\begin{equation*}
\label{proof part2 equation 47}
|\mathcal{R}_1| \le C  \left[R\,\epsilon+\delta^\frac{p-2}{2p}\right] \|\sym \nabla\tilde \bv^\epsilon\|_{[\Lb^2(\R^d)]^2}\,.
\end{equation*}
Here we are assuming, without loss of generality, that $\delta<1$, so that $\delta^{\frac12}< \delta^{\frac{p-2}{2p}}$ for $p>2$.

\

\textbf{Step 8}: \emph{Estimates: part 2}.\\
Let us now examine the second term on the RHS of \eqref{proof part2 equation 34}.

We have
\begin{multline}
\label{proof part2 equation 48}
\epsilon^2\int_{\I^\epsilon} \C_0 \nabla \bu^\epsilon_{LC}\cdot \nabla \bv
=
\epsilon^2\int_{\I^\epsilon} \C_0 \left((1+\lambda_0 \bb^\epsilon_{\lambda_0})\nabla\bu_{L,\epsilon}\right)\cdot \nabla \bv
\\
+
\epsilon^2\int_{\I^\epsilon} \lambda_0[\C_0]_{\alpha\beta\mu\nu} \,\partial_\mu[\bb^\epsilon_{\lambda_0}]_{\nu\gamma}\,[\bu_{L,\epsilon}]_\gamma \, \partial_\alpha[\bv^\epsilon_0]_\beta
+
\epsilon^2\int_{\I^\epsilon} \C_0 \nabla (\epsilon \partial_k [\bu_{L,\epsilon}]_j \hat\Nb^{(j,k)}(\cdot/\epsilon))\cdot \nabla \bv\,.
\end{multline}
Now, integrating by parts and taking into account \eqref{proof lemma properties b lambda equation 1}, the second contribution to the RHS can be recast as
\begin{multline}
\label{proof part2 equation 50}
\epsilon^2\int_{\I^\epsilon} \lambda_0\C_0 \,\nabla\bb^\epsilon_{\lambda_0}\,\bu_{L,\epsilon}\cdot \nabla (\bv^\epsilon_0+\tilde \bv^\epsilon)
\\
=
\int_{\I^\epsilon} \lambda_0 (1+\lambda_0 \bb^\epsilon_{\lambda_0})\,\bu_{L,\epsilon}\cdot \bv^\epsilon_0
-\epsilon^2\int_{\I^\epsilon} \lambda_0 [\C_0]_{\alpha\beta\mu\nu} \,\partial_\mu[\bb^\epsilon_{\lambda_0}]_{\nu\gamma}\,\partial_\alpha[\bu_{L,\epsilon}]_\gamma \, [\bv^\epsilon_0]_\beta
\\
+
\epsilon^2\int_{\I^\epsilon} \lambda_0\C_0 \left(\nabla\bb^\epsilon_{\lambda_0}\,\bu_{L,\epsilon}\right)\cdot \nabla \tilde \bv^\epsilon\,.
\end{multline}

Substituting \eqref{proof part2 equation 50} into \eqref{proof part2 equation 48} and using 
\eqref{proof part2 equation 9}, \eqref{proof part2 equation 10}, 
 \eqref{lemma properties b lambda equation 1}, 
and
\eqref{estimates v extension equation 1}--\eqref{estimates v extension equation 3}
we arrive at
\begin{equation}
\label{proof part2 equation 51}
\epsilon^2\int_{\I^\epsilon} \C_0 \nabla \bu^\epsilon_{LC}\cdot \nabla \bv=
\int_{\I^\epsilon} \lambda_0 (1+\lambda_0 \bb^\epsilon_{\lambda_0})\,\bu_{L,\epsilon}\cdot \bv^\epsilon_0+ \mathcal{R}_2,
\end{equation}
where
\begin{equation*}
\label{proof part2 equation 52}
|\mathcal{R}_2|\le C(\epsilon+\epsilon^2 \,R)\, \sqrt{a^\epsilon(\bv,\bv)}\,.
\end{equation*}

\

\textbf{Step 9}: \emph{Back to the bilinear form}.\\
Substituting \eqref{proof part2 equation 46} and \eqref{proof part2 equation 51} into \eqref{proof part2 equation 34}, we obtain
\begin{equation}
\label{proof part2 equation 53}
a^\epsilon(\bu^\epsilon_{LC},\bv)=
\int_{\R^d} 
\hat \C^\mathrm{hom} \nabla \bu_{L,\epsilon}\cdot \nabla \tilde\bv^\epsilon
+
\int_{\R^d} \lambda_0 (1+\lambda_0 \bb^\epsilon_{\lambda_0})\,\bu_{L,\epsilon}\cdot \bv^\epsilon_0
+
\int_{\R^d} \bu^\epsilon_{LC}\cdot \bv +\mathcal{R}_3\,, 
\end{equation}
where
\begin{equation}
\label{proof part2 equation 54}
|\mathcal{R}_3|\le C\left[(1+R)\epsilon+\delta^{\frac{p-2}{2p}}\right]\, \sqrt{a^\epsilon(\bv,\bv)}\,.
\end{equation}

With \eqref{proof part2 equation 20} in mind, and recalling \eqref{proof part2 equation 28} and
\eqref{proof part2 equation 22}, we rewrite
\eqref{proof part2 equation 53} as
\begin{multline}
\label{proof part2 equation 55}
a^\epsilon(\bu^\epsilon_{LC},\bv)=
\int_{\R^d} 
\left(-\dive \hat \C^\mathrm{hom} \nabla -\boldsymbol{\ell}\right) \bu_{L,\epsilon}\cdot \tilde\bv^\epsilon
\\
+
\int_{\R^d} \left[\boldsymbol{\ell}- \lambda_0 (1+\lambda_0 \bb^\epsilon_{\lambda_0})\right]\,\bu_{L,\epsilon}\cdot 
\tilde \bv^\epsilon
\\
+
\int_{\R^d} \left[\lambda_0 (1+\lambda_0 \bb^\epsilon_{\lambda_0})\bu_{L,\epsilon}+ \bu^\epsilon_{LC}\right]\cdot \bv +\mathcal{R}_3\,.
\end{multline}

By combining formulae 
\eqref{proof part2 equation 55},
\eqref{proof part2 equation 54},
\eqref{proof part2 equation 20},
and using \eqref{proof part2 equation 28} and \eqref{proof part2 equation 24}, we obtain
\begin{multline}
\label{proof part2 equation 56}
|a^\epsilon(\bu^\epsilon_{LC}-\hat \bu^\epsilon_L,\bv)|\le C\left[(1+R)\epsilon+\delta^{\frac{p-2}{2p}}+L^{-1}\right]\, \sqrt{a^\epsilon(\bv,\bv)}
\\
+
\left|\int_{\R^d} \left[\boldsymbol{\ell}- \lambda_0 (1+\lambda_0 \bb^\epsilon_{\lambda_0})\right]\,\bu_{L,\epsilon}\cdot 
\tilde \bv^\epsilon\right|\,.
\end{multline}
Therefore, the remaining task is to estimate the second term on the RHS of the latter inequality.

\

\textbf{Step 10}: \emph{Estimates: part 3}.\\
The argument consists in utilising  \eqref{proof part2 equation 20}  while replacing  $\bu_{L,\epsilon} \otimes
\tilde \bv^\epsilon$ with its average values on the cubes of size $\e(R+\kappa)$, cf. \eqref{proof part2 equation 13}, and using  Poincar\'e's inequality to control the error.
\begin{multline}
\label{proof part2 equation 58}
\left|\int_{\R^d} \left(\boldsymbol{\ell}- \lambda_0 (1+\lambda_0 \bb^\epsilon_{\lambda_0})\right)\,\bu_{L,\epsilon}\cdot 
\tilde \bv^\epsilon\right|
\\
\le 
\sum_{i=1}^{N^d}\left|\int_{\epsilon\square_{\epsilon x_i}^{R+\kappa}} \left(\boldsymbol{\ell}- \lambda_0 (1+\lambda_0 \bb^\epsilon_{\lambda_0})\right)\cdot \fint_{\epsilon\square_{\epsilon x_i}^{R+\kappa}} \bu_{L,\epsilon}\otimes 
\tilde \bv^\epsilon \right|
\\
+
\left|\sum_{i=1}^{N^d}\int_{\epsilon\square_{\epsilon x_i}^{R+\kappa}} \left(\boldsymbol{\ell}- \lambda_0 (1+\lambda_0 \bb^\epsilon_{\lambda_0})\right) \cdot \left(\bu_{L,\epsilon}\otimes 
\tilde \bv^\epsilon - \fint_{\epsilon\square_{\epsilon x_i}^{R+\kappa}} \bu_{L,\epsilon}\otimes 
\tilde \bv^\epsilon\right)\,\right|
\\
\overset{\eqref{proposition cubes equaiton 5}}{\le} 
C\,\delta\,\|\tilde \bv^\epsilon\|_{\Lb^2(\R^d)}
\\
+ C\epsilon(R+\kappa)\|\boldsymbol{\ell}- \lambda_0 (1+\lambda_0 \bb^\epsilon_{\lambda_0})\|_{\Lb^2(\square_{\epsilon\overline x}^{\epsilon N(R+\kappa)})}
\left\|\nabla\left(\bu_{L,\epsilon}\otimes 
\tilde \bv^\epsilon\right)\right\|_{\Lb^2(\square_{\epsilon\overline x}^{\epsilon N(R+\kappa)})}
\\
\overset{\eqref{lemma properties b lambda equation 1}}{\le} 
C\left[\epsilon(R+\kappa) +\delta \right]\left( \|\tilde \bv^\epsilon\|_{\Lb^2(\R^d)}+\|\sym \tilde \nabla\bv^\epsilon\|_{[\Lb^2(\R^d)]^d}\right)\,.
\end{multline}
In the last step we also used the Korn inequality in 
$\square_{\epsilon\overline x}^{\epsilon N(R+\kappa)}$, and then the domain monotonicity of the $L^2$-norm.

\

\textbf{Step 11}: \emph{Conclusion}.\\
By combining \eqref{proof part2 equation 56}, \eqref{proof part2 equation 58} and \eqref{estimates v extension equation 3} we arrive at
\eqref{proof part2 equation 31} with
\begin{equation*}
\label{proof part2 equation 59}
\mathcal{E}(\epsilon,\delta, L)=(1+R)\epsilon+\delta^{\frac{p-2}{2p}}+L^{-1}\,,
\end{equation*}
which clearly satisfies \eqref{proof part2 equation 32}. 

This concludes the proof.
\end{proof}

\section{Examples}
\label{Examples}
In this section we will discuss some explicit examples, to showcase the different properties of $\boldsymbol{\beta}(\lambda)$ and $\beta_\infty(\lambda)$, as well as to compare and contrast the stochastic and the periodic settings.

\color{black}
Note that, in general, $\beta(\lambda)$ (recall~\eqref{non-bold beta}) and $\beta_\infty(\lambda)$ are quite different: as the examples will show, $\boldsymbol{\beta}(\lambda)$ results from ``averaging'' over the inclusions, whereas $\beta_\infty(\lambda)$ from taking a ``supremum''. This means, in particular, that $\sigma(\A^\hom)$ is generally a proper subset of $\mathcal{G}$ --- cf.~Example~3. Remarkably, this is not the case in Example~1, where the randomness (in the final part of the example) is designed in such a way that $\beta_\infty(\lambda)=\beta(\lambda)$, which implies $\sigma(\mathcal{A}^\hom)=\mathcal{G}$.

Let us emphasise that in Examples 1 and 3 Assumption~\ref{assumption finite correlation} is satisfied, and hence the limiting spectrum coincides with the set $\mathcal{G}$ --- see also Remark~\ref{remark after examples}(ii).
\color{black}

\subsection{Example 1: statistical micro-symmetries}
\label{Example 1: statistical micro-symmetries}

Let us work in Euclidean space $\mathbb{R}^3$ equipped with Cartesian coordinates $x_j$, $j\in\{1,2,3\}$.

Let us assume, for simplicity, that for almost every $\omega\in \Omega$ each inclusion $\omega^k$ belongs to a finite set of shapes, up to rigid translations and rotations, and that the tensor $\C_0$ is isotropic, i.e. it acts on 2-tensors as (see, e.g., \cite{ciarlet})
\begin{equation}
\label{isotropic tensor}
\C_0 \xi =c_1 (\textrm{tr} \xi)I+c_2 \sym \xi\,
\end{equation}
for some $c_1,c_2 \in \mathbb{R}$. 

Let us denote by $\omega^0$ the connected component of $\omega$ containing the origin of $\R^3$. In the case when $0\notin \omega$ we set $\o^0 = \emptyset$. 
It was shown in \cite[Lemma~B.4]{CCV2}\footnote{The probabilistic setup of \cite{CCV2} is slightly different from that of the current paper, but it is equivalent to it. We refer the reader to \cite[Remark~2.18]{decay} for a detailed comparison of the two.} that the set-valued mapping $P_0:\omega \mapsto \omega^0$ is measurable, when the target space is endowed with the $\sigma$-algebra $\mathcal{H}$ generated by the topology on $[-2,2]^d$ induced by the Hausdorff distance. Observe that, by assumption, the mapping $P_0$ takes values in a finite set of shapes, up to rigid translations and rotations.

Let us consider the following two types of discrete isometries of $\R^3$.
\vspace{-1.8mm}
\begin{enumerate}[(i)]
\itemsep0em
\item Reflections $\{S_j\}_{j=1}^3$: $S_j$ describes the reflection across the plane $\{x_j=0\}$.
\item Rotations $\{R_j^{\pi/2}\}_{j=1}^3$: $R_j^{\pi/2}$ describes the counterclockwise rotation by  $\pi/2$ about the $x_j$ coordinate axis. 
\end{enumerate} 	

Suppose that the inclusions possess the following additional statistical micro-symmetries:
\begin{eqnarray} 
\label{sim1} 
P(\omega^0 \in \mathfrak{H})=P(\omega^0 \in S_j \mathfrak{H}),& & \qquad \forall \mathfrak{H} \in \mathcal{H},\forall j\in\{1,2,3\}
\end{eqnarray} 
and
\begin{eqnarray} 
\label{sim2} 
P(\omega^0 \in \mathfrak{H})=P(\omega^0 \in R^{\pi/2}_j \mathfrak{H}),& & \qquad \forall \mathfrak{H} \in \mathcal{H}, \forall j\in\{1,2,3\}.
\end{eqnarray} 
Here $S_j \mathfrak{H}$ and $R^{\pi/2}_j \mathfrak{H}$ denote the sets obtained by applying the appropriate symmetry on each element of $\mathfrak{H}$. 
The prefix ``micro'' refers to the fact that \eqref{sim1}  and \eqref{sim2} are concerned with statistical properties of individual inclusions, rather than with the collective properties of the set $\omega$. In plain English, the above conditions encode the fact that, on average, we have as many inclusions containing the origin with a given spatial orientation as we have with reflected --- in the case of \eqref{sim1} --- or with $\pi/2$-rotated  --- in the case of \eqref{sim2} --- about any of the three axes.

For $\lambda\not\in\sigma (\A_0)$
and $c \in \mathbb{R}^3$, 
let $\mathbf{b}^{c}_{\lambda,\omega_0} \in H_0^1(\omega^0;\mathbb{R}^3)$ be the unique solution to 
\begin{equation*}
\label{ex1lemma}
-\operatorname{div} \C_0 \nabla \mathbf{b}^{c}_{\lambda,\omega^0} - \lambda \,  \mathbf{b}^{c}_{\lambda,\omega^0}=c.
\end{equation*}
Note that by Lemma~\ref{lemma b in physical space} we have 
\begin{equation}
\label{example 1 equation 1}
\overline{\mathbf{b}^{(i)}_\lambda}(\omega)=\mathbf{b}^{e_i}_{\lambda,\omega^0}(0)\,.
\end{equation}
On account of the fact that $\C_0$ satisfies \eqref{isotropic tensor}, it is easy to see that we have
\begin{equation}
\label{simeff1} 
\mathbf{b}^{e_k}_{\lambda,S_j\omega^0}=S_j\, \mathbf{b}^{[S_j e_k]}_{\lambda,\omega^0} \circ S_j
\end{equation}
and
\begin{equation}
\label{simeff2} 
\mathbf{b}^{e_k}_{\lambda,R^{\pi/2}_j\omega^0}=R^{\pi/2}_j \,\mathbf{b}^{[(R^{\pi/2}_j)^T e_k]}_{\lambda,\omega^0} \circ (R^{\pi/2}_j)^T
\end{equation}
for all $j,k\in{1,2,3}$. Here the superscript $T$ stands for the transposed (which in  case of rotations coincides with the inverse).

Formulae \eqref{beta matrix}, \eqref{example 1 equation 1}, \eqref{simeff1} and \eqref{sim1} yield
\begin{equation}
\label{example 1 equation 2}
\bbeta(\lambda)= S_j \,\bbeta(\lambda) \,S_j \qquad \forall j \in\{1,2,3\}\,,
\end{equation}
whereas formulae \eqref{beta matrix}, \eqref{example 1 equation 1}, \eqref{simeff2} and \eqref{sim2} yield
\begin{equation}
\label{example 1 equation 3}
\bbeta(\lambda)= R^{\pi/2}_j \,\bbeta(\lambda)\, (R^{\pi/2}_j)^T \qquad \forall j \in\{1,2,3\}\,.
\end{equation}

\begin{proposition}
\label{proposition 1 example 1}
\begin{enumerate}[(a)]
\item 
Under assumption \eqref{sim1}, the matrix $\bbeta(\lambda)$ is diagonal.

\item
Under assumption \eqref{sim2}, the matrix $\bbeta(\lambda)$ is scalar, i.e., proportional to the identity matrix.
\end{enumerate}
\end{proposition}
\begin{proof}
Parts (a) and (b) follow from a straightforward examination of \eqref{example 1 equation 2} and \eqref{example 1 equation 3}, respectively. Observe that, in fact, for (b) to hold it suffices to have the identity \eqref{example 1 equation 3} for any two of the $R_j^{\pi/2}$'s.
\end{proof}

Let us now specialise our geometry even further, by considering a cuboid $Q=(-\frac{l_1}2,\frac{l_1}2)\times(-\frac{l_2}2,\frac{l_2}2)\times(-\frac{l_3}2,\frac{l_3}2)$, $l_i\ne l_j$ for $i\ne j$. 
Requiring that the edges are of different length serves the purpose of embedding some directional anisotropy into the problem.

If one considers the periodic problem in which the inclusions are obtained by rigid translations of $Q$, then clearly \eqref{sim1} is satisfied and the corresponding $\bbeta$-matrix, which we denote by $\bbeta_\mathrm{per}^Q(\lambda)$, is diagonal by Proposition~\ref{proposition 1 example 1}(a).
Let $\{\nu_{n}\}_{n\in\mathbb{N}}$ be the the eigenvalues of the operator $\A_0^Q$, enumerated in increasing order and without account of multiplicities (so that $(\nu_n,\nu_{n+1})$ is non-empty). By definition, the function $\bbeta_\mathrm{per}^Q$ is well-defined in $\mathbb{R}\setminus \bigcup_n \{\nu_n\}$. It was shown in \cite[Section 4.1]{avila08} that as $\lambda$ ranges in $(\nu_n, \nu_{n+1})$ the smallest eigenvalue of $\bbeta_{\textrm{per}}^Q(\lambda)$ ranges from $-\infty$ to either some finite value or $+\infty$. On the other hand, the largest eigenvalue of  
$\bbeta_{\textrm{per}}^Q(\lambda)$ ranges from either some finite value or $-\infty$ to $+\infty$. 

\

In the remainder of this subsection we will examine two examples satisfying \eqref{sim2}, built out of $Q$.

\

First, let us construct a unit cell containing the three inclusions $R_j^{\pi/2}Q$, $j\in\{1,2,3\}$, suitably scaled and positioned relative to one another so that they satisfy Assumption~\ref{main assumption}. Extending said cell by periodicity to the whole of $\R^3$ produces a single set of inclusions $\omega$. Our probability space is then the collection of all such $\omega$'s together with their translations by random vectors chosen uniformly in the unit cell $[0,1]^3$ (cf.~\cite[Sections~7.2 and~7.3]{CCV1}). We refer the reader to \cite[Section~5.6]{CCV2} for further details on the construction of the probability space. In this case,  the $\beta$-matrix, which we denote by $\bbeta_{\textrm{per}}(\lambda)$, is a multiple of the unit matrix by Proposition~\ref{proposition 1 example 1}(b). More precisely, we have
\begin{equation}
\label{example 1 equation 4}
\bbeta_{\textrm{per}}(\lambda)=\frac{1}{3} \left(\operatorname{tr} \bbeta^Q_{\textrm{per}}(\lambda)\right) I.
\end{equation}
Moreover, by direct inspection, it is easy to see that the $\beta_\infty$-function reads
\begin{equation}
\label{example 1 equation 5}
(\beta_\mathrm{per})_{\infty}(\lambda)=\frac{1}{3} \operatorname{tr} \bbeta^Q_{\textrm{per}}(\lambda)\,.
\end{equation}
In this case, $(\beta_\mathrm{per})_{\infty}(\lambda)$ is the mean of the eigenvalues of $\bbeta_{\textrm{per}}^Q(\lambda)$, including its smallest and biggest eigenvalues, therefore it necessarily ranges from $-\infty$ to $+\infty$ as $\lambda$ varies in $(\nu_n, \nu_{n+1})$, due to the results from \cite[Section 4.1]{avila08} reported above.

\

Next, let us construct another collection of inclusions $\omega$ by placing at the centre of each cell $\square+z$, $z\in\Z^3$, one of the three shapes $R^{\pi/2}_j Q$, $j\in{1,2,3}$, independently and with probability $1/3$. As before, we take our probability space to be the set of all such $\omega$'s together with their translations by random vectors chosen uniformly in the unit cell $[0,1]^3$. By direct inspection, one sees that in this case the $\bbeta$-matrix is still given by the expression on the right hand side of \eqref{example 1 equation 4}, but we have
\begin{equation}
\label{example 1 equation 6}
\beta_{\infty} (\lambda)=\max_{k\in\{1,2,3\}} [\bbeta^Q_{\textrm{per}}(\lambda)]_{kk}\,,
\end{equation}
compare with \eqref{example 1 equation 5}. Therefore, by choosing $Q$ appropriately one can have that in some interval $(\nu_n, \nu_{n+1})$ the quantity $\beta_{\infty} (\lambda)$ is always positive.

\

The above examples demonstrate, in concrete scenarios, that $\bbeta(\lambda)$ is determined by the shape of the inclusions $\omega^0$ sitting at the origin and their probability distribution, whereas $\beta_{\infty}(\lambda)$ is sensitive to the ``global geometry'' of the collection of inclusions. Furthermore, whereas in the periodic setting one needs each individual inclusion to be invariant under symmetries in order to influence the form of $\bbeta$-matrix \cite[Section~4]{pastukova}, the stochastic setting allows one to do so under the weaker assumption of \emph{statistical} symmetries (in the sense of \eqref{sim1}, \eqref{sim2}).


\subsection{Example 2: statistical macro-symmetries}
\label{Example 2: statistical macro-symmetries}

Whilst obtaining an isotropic homogenised tensor $\C^{\hom}$ in the purely periodic setting is not easy to achieve (if at all possible), this is something that can be done very naturally in the stochastic setting by means of statistical macro-symmetries. Here ``macro'' (as opposed to ``micro'' from Subsection~\ref{Example 1: statistical micro-symmetries}) refers to the fact that symmetries apply to the collection of inclusions as a whole. This is the subject of this short subsection.

\

Suppose that $\C_1$ is isotropic and that the inclusions possess the following additional statistical macro-symmetries:
\begin{equation} 
\label{macrosim}  
 P(\mathfrak{f})=P(R  \mathfrak{f}), \quad \forall  \mathfrak{f} \in \mathcal{F},\,\,\forall R \in SO(3).
 \end{equation}  
Here $R  \mathfrak{f}$ denotes the set obtained by applying the transformation $R$ to each element $\omega$ of $ \mathfrak{f}$. Recall that $\mathcal{F}$ is the $\sigma$-algebra introduced in subsection~\ref{Probabilistic and geometric setting}.

It then follows easily from \eqref{definition Chom} and \eqref{macrosim} that in this case we have 
\begin{equation*}
C^\hom (R\xi R^T)\cdot  (R\xi R^T)= \C^\hom \xi\cdot \xi, \quad \forall R \in SO(3),
\end{equation*}
which, in turn, implies that the tensor $\C^\hom$ is isotropic (see, e.g.,~\cite[Chapter~4]{ciarlet}). 

\

\textcolor{black}{
Note that this general example is not concerned with the spectral properties. In fact, one may come up with models satisfying the assumption~\eqref{macrosim}   such that in the limit one has either of the possibilities --- $\sigma(\A^\hom) = \mathcal{G}$ or $\sigma(\A^\hom)$ is a proper subset of $\mathcal{G}$ (whether the latter coincides with $\lim_{\varepsilon\to 0} \sigma(\A^\e)$ or not).}

\subsection{Example 3: random scaling}
\label{Example 3: random scaling}

As our last example, let us consider a distribution of inclusions obtained from a single shape $Q$ (not necessarily the same $Q$ \textcolor{black}{as in} subsection~\ref{Example 1: statistical micro-symmetries}) placed in each cell $\square+z$, $z\in\Z^3$, and randomly scaled by a factor $r\in [r_1,r_2]$, $0<r_1<r_2<1$. For simplicity, we assume that the scaling factor $r$ is chosen independently in each cell.
Let $\{(\nu_{n},\bvarphi_n)\}_{n\in\mathbb{N}}$ be the orthonormalised eigenpairs of the operator $\A_0^Q$, where eigenvalues are enumerated in increasing order and with account of multiplicities.

Let $X: \Omega \to \{-\infty\}\cup [r_1,r_2]$  be the random variable returning, for each configuration of inclusions $\omega$, the scaling factor of the inclusion $\omega^0$ containing the origin, where $X(\omega):=-\infty$ if $0\not\in \omega$. We denote by $\operatorname{supp} X$ the support of the random variable $X$. Then we have
\begin{equation*}
\label{example 3 equation 1} 
\sigma(\A_0)=\bigcup_{n \in \mathbb{N},\, r \in \textrm{supp} X\setminus \{-\infty\}}\{r^{-2}\, \nu_n\}
\end{equation*} 
and
\begin{equation*}
\label{example 3 equation 2} 
\bbeta(\lambda)= \frac{1}{P(X\in [r_1,r_2])}\int_{\textrm{supp} X \setminus \{-\infty\}} \bbeta_r(\lambda)\,dP(X=r)\,,
\end{equation*}
where
\begin{equation}
\label{example 3 equation 3} 
\bbeta_r(\lambda):=\lambda I+\lambda^2 \sum_{n\in \mathbb{N}}   r^{-6}\frac{\left(\int_{Q} \bvarphi_n(y)\,dy\right) \otimes \left(\int_{Q} \bvarphi_n(y)\,dy \right)}{r^{-2}\nu_n-\lambda}\,.
\end{equation}
Observe that the quantity \eqref{example 3 equation 3} is the $\beta$-matrix of the periodic problem with reference inclusion $rQ$.

On the other hand we have that 
\begin{equation*}
\label{example 3 equation 4} 
\beta_{\infty} (\lambda)= \sup_{r \in \textrm{supp}  X\setminus\{-\infty\}} \max_{\hat{k} \in \mathbb{S}^{d-1}} \hat{k}\cdot \bbeta_r(\lambda) \hat{k}\,. 
\end{equation*}

Once again, we see how the $\bbeta$-matrix is obtained by averaging over the shapes of the inclusions, whereas the $\beta_\infty$ function captures ``extremal'' properties (supremum over $r$ of the greatest eigenvalue of $\bbeta_r$).

\begin{remark}
\label{remark after examples}
\begin{enumerate}[(i)]

\item
One could also consider the case where geometry of the inclusions is fixed, but the tensor $\C_0$ is random. Since this requires one to slightly adjust the probability framework but yields analogous results with no additional difficulties, we refrain from discussing this example in detail.

\item 
Alternatively, one can rely on appropriate point processes (e.g., the random parking model) to position the inclusions randomly in space, without reference to periodic lattice. The additional technical aspects arising from doing so can be dealt with similarly to the scalar case \cite[subsection~5.6.6]{CCV2}. \textcolor{black}{We should like to mention that, although the random parking model does not satisfy Assumption~\ref{assumption finite correlation}, one can show, using the special properties of the model, that there exist non-typical areas of inclusions that enable one to carry out the arguments in the proof of Theorem~\ref{theorem G sub limiting spectrum}. Hence, for the random parking model we have that the limiting spectrum coincides with the set $\mathcal{G}$.}
\end{enumerate}
\end{remark}

\section{Acknowledgements}

The research of Matteo Capoferri was partially supported by the Leverhulme Trust Research Project Grant RPG-2019-240 and EPSRC Fellowship EP/X01021X/1.
The research of Mikhail Cherdantsev was partially supported by the Leverhulme Trust Research Project Grant RPG-2019-240.
The research of Igor Velčić was partially supported by the Croatian Science Foundation under the Grant Agreements no.~IP-2018-01-8904 (Homdirestroptcm) \textcolor{black}{and IP-2022-10-5181 (HOMeOS)}.

\color{black}
The authors are grateful to the anonymous referees for their thorough comments and suggestions, which helped improving the paper.
\color{black}

\begin{appendices}

\section{Auxiliary materials}
\label{Appendix auxiliary materials}

\subsection{Two-scale convergence}
\label{appendix: two-scale}

In this appendix we summarise, in the form of a theorem, some properties of stochastic two-scale convergence used in our paper, adjusted to the case of systems. We refer the reader to \cite{ZP} for further details.

\begin{theorem}
 \label{theorem stochastic two scale convergence}  

 Stochastic two-scale convergence enjoys the following properties.
\begin{enumerate}[(i)]

\item
Let $\{\bu^\epsilon\}$ be a bounded sequence in $\Lb^2(\R^d)$. Then there exists $\overline{\bu}\in \Lb^2(\R^d\times \Omega)$ such that, up to extracting a subsequence,
$
\bu^\epsilon \overset{2}{\rightharpoonup} \overline{\bu}.
$

\item
If $\bu^\epsilon \overset{2}{\rightharpoonup} \overline{\bu}$, then $\|\overline{\bu}\|_{\Lb^2(\R^d\times \Omega)}\le \liminf_{\epsilon\to 0}\|\bu^\epsilon\|_{\Lb^2(\R^d)}$.

\item
If $\bu^\epsilon\to \overline{\bu}$ in $\Lb^2(\R^d)$, then $\bu^\epsilon \overset{2}{\rightharpoonup} \overline{\bu}$.

\item 
Let $\{\bv^\epsilon\}$ be a uniformly bounded sequence in $\Lb^\infty(\R^d)$ such that $\bv^\epsilon\to \bv$ in $\Lb^1(\R^d)$, $\|\bv\|_{\Lb^{\infty}(\R^d)}<+\infty$. Suppose that $\{\bu^\epsilon\}$ is a bounded sequence in $\Lb^{2}(\R^d)$ such that
$\bu^\epsilon \overset{2}{\rightharpoonup} \overline{\bu}$ for some $\overline{\bu}\in \Lb^2(\R^d \times \Omega)$. Then $\bv^\epsilon\cdot
\bu^\epsilon \overset{2}{\rightharpoonup} \bv\cdot \overline{\bu}$.

\item
Let $\{\bu^\epsilon\}$ be a bounded sequence in $\Hb^1(\R^d)$. Then, there exist $\bu_0\in \Hb^1(\R^d)$ and $\overline{\mathbf{p}}\in L^2(\R^d; \mathcal V^2_{\rm pot}(\Omega))$ such that, up to extracting a subsequence,
\[
\bu^\epsilon\rightharpoonup \bu_0 \quad \text{in}\quad \Hb^1(\R^d),
\]
\[
\nabla \bu^\epsilon  \overset{2}{\rightharpoonup} \nabla \bu_0 +\overline{\mathbf{p}}.
\]

\item
Let $\{\bu^\epsilon\}$ be a bounded sequence in $\Lb^2(\R^d)$ such that $\{\epsilon \nabla \bu^\epsilon\}$ is bounded in $[\Lb^2(\R^d)]^d$. Then there exists $\overline{\bu}\in L^2(\R^d;\Hb^1(\Omega))$ such that, up to extracting a subsequence,
\[
\bu^\epsilon \overset{2}{\rightharpoonup} \overline{\bu},
\]
\[
\epsilon\nabla \bu^\epsilon \overset{2}{\rightharpoonup} \overline{\nabla_{y} \bu}.
\]

\end{enumerate} 
\end{theorem}

\subsection{The Ergodic Theorem}
\label{The ergodic theorem}

A bridge between random variables and their realisations is provided by the Ergodic Theorem, a classical result on ergodic dynamical systems that appears in the literature in various (not always equivalent) formulations. For the reader's convenience, we report here the version used in our paper, see, e.g., \cite{RS}, or \cite{AK} for a more general take.

\begin{theorem}[Birkhoff's Ergodic Theorem]
\label{ergodic theorem}
Let $(\Omega, \mathcal{F}, P)$ be a complete probability space equipped with an ergodic dynamical system $(T_y)_{y\in \mathbb{R}^d}$. Let $\textcolor{black}{\overline{f}}\in L^p(\Omega)$,  $1\le p< \infty$. Then we have 
\[
f(\,\cdot\,/\epsilon,\omega)=\overline{f}(T_{\cdot/\epsilon}\omega) \rightharpoonup \textcolor{black}{\E[\overline{f}]}
\]
in $L^p_\mathrm{loc}(\R^d)$ as $\epsilon\to 0$ \textcolor{black}{almost surely}.
\end{theorem}

\begin{remark}
Observe that Theorem~\ref{ergodic theorem} implies
\begin{equation}
\label{ergodic theorem concrete}
\lim_{R\to +\infty} \frac{1}{R^d}\int_{\square^R} f(y,\omega)\,dy=\textcolor{black}{\E[\overline{f}]}.
\end{equation}
In the current paper, we often apply the Ergodic Theorem in the more concrete version \eqref{ergodic theorem concrete}, with the ergodic dynamical system being \eqref{T_y}.
\end{remark}

\subsection{Higher regularity for the periodic corrector}
\label{Appendix higher regularity}

In this appendix we state a result on the higher (than $\Hb^1$) regularity of the periodic homogenisation corrector $\hat \Nb^{(j,k)}$, needed in the proof of Theorem~\ref{theorem G sub limiting spectrum}. 

\begin{theorem}
\label{theorem higher regularity corrector} 
Under Assumptions~\ref{main assumption} and~\ref{main assumption 2} and with the notation from formulae~\eqref{proof part2 equation 5}, \eqref{proof part2 equation 6} and surrounding text, there exists $p>2$ and a universal constant $C>0$ such that
\begin{equation*}
\label{theorem higher regularity corrector equation 3}
\|\nabla \hat \Nb^{(j,k)}\|_{[\Lb^p(\tilde\square_{x_1}^{R+\kappa})]^d}\le C (R+\kappa)^{\frac{d}{p}}\,.
\end{equation*}
\end{theorem}

\begin{proof}
The proof retraces that of \cite[Theorem 5.20]{CCV2}, therefore we will only streamline the changes required. 

The approach of \cite[Theorem 5.20]{CCV2} relies on the use of the reverse H\"older inequality \cite[Lemma C.4]{CCV2}. In the case of systems one needs to do the following.
\begin{itemize}
\item 
Use the Extension Theorem \ref{theorem extension} for \emph{some} $p>2$.

\item
Replace the uniform Poincar\'e--Sobolev inequality with the uniform Korn--Sobolev inequality.

\item 
Replace \cite[Lemma C.1]{CCV2} with the following statement. There exists $m\ge 1$, $R_1>0$ and $C>0$ such that
for every $\mathbf{f} \in \W^{1,2} (\mathbb{R}^d \setminus \omega)$ we have
\begin{equation}\label{74}
\|\mathbf{f} - \mathbf{c}_{R}(x) \|_{\Lb^q(B_{R}(x) \setminus  \omega)} \leq C R^{d(\frac{1}{q} - \frac{1}{2})+1} \|\sym \nabla \mathbf{f}\|_{\Lb^2( B_{mR}(x) 	\setminus \omega)},
\end{equation}
for all $x\in \R^d$ and $R<R_1$, where $2\leq q\leq 2d/(d-2)$ and $\textbf{c}_R(x)$ is some infinitesimal rigid motion depending on $\left.\mathbf{f}\right|_{B_{m\,R}(x)\setminus \omega}$. 
\end{itemize}
Since the RHS of \eqref{74} contains a scaling factor $m$, when proving \eqref{74} one is faced with two cases:
\begin{enumerate}[(i)]
\item \emph{The radius $R$ is small with respect to the characteristic size of the inclusions}. Then, if $B_R(x)$ does not intersect $\omega$ one can take $m=1$ so that \eqref{74} is just the usual Korn--Sobolev inequality for balls (appropriately scaled). If, on the other hand, $B_R(x)$ intersects $\omega$, \textcolor{black}{without loss of generality} we can assume that $x\in \partial \omega^k$ for some $k$. In this case, using Assumption~\ref{main assumption 2}, one has that $B_R(x)\setminus \omega$ is $C^2$-diffeomorphic to a half-ball. The uniformity then follows by contradiction arguing along the lines of \cite[Lemma~1 and Remark~6]{velcic}.

\item \emph{The radius $R$ is comparable or bigger than the characteristic size of the inclusions}. Then, by choosing $m$ sufficiently large, one can assume that $ B_{mR}(x)$ fully contains the inclusions $\omega^k$ with non-empty intersection with $B_{R}(x)$ together with their extension domains $\mathcal{B}_\omega^k$. Hence, one obtains \eqref{74} by using the Korn--Sobolev inequality in $B_{R}(x)$ applied to the extension of $\mathbf{f}$ via Theorem~\ref{theorem extension}.
\end{enumerate}
 One can then apply reverse H\"older inequality and prove the theorem exactly as in \cite[Appendix~C]{CCV2}.
\end{proof}

We refer the reader to Remark~\ref{remark on sufficiently regular} for further comments on the (non) optimality of our assumptions for Theorem~\ref{theorem higher regularity corrector}.

\end{appendices}


\begin{thebibliography}{37}
\addcontentsline{toc}{section}{References}

\bibitem{AK}
M.~A.~Ackoglu and U.~Krengel, 
Ergodic theorems for superadditive processes,
{\it J.~Reine Angew. Math.} \textbf{323} (1981) 53--67.

%
%
%

\bibitem{allaire}
G.~Allaire,
Homogenization and two-scale convergence,
{\it SIAM J. Math.  Anal.} {\bf 23} no.~6 (1992) 1482--1518.

\bibitem{APZ21}
B.~Amaziane, A.~Piatnitski and E.~Zhizhina, 
Homogenization of diffusion in high contrast random media and related Markov semigroups,
{\it Discrete Contin.~Dyn.~Syst.~Ser.~B} {\bf 28} no.~8 (2023) 4661--4682.

\bibitem{armstrong}
S.~Armstrong, T.~Kuusi and J.-C.~Mourrat,
{\it Quantitative Stochastic Homogenization and Large-Scale Regularity},
Grund.  math.  Wissenschaften {\bf 352}, Springer International Publishing, 2019.

\bibitem{ArmSmart}
S. Armstrong, C. Smart,  Quantitative stochastic homogenization of elliptic equations
in nondivergence form. {\it Arch. Rat. Mech. Anal.} {\bf 214} no.~3 (2014) 867--911.

\bibitem{avila08} 
A.~Avila, G.~Griso, B.~Miara and E.~Rohan, 
Multiscale modeling of elastic waves: theoretical justification and numerical simulation of band gaps, 
{\it Multiscale Model. Simul.} {\bf 7} no.~1 (2008) 1--21.

%
%

\bibitem{bouchitte}
G.~Bouchitt\'e, C.~Bourel and D.~Felbacq,
Homogenization near resonances and artificial magnetism in three dimensional dielectric metamaterials,
{\it Arch. Rat. Mech. Anal. } {\bf 225} (2017) 1233--1277.

\bibitem{decay}
M.~Capoferri, M.~Cherdantsev and I.~Vel\v{c}i\'c,
Eigenfunctions localised on a defect in high-contrast random media,
\emph{SIAM J. Math. Anal.} {\bf 55} no.~6 (2023) 7449--7489.

%

\bibitem{CCV1}
M.~Cherdantsev, K.~Cherednichenko and I.~Vel\v{c}i\'c,
Stochastic homogenisation of high-contrast media,
{\it Applicable Analysis} \textbf{98} no.~1-2 (2019) 91--117.

%
\bibitem{CCV2}
M.~Cherdantsev, K.~Cherednichenko and I.~Vel\v{c}i\'c,
High-contrast random composites: homogenisation framework and new spectral phenomena.
Preprint arXiv:2110.00395v4 (2021).

%
%

\bibitem{maxwell}
K.~Cherednichenko and S.~Cooper,
Asymptotic behaviour of the spectra of systems of Maxwell equations in periodic composite media with high contrast,
{\it Mathematika} {\bf 64} no.~2 (2016) 583--605.

\color{black}
\bibitem{CCarma}
K.~Cherednichenko and S.~Cooper,
Resolvent estimates for high-contrast elliptic problems with periodic coefficients,
{\it Arch. Rational Mech. Anal.} \textbf{219} (2016) 1061--1086.
\color{black}

\bibitem{kirill2}
K.~Cherednichenko, Y.~Ershova and A.~Kiselev,
Effective Behaviour of Critical-Contrast PDEs: Micro-resonances, Frequency Conversion, and Time Dispersive Properties I,
{\it Comm. Math.  Phys.} {\bf 375} (2020) 1833--1884.

\bibitem{iosip}
K.~Cherednichenko, A.~Kiselev, I.~Vel\v{c}i\'c and J.~\v{Z}ubrini\'c,
Effective behaviour of critical-contrast PDEs: micro-resonances, frequency conversion, and time dispersive properties. II.
Preprint arXiv:2307.01125 (2023).

\bibitem{CSZ}
K.~Cherednichenko, V.~P.~Smyshlyaev and V.~Zhikov,
Non-local homogenised limits for composite media with highly anisotropic periodic fibres,
{\it Proc. R. Soc. Edinb. A} {\bf 136} (2006) 87--114.

\bibitem{ciarlet}
P.G.~Ciarlet, {\it Mathematical Elasticity, Volume I: Three-dimensional Elasticity,} North Holland (1993).

\bibitem{cooper1}
S.~Cooper,
Homogenisation and spectral convergence of a periodic elastic composite with weakly compressible inclusions, 
{\it Appl. Anal.} {\bf 93} no.~7 (2014) 1401--1430.

\bibitem{CKS14}
S.~Cooper, I.V.~Kamotski, V.P.~Smyshlyaev, 
On band gaps in photonic crystal fibers.
Preprint arXiv:1411.0238 (2014).

\bibitem{CKS23}
S.~Cooper, I.V.~Kamotski, V.P.~Smyshlyaev, 
Uniform asymptotics for a family of degenerating variational problems and error estimates in homogenisation theory.
Preprint arXiv:2307.13151 (2023).

%

\bibitem{DG}
M.~Duerinckx and A.~Gloria,
 Stochastic Homogenization of Nonconvex Unbounded Integral Functionals with Convex Growth,
{\it Archive for Rational Mechanics and Analysis} \textbf{221} (2016) 1511--1584.

%
%
%

\bibitem{gloria1}
A.~Gloria and F.~Otto,
An optimal variance estimate in stochastic homogenization of discrete elliptic equations,
{\it Ann. Probab. } {\bf 39} no.~3 (2011) 779--856. 

\bibitem{gloria2}
A.~Gloria and F.~Otto,
An optimal error estimate in stochastic homogenization of discrete elliptic equations,
{\it Ann.  Appl.  Probab.} {\bf 22} no.~1 (2012) 1--28.

\bibitem{gloria3}
A.~Gloria, S.~Neukamm and F.~Otto,
Quantification of ergodicity in stochastic homogenization: optimal bounds via spectral gap on Glauber dynamic,
{\it Invent. Math. } {\bf 199} no.~2 (2015) 455--515.

\bibitem{griso}
G.~Griso, 
Decompositions of displacements of thin structures, {\it J. Math. Pures Appl.} {\bf 89} no.~2 (2008) 199--223.
%

\bibitem{heida}
M.~Heida,
Stochastic homogenization on randomly perforated domains.
Preprint arXiv:2001.10373 (2020).

\bibitem{HMT}
S.~Hofmann, M.~Mitrea, and M.~Taylor, 
Geometric and transformational properties of Lipschitz domains, Semmes--Kenig--Toro domains, and other classes of finite perimeter domains, 
{\it J. Geom. Anal.} {\bf 17} (2007) 593--647. 

\bibitem{ZKO}
V.V.~Jikov, S.M.~Kozlov and O.A.~Oleinik, 
{\it Homogenization of Differential Operators and Integral Functionals},
 Springer-Verlag, Berlin (1994).

\bibitem{KamSm1} I.V. Kamotski, V.P. Smyshlyaev, Localized modes due to defects in high contrast periodic media via two-scale homogenization, {\it  J. Math. Sci.,}  \textbf{232} no.~3 (2018) 349--377.

\bibitem{KS2018} I.V. Kamotski, V.P. Smyshlyaev, Two-scale homogenization for a general class of high contrast PDE systems with periodic coefficients, {\it  Applicable Analysis,}  \textbf{98} no.~1-2 (2018) 64-90.
  
\bibitem{KS} I.~Karatzas and S.E.~Shreve, {\it Brownian motion and stochastic calculus}, Second Edition, Springer-Verlag, New York (1991).

%

\bibitem{kozlov}
S.~M.~Kozlov,
Averaging differential operators with almost periodic, rapidly oscillating coefficients,
{\it Math. USSR Sb. } {\bf 35} no.~4 (1979) 481--498.


%
%

\bibitem{kunnemann}
R.~K\"unnemann,
The diffusion limit for reversible jump processes on $\mathbb{Z}^d$ with ergodic random bond conductivities,
{\it Comm.~Math.~Phys.} {\bf 90} (1983)  27--68.

%

\bibitem{naddaf}
A.~Naddaf and T.~Spencer,
Estimates on the variance of some homogenization problems. 
Preprint (1998).

\bibitem{OSY}
O.A.~Oleinik, A.S.~Shamaev and G.A.~Yosifian,
{\it Mathematical problems in elasticity and homogenization},
Studies in mathematics and its applications {\bf 26}, North Hollands,  1992.

\bibitem{papanicolau}
G.~C.~Papanicolaou and S.~R.~S.~Varadhan,
Boundary value problems with rapidly oscillating random coefficients, 
{\it Colloq.  Math. Soc. J\'anos Bolyai} {\bf 27} (1979) 835--873.

\bibitem{RS}
M.~Reed and B.~Simon,
{\it Methods of modern mathematical physics 1: Functional analysis},
Academic Press (1980).

%

\bibitem{smyshlyaev_materials}
V.P.~Smyshlyaev
Propagation and localization of elastic waves in highly anisotropic periodic composites via two-scale homogenization,
{\it Mech.  Mater.} {\bf 41} no.~4 (2009) 434--447.

\bibitem{stein}
E.M.~Stein,
{\it Singular integrals and differentiability properties of functions}, 
Princeton Mathematical Series {\bf 30},  Princeton University Press (1970).

\bibitem{velcic} 
I.~Vel\v{c}i\'{c}, Nonlinear weakly curved rod by {$\Gamma$}-convergence, {\it J. Elasticity} {\bf 108} no.~2 (2012) 125--150.

\bibitem{vishik}
M.~I.~Vishik and L.~A.~Lyusternik, 
Regular degeneration and boundary layer for linear differential equations with small parameter, 
{\it Uspekhi Mat. Nauk} {\bf 12} no.~5 (1957) 3--122.


\bibitem{yurinskii}
V.~V. Yurinski\u{i},
Averaging of symmetric diffusion in random medium,
{\it Sibirskii Matematicheskii Zhurnal} {\bf 27} no.4 (1986) 167--180.

\bibitem{zhikov2000}
V.V.~Zhikov, 
On an extension of the method of two-scale convergence and its applications, (Russian) 
{\it Mat. Sb.} \textbf{191} no.~7 (2000) 31--72; (English translation) {\it Sb. Math.} {\bf 191} no.~7-8 (2000) 973--1014.

\bibitem{zhikov2004}
V.~V.~Zhikov, 
Gaps in the spectrum of some elliptic operators in divergent form with periodic coefficients, (Russian) {\it Algebra i Analiz} \textbf{16} no.~5 (2004) 34--58; (English translation) {\it St. Petersburg Math. J.} {\bf 16} no.~5 (\textcolor{black}{2005})  773--790.

\bibitem{pastukova}
V.V.~Zhikov and S.E.~Pastukova,
On gaps in the spectrum of the operator of elasticity theory on a high contrast periodic structure,
{\it J. Math. Sci.} {\bf 188} no.~3 (2013) 227--240.

\bibitem{ZP}
V.V.~Zhikov and A.L.~Pyatnitskii,
Homogenization of random singular structures and random measures,
{\it Izv. Math. } {\bf 70} no.~1 (2006) 19--67.

\end{thebibliography}
\end{document}